\documentclass[a4paper,10pt]{amsart}

\usepackage[dvipsnames]{xcolor} 
\usepackage{amsmath,amsthm,amssymb}
\usepackage{enumitem}
\usepackage{tikz}
\usepackage{hyperref}
\usepackage{cleveref}
\usepackage{amsfonts,marvosym,amscd}
\usepackage{mathtools}
\usepackage{mathscinet}
\usepackage{wasysym} 
\usepackage{graphicx}
\usepackage{psfrag}
\usepackage{datetime}
\usepackage{tikz-cd} 
\usepackage{adjustbox}
\usepackage{verbatim}
\usepackage{subcaption}

\usepackage[margin=25mm]{geometry}

\usepackage{tabstackengine}
\setstackgap{L}{.75\baselineskip}

\usetikzlibrary{decorations.pathmorphing}
\usetikzlibrary{positioning}
\usetikzlibrary{calc}

\makeatletter
\tikzset{
    dot diameter/.store in=\dot@diameter,
    dot diameter=3pt,
    dot spacing/.store in=\dot@spacing,
    dot spacing=10pt,
    dots/.style={
        line width=\dot@diameter,
        line cap=round,
        dash pattern=on 0pt off \dot@spacing
    }
}
\makeatother

\newtheorem{theorem}{Theorem}[section]
\newtheorem{lemma}[theorem]{Lemma}
\newtheorem{proposition}[theorem]{Proposition}
\newtheorem{corollary}[theorem]{Corollary}

\newtheorem{conjecture}[theorem]{Conjecture}
\newtheorem{fact}[theorem]{Fact}
\newtheorem{problem}[theorem]{Problem}

\newtheorem{innercustomthm}{Theorem}
\newenvironment{customthm}[1]
  {\renewcommand\theinnercustomthm{#1}\innercustomthm}
  {\endinnercustomthm}

\theoremstyle{definition}
\newtheorem{definition}[theorem]{Definition}

\newtheorem{example}[theorem]{Example}

\theoremstyle{remark}
\newtheorem{remark}[theorem]{Remark}

\DeclareMathOperator{\Hom}{Hom}

\DeclareMathOperator{\Ext}{Ext}
\DeclareMathOperator{\End}{End}

\DeclareMathOperator{\modules}{mod}
\renewcommand{\mod}{\modules}
\DeclareMathOperator{\additive}{add}

\DeclareMathOperator{\Stab}{Stab}

\DeclareMathOperator{\Filt}{Filt}

\DeclareMathOperator{\brick}{brick}
\DeclareMathOperator{\tors}{tors} 
\DeclareMathOperator{\torf}{torf}
\DeclareMathOperator{\dimu}{\underline{dim}}

\renewcommand{\emptyset}{\varnothing}

\newcommand{\mg}[1]{\mathsf{MC}(#1)} 
\newcommand{\mge}[1]{\widetilde{\mathsf{MC}}(#1)} 
\newcommand{\mgel}[1]{\widetilde{\mathsf{MC}}_{\lambda}(#1)} 
\newcommand{\mgea}[2]{\widetilde{\mathsf{MC}}_{#1}(#2)} 
\newcommand{\mgmap}[1]{\mg{#1}}
\newcommand{\mgtmap}[1]{\mge{#1}}
\newcommand{\mgs}[1]{\mathsf{MG}(#1)} 
\newcommand{\mgse}[1]{\widetilde{\mathsf{MG}}(#1)} 

\newcommand{\rtss}[1]{\mathsf{R2SS}(#1)} 
\newcommand{\inv}[1]{\mathrm{inv}(#1)} 
\newcommand{\invc}[1]{\mathrm{inv}_{c}(#1)} 
\newcommand{\inva}[2]{\mathrm{inv}_{#1}(#2)}

\newcommand{\sbf}[2]{\mathcal{B}(#1, #2)} 
\newcommand{\sbfb}[2]{\mathcal{B}\bigl(#1, #2\bigr)} 
\newcommand{\ef}[2]{\mathcal{E}_c(#1, #2)} 
\newcommand{\asef}[2]{\mathcal{W}_c(#1, #2)} 
\newcommand{\asefa}[3]{\mathcal{W}_{#1}(#2, #3)}

\newcommand{\hor}[1]{#1^{\bot}} 
\newcommand{\hol}[1]{\prescript{\bot}{}{#1}} 

\newcommand{\bruhat}[2]{\mathcal{B}(#1, #2)} 
\newcommand{\stash}[2]{\mathcal{S}(#1, #2)} 

\newcommand{\st}{\mid} 

\newcommand{\bricks}[1]{\mathsf{B}(#1)} 

\newcommand{\covrel}[1]{\mathsf{E}(#1)}
\newcommand{\polygon}[2]{#1 \textup{\pentagon} #2}
\newcommand{\coll}[1]{\mathsf{Pos}(#1)}
\newcommand{\preord}{\preccurlyeq} 
\newcommand{\preordstr}{\prec} 
\newcommand{\preordop}{\succcurlyeq}
\newcommand{\precdot}{\prec\mathrel{\mkern-5mu}\mathrel{\cdot}}
\newcommand{\preordot}{\precdot}

\newcommand{\heap}[1]{\mathsf{Heap}(#1)}

\makeatletter
\def\@tocline#1#2#3#4#5#6#7{\relax
  \ifnum #1>\c@tocdepth 
  \else
    \par \addpenalty\@secpenalty\addvspace{#2}%
    \begingroup \hyphenpenalty\@M
    \@ifempty{#4}{%
      \@tempdima\csname r@tocindent\number#1\endcsname\relax
    }{%
      \@tempdima#4\relax
    }%
    \parindent\z@ \leftskip#3\relax \advance\leftskip\@tempdima\relax
    \rightskip\@pnumwidth plus4em \parfillskip-\@pnumwidth
    #5\leavevmode\hskip-\@tempdima
      \ifcase #1
       \or\or \hskip 1em \or \hskip 2em \else \hskip 3em \fi%
      #6\nobreak\relax
    \dotfill\hbox to\@pnumwidth{\@tocpagenum{#7}}\par
    \nobreak
    \endgroup
  \fi}
\makeatother

\usepackage[backend=biber,
        giveninits=true,
	style=alphabetic,
	url=false,
	doi=false,
	isbn=false,
	maxbibnames=99]{biblatex}
\DeclareFieldFormat{title}{#1} 
\renewbibmacro{in:}{} 
\DeclareFieldFormat{postnote}{#1} 

\title[
Contractions of preorders on maximal chains
]
{Preorders on maximal chains: hyperplane arrangements, \\ Cambrian lattices, and maximal green sequences}
\author{Mikhail Gorsky}
\email{mikhail.gorskii@uni-hamburg.de}
\address{Universit\"at Hamburg, Fachbereich Mathematik, Bundesstraße 55, 20146 Hamburg, Germany}
\author{Nicholas J. Williams}
\email{nw480@cam.ac.uk}
\address{Department of Pure Mathematics and Mathematical Statistics, Centre for Mathematical Sciences, University of Cambridge, Wilberforce Road, Cambridge, CB3 0WB, United Kingdom}

\subjclass[2020]{Primary: 06A07; Secondary: 16G20, 17B22, 13F60}

\keywords{Preorders, posets, lattices, edge labellings, quotients, contractions of posets, weak Bruhat order, Cambrian lattices, higher Bruhat orders, preprojective algebras, maximal green sequences.
}
\thanks{}

\begin{document}

\begin{abstract}
We study preorders on (equivalence classes of) maximal chains in the general context of polygonal lattices endowed with suitably nice edge labellings.
We show that, given a quotient of polygonal lattices, such edge labellings descend to the quotient, and that there is an induced order-preserving surjective map on the preordered sets of equivalence classes of maximal chains. Under a natural condition ensuring that the domain is a poset, the map is a contraction of preordered sets.
We apply this to lattices of regions of simplicial hyperplane arrangements, where the preorders are partial orders, in particular to finite Coxeter arrangements.
For the latter, each choice of Coxeter element gives us a different partial order on the set of equivalence classes of maximal chains; these generalise certain reoriented higher Bruhat orders in dimension two.
The maps of posets of maximal chains induced by Cambrian congruences generalise the map of Kapranov and Voevodsky from the higher Bruhat orders to the higher Stasheff--Tamari orders in dimension two.
While the fibres of this map are known not always to be intervals, our results show that they are always connected.
We show that, in the case of Cambrian lattices, the induced maps on maximal chains have nice descriptions in terms of orientations of rank-two root subsystems, in a way which resembles taking the stable objects of a Rudakov stability condition.
We finally consider the algebraic realisation of the weak order and Cambrian lattices via torsion-free classes of preprojective and path algebras, relating the posets of maximal chains to our earlier work on maximal green sequences.
\end{abstract}

\maketitle

\tableofcontents

\section{Introduction}

Many interesting mathematical objects can be construed as maximal chains of a particular poset.
Composition series are maximal chains in the poset of submodules and flags are maximal chains in the poset of vector subspaces.
Standard Young tableaux are maximal chains of the subposets given by principal order ideals of Young's lattice.
Particular examples that we will study in this paper are reduced expressions of the longest element of a Coxeter group and maximal green sequences of a finite-dimensional algebra.

In some cases, sets of maximal chains can themselves possess natural partial orders, perhaps when first subject to an equivalence relation.
The original example here \textit{par excellence} is the higher Bruhat orders of Manin and Schechtman \cite{ms}, originally introduced in the context of hyperplane arrangements and higher braid groups, but which have subsequently gone on to arise in diverse areas such as Soergel bimodules \cite{elias_bruhat}, the theory of social choice in economics \cite{gr_bruhat}, and Steenrod operations in algebraic topology \cite{njw-gla}. 
Elements of the two-dimensional Bruhat order are commutation classes of reduced expressions of the longest element $w_0$ in the symmetric group $\mathfrak{S}_n$, and the three-dimensional Bruhat order is defined on equivalence classes of maximal chains of such commutation classes. Some variations of these orders for not necessarily reduced expressions have been applied to establish the diamond lemma for Hecke-type algebras and categories \cite{elias_diamond}, and, in the guise of Demazure weaves \cite[Remark 4.24]{CGGS1}, to produce seeds of cluster structures on braid varieties \cite{CGGLSS} and exact embedded Lagrangian fillings of Legendrian links \cite{casals_gao, CLSBW}.
A similar example is the higher Stasheff--Tamari orders of Kapranov and Voevodsky \cite{kv-poly}, subsequently studied in \cite{er,njw-equal}.
These arise out of the combinatorics of KP solitons, together with the higher Bruhat orders \cite{dm-h}, and also in the representation theory of finite-dimensional algebras \cite{njw-hst}.
In particular, the three-dimensional higher Stasheff--Tamari orders appear as posets of equivalence classes of maximal green sequences of certain finite-dimensional algebras: the path algebras of linearly oriented quivers of type $A$.
In previous work \cite{gw1} we extended this by studying a natural equivalence relation on maximal green sequences of arbitrary finite-dimensional algebras and partial orders on equivalence classes.

We begin this paper by studying preorders on equivalence classes of maximal chains in a general order-theoretic context.
This is in a similar spirit to the maximal chain descent orders of \cite{Lacina}, where the setting is finite bounded posets with a CL-labelling.
Our setting is finite polygonal lattices $L$ with poset-valued edge labellings $\lambda$ which have certain properties we call being ``forcing consistent'', ``polygonal'', and ``polygon-complete''. Polygonal lattices are, very roughly speaking, lattices whose Hasse graphs are glued from planar polygons, each having precisely two distinct maximal chains; see \Cref{subsec:background_polygonal} for precise definition. Examples of polygonal lattices include lattices of regions in simplicial hyperplane arrangements \cite{reading_regions}, weak orders of finite Coxeter groups, Cambrian lattices, lattices of torsion classes of $\tau$-tilting finite algebras \cite{dirrt} and, broadly, various lattices of biconvex sets arising in algebraic combinatorics \cite{g-mc}. 
Forcing-consistency is one of the defining conditions of CN-labellings \cite{reading2003}.
Polygonality is a variation of the notion of EL-labelling, the latter being key in the study of shellings of (order complexes of) posets \cite{bjorner}. 
Given a forcing-consistent polygonal edge labelling $\lambda$, we define a preorder of equivalence classes of maximal chains $\mgea{\lambda}{L}$. We show that given such an edge labelling $\lambda$ and a lattice congruence $\theta$ on $L$, there is a natural induced  edge labelling $\lambda_{\theta}$ of the quotient lattice $L/\theta$ which is also forcing-consistent and polygonal. We then have an induced order-preserving surjective map of preorders $
\mgtmap{q} \colon \mgea{\lambda}{L} \to \mgea{\lambda_{\theta}}{L/\theta}$, see \Cref{thm:contraction}\ref{op:contraction:op_surj}.
When $\lambda$ is moreover polygon-complete, the following theorem holds.

\begin{customthm}{A}[{\Cref{lem:polygon-complete_poset} and \Cref{thm:contraction}\ref{op:contraction:contraction}}]\label{thm:intro:contraction}
Suppose that $L$ is a finite polygonal lattice with $\lambda \colon \covrel{L} \to E$ a forcing-consistent polygon-complete
polygonal edge labelling. Then the preorder $\mgea{\lambda}{L}$ is a partial order.
Further, given a lattice quotient $q \colon L \to L/\theta$, the induced map \[
\mgtmap{q} \colon \mgea{\lambda}{L} \to \mgea{\lambda_{\theta}}{L/\theta}
\]
is a contraction of preordered sets.
\end{customthm}

Here, a ``contraction'' of preordered sets is a type of quotient, which in particular includes all finite lattice congruences.
In general, the preordered sets $\mgea{\lambda}{L}$ and $\mgea{\lambda_{\theta}}{L/\theta}$ need not be lattices.
Moreover, when the codomain $\mgea{\lambda_{\theta}}{L/\theta}$ is also a poset, the map $\mgtmap{q}$ need not be an order congruence of posets, which is the analogue of a lattice congruence for arbitrary finite posets \cite{cs_cong,reading_order}.
We emphasise that $\mgtmap{q}$ 
being well-defined requires $L$ to be a lattice and $\theta$ to be a lattice congruence, 
as an order congruence of posets may send a maximal chain to a non-maximal one.

Contractions of posets first appeared in \cite{stanley_tpp} in the description of faces of the order polytope, but also later appeared independently in the context of toric varieties in \cite{Wagner}.
The connection between these two appearances was made in \cite{thomas_03}, where all three of contractions, order polytopes, and toric varieties were studied.
Contractions of posets were generalised to contractions of preordered sets in \cite{cebrian2022directed}.
An overview of the different types of congruences and quotients of posets can be found in \cite{njw-survey}.

We show how the setting of \Cref{thm:intro:contraction} arises in posets of regions of hyperplane arrangements.
There, natural forcing-consistent polygonal edge labellings are valued in the heap posets of fixed initial maximal chains, and the induced quotient labellings can be described in terms of shards \cite{reading_lcfha}.
Moreover, in this case, the preordered set $\mgea{\lambda}{L}$ is a poset.
We use this construction for hyperplane arrangements to then specialise to the setting of the weak Bruhat order on Coxeter groups, where the map from \Cref{thm:intro:contraction} coincides with a map from the higher Bruhat orders to the higher Stasheff--Tamari orders first defined by Kapranov and Voevodsky \cite{kv-poly}.
Indeed, the one-dimensional higher Bruhat order is the weak Bruhat order on $\mathfrak{S}_n$ and the two-dimensional higher Stasheff--Tamari order is the Tamari lattice.
There is a quotient map from the former to the latter which appears in \cite{bw_coxeter,bw_shell_2,tonks,lr_hopf,lr_order,loday_dialgs}.
As above, one can look at the map of equivalence classes of maximal chains associated with this quotient, as done by Kapranov and Voevodsky in \cite{kv-poly}.
They conjectured this map to be a surjection in all dimensions, but this remains an open problem.
It is known that this map cannot always be an order congruence by \cite{thomas-bst}, but it follows from \Cref{thm:intro:contraction} that in dimension~two it is a contraction (see \Cref{cor:map_f_contraction}), so in particular it has connected fibres.

The quotient from the weak Bruhat order on $\mathfrak{S}_n$ to the Tamari lattice is the prototypical example for Reading's theory of Cambrian lattices \cite{reading_cambrian}, which extends these quotients to other finite Coxeter groups and shows how different choices of Coxeter element $c$ give rise to different quotients.
In the same way as Kapranov and Voevodsky, one can look at the induced map on equivalence classes of maximal chains given by such a quotient.

In the case of Cambrian lattices, we show that the map from \Cref{thm:intro:contraction} is a contraction of posets, and that it has the following nice interpretation.
The edge labelling of the weak Bruhat order for the construction underlying \Cref{thm:intro:contraction} labels an edge by the positive root which gets added to the inversion set.
These roots are ordered according to the heap poset $\heap{\mathbf{w}_0(c)}$ of the $c$-sorting word $\mathbf{w}_0(c)$ of the longest element $w_0$.
We then obtain a nice description of which edge labels survive in the quotient using the following notion.
Given a positive root $\beta$, we say that $\beta$ is \emph{$c$-stable} in a maximal chain $C$ of a finite simply-laced Coxeter group $W$ if and only if all rank-two root subsystems in which $\beta$ occurs as a non-simple root are ordered the same by $C$ as by $\heap{\mathbf{w}_0(c)}$.

\begin{customthm}{B}[{Corollary~\ref{cor:chain_stability}}]\label{thm:intro:stability}
Let $W$ be a finite simply-laced Coxeter group with $c$ a Coxeter element, $q_c \colon W \to W_c$ the associated $c$-Cambrian quotient, and $C$ a maximal chain in the weak Bruhat order on~$W$.
Then a covering relation of $C$ labelled by a root $\beta$ is not contracted by $q_c$ if and only if $\beta$ is $c$-stable in~$C$.
Hence, the maximal chain $\mgtmap{q_c}(C)$ is labelled by the sequence of $c$-stable edge labels of~$C$.
\end{customthm}

This resembles
Rudakov stability conditions on an abelian category \cite{rudakov}.
An object $B$ is stable under such a stability condition if all short exact sequences $0 \to A \to B \to C \to 0$ are ordered $A < B < C$ by the stability condition.
This is analogous to our notion of $c$-stability, with short exact sequences replaced by (non-commutative) rank-two subsystems and the heap poset giving the direction of homomorphisms in the Auslander--Reiten quiver \cite{bedard}, \cite[Theorem~9.3.1]{stump2015cataland}.
The map $\mgtmap{q_c}$ is then akin to sending a Rudakov stability condition to its sequence of stable objects.

Cambrian quotients are categorified in \cite{dirrt} for simply-laced Coxeter groups, building on \cite{mizuno-preproj,it}.
Indeed, the weak Bruhat order $W$ is isomorphic to the lattice of torsion-free classes $\torf \Pi$ of the associated preprojective algebra $\Pi$, while the Cambrian lattice $W_c$ is isomorphic to the lattice of torsion-free classes $\torf \Lambda_{c}$ of the path algebra $\Lambda_{c}$ defined by a Coxeter element $c$; the quotient map of lattices is induced by the canonical quotient map of algebras.

Again, one can look at the induced map on equivalence classes of maximal chains given by this quotient.
Since maximal green sequences are maximal chains in the lattice of torsion-free classes, one obtains a map on equivalence classes of maximal green sequences.
One would like this to be an order-preserving map of the posets of equivalence classes of maximal green sequences from \cite{gw1}, but for the preprojective algebra $\Pi$, these posets $\mgse{\Pi}$ have no relations.
However, as we explain in this paper, there is actually good reason for this: there is no single order on $\mgse{\Pi}$, but actually a different natural order --- given by our general construction of orders on maximal chains from poset-valued edge labellings --- for each choice of Coxeter element, or equivalently each choice of orientation of the Coxeter diagram.
Hence, as an application of \Cref{thm:intro:contraction}, we obtain the following theorem, which was another large part of our original motivation.
In the following theorem, we let $(W, S)$ be a finite Coxeter system with a crystallographic Cartan matrix~$A$ and $c \in W$ a Coxeter element.
As explained in Section~\ref{sect:preproj_mgs_orders:alg_back}, one can construct a preprojective algebra $\Pi$ and hereditary tensor algebra $\Lambda_c$ compatible with this data.

\begin{customthm}{C}[{Corollary~\ref{cor:mgs_contraction}}]\label{thm:intro:mgs_preproj}
    There exists a partial order on $\mgse{\Pi}$ such that the map \[\mgtmap{q_c} \colon \mgse{\Pi} \to \mgse{\Lambda_c}\]
    induced by $q_c$ is a contraction of posets, where $q_c \colon \tors\Pi \to \tors\Lambda$ is the Cambrian quotient.
\end{customthm}

In \cite{gw1}, we proved that the poset $\mgse{\Lambda_c}$ has a natural maximum and a unique minimum.
We do not know if this is the case for the partial orders on the set $\mgse{\Pi}$ since the fibres of the contraction are not necessarily intervals (see Remark~\ref{rem:elias} for further discussion).
Still, we expect such partial orders to have nicer properties compared to other possible partial orders on the set $\mgse{\Pi} \cong \mge{W}$ of equivalence classes of maximal chains in the weak Bruhat order.
As an application of \Cref{thm:intro:mgs_preproj}, we are able to give a new proof of the categorical interpretation of Cambrian congruences from \cite[Theorem~7.2]{dirrt} which also applies in non-simply-laced cases (\Cref{prop:not_simply_laced}).

\subsection*{Organisation of the paper} 
In Section~\ref{sect:lattice_quot_and_max_chains}, we study properties of maps induced by lattice quotients between sets of (equivalence classes of) maximal chains in finite lattices.
In Section~\ref{sect:max_chain_posets}, we introduce preorders induced by edge labellings on equivalence classes of maximal chains in a polygonal lattice and prove \Cref{thm:intro:contraction}.
In Section~\ref{sect:regions}, we apply this general framework to lattices of regions of simplicial hyperplane arrangements.
In Section~\ref{sect:scm}, we specialise to finite Coxeter arrangements and quotients of the corresponding Coxeter groups, and in Section~\ref{sect:cambrian} we treat Cambrian congruences in more detail and prove \Cref{thm:intro:stability}.
 We finish the paper with the algebraic interpretations of our results.
In Section~\ref{sect:preproj_mgs_orders} we interpret our two-dimensional Bruhat and Cambrian posets via maximal green sequences of 
preprojective algebras and hereditary tensor algebras.
Our general results on contractions of posets in this setting specialise to \Cref{thm:intro:mgs_preproj}.
In Section~\ref{sect:hbo_hst} 
 we discuss the relation between our results and the higher Bruhat and Stasheff--Tamari orders.

\subsection*{Acknowledgements}
We thank Ben Elias for explanations of the relevance of higher Bruhat orders in the context of Soergel calculus and Andrew Hubery for a helpful discussion about algebraic statements in non-simply-laced types.

Some preliminary results were first presented by MG at the TRAC online seminar in April 2022, and he thanks the organisers for the opportunity. This work is part of a project that has received funding from the European Research Council (ERC) under the European Union’s Horizon 2020 research and innovation programme (grant agreement No.\ 101001159). 
Parts of this work were done during stays of MG at the University of Stuttgart, and he is very grateful to Steffen Koenig for the hospitality. MG acknowledges support by the 
Deutsche Forschungsgemeinschaft (DFG, German Research Foundation) – SFB 1624 – ``Higher structures, moduli spaces and integrability'' – 506632645.
NJW is currently supported by EPSRC grant EP/W001780/1 and was previously supported by EPSRC grant EP/V050524/1.

\section{Lattice quotients and maximal chains}
\label{sect:lattice_quot_and_max_chains}

We begin by studying quotients of lattices and induced maps on their maximal chains.

\subsection{Background}

We start with general background on preorders, posets, lattices, and their quotients.

\subsubsection{Preorders, posets, and lattices}

A \emph{preorder} $\preord$ on a set $P$ is a reflexive, transitive binary relation.
The pair $(P, \preord)$ is then called a \emph{preordered set}.
A~map $f \colon P \to P'$ of preordered sets is \emph{order-preserving} if $f(p) \preord f(p')$ whenever $p \preord p'$.
A \emph{partial order} $\leqslant$ on a set $P$ is an anti-symmetric preorder; in this case, the pair $(P, \leqslant)$ is called a \emph{partially ordered set} or \emph{poset}.
A symmetric preorder is an equivalence relation.
Indeed, every preorder $\preord$ induces an equivalence relation $\phi_{\preord}$ defined by $a \mathrel{\phi_{\preord}} b$ if and only if $a \preord b$ and $a \preordop b$.
This gives a canonical way of defining a poset from a preordered set.
Indeed, given a preordered set $\preord$ on a set $P$, its \emph{collapse} is the poset $(\coll{P}, \leqslant)$ where $\coll{P}$ is the set of $\phi_{\preord}$-equivalence classes of $P$ with the partial order $\leqslant$ defined by $[x] \leqslant [y]$ if and only if there exist $x' \in [x]$ and $y' \in [y]$ such that $x' \preord y'$.
The collapse construction is functorial: given an order-preserving map $f \colon P \to Q$ between preordered sets, then one can define $\coll{f} \colon \coll{P} \to \coll{Q}$ by $\coll{f}([x]) := [f(x)]$.
In fact, it is the left adjoint to the forgetful functor from posets to preordered sets.

Given a preorder $\preord$ on a set $P$ and $x, y \in P$, we write $x \preordstr y$ if $x \preord y$ and $y \not\preord x$.
If $x \preordstr y$, then we say that $y$ \emph{covers} $x$ if $x \preord z \preord y$ implies that $z = x$ or $z = y$.
In this case, we write $x \preordot y$ and call this a \emph{covering relation}; we similarly write $\lessdot$ for covering relations of partial orders.
We write $\covrel{P}$ for the set of covering relations $(x, y)$ of a preordered set~$P$.

A preordered set $(P, \preord)$ is \emph{connected} if for each pair  $x, y \in P$, there exists a finite sequence $x = a_0, a_1, \ldots, a_n = y$ such that for all $1 \leqslant i \leqslant n$, either $a_{i-1} \preord a_i$ or $a_i \preord a_{i-1}$.
The poset $P$ is a \emph{total order} if either $x \leqslant y$ or $y \leqslant x$ for all $x, y \in P$.
The \emph{Hasse diagram} of a poset $P$ is the directed graph with vertices elements of $P$ and arrows $x \to y$ for covering relations $x \lessdot y$.
We hence sometimes also refer to covering relations as \emph{edges}.

Given a poset $P$, a \emph{minimal element} of $P$ is an element $m$ such that $p \not< m$ for all $p \in P$.
If $P$ has a unique minimal element $m$, then we denote $m$ by $\min P$.
\emph{Maximal elements} of $P$ are defined dually, and if there is a unique maximal element of $P$, then it is denoted $\max P$.
If the poset $P$ has a unique minimal element, then the elements covering the minimal element are called \emph{atoms}.
Dually, if $P$ has a unique maximal element, then the elements covered by the maximal element are called \emph{coatoms}.
A \emph{chain} of $P$ is a subset of $P$ which is totally ordered by $\leqslant$ and a \emph{maximal chain} is a chain which is not contained in a strictly larger chain.
We write $\mg{P}$ for the set of maximal chains of $P$.

Given a poset $P$ and a pair of elements $x, y \in P$, an \emph{upper bound} $u$ for $\{x, y\}$ is an element $u \in P$ such that $x \leqslant u$ and $y \leqslant u$.
A \emph{supremum} for $\{x, y\}$ is an upper bound $u$ such that for any other upper bound $u'$ of $\{x, y\}$, we have that $u \leqslant u'$.
The notions of \emph{lower bound} and \emph{infimum} are defined dually.
It is clear that suprema and infima are unique if they exist.

Recall that a \emph{lattice} is a poset $L$ in which every pair of elements $\{x, y\}$ has both a supremum, denoted $x \vee y$, and an infimum, denoted $x \wedge y$.
Here $x \vee y$ is called the \emph{join} of $x$ and $y$ and $x \wedge y$ is called the \emph{meet} of $x$ and $y$.
Every finite lattice is \emph{bounded}, that is, it has a unique maximal element and a unique minimal element.

\subsubsection{Quotients of preordered sets, posets and lattices}\label{sect:back:poset_quotients}

We consider quotients of posets to begin with.
Given a poset $(P, \leqslant)$ and an equivalence relation $\theta$, the quotient $(P/{\theta}, R)$ is defined to be the set of $\theta$-equivalence classes $[x]$ of $P$ with the relation $R$ such that, given $[x], [y] \in P/{\theta}$, we have that $[x]\mathrel{R}[y]$ if and only if there exist $x' \in [x]$ and $y' \in [y]$ such that $x' \leqslant y'$. However, $R$ is in general only a reflexive relation on $P/{\theta}$. It is not generally transitive  or anti-symmetric. If $(P/{\theta}, R)$ is a poset, then there is a canonical order-preserving map $P \to P/{\theta}$, $x \to [x]$.
Several conditions on the equivalence relation $\theta$ exist in the literature which guarantee that the quotient $(P/\theta, R)$ is a well-defined poset.
See \cite{njw-survey} for a survey.

For lattices, we have the following class of well-behaved equivalence relations.
An equivalence relation $\theta$ on a lattice $L$ is a \emph{lattice congruence} if for any $x, y, z \in L$ with $x \mathrel{\theta} y$, we also have $(x \vee z) \mathrel{\theta} (y \vee z)$ and $(x \wedge z) \mathrel{\theta} (y \wedge z)$.
We have that $L/\theta$ is a lattice, with the canonical map $q\colon L \to L/{\theta}$ a lattice homomorphism, i.e., $q(x \wedge y) = q(x) \wedge q(y)$ and $q(x \vee y) = q(x) \vee q(y)$ for all $x, y \in L$.
We refer to $q$ as a \emph{lattice quotient}.
\textit{A fortiori}, $L/\theta$ is a well-defined poset.

For general finite posets, the following notion is an analogue of lattice congruences.
An equivalence relation $\theta$ on a finite poset $P$ is an \emph{order congruence} \cite{reading_order} if the following hold.
\begin{enumerate}
\item 
Every $\theta$-equivalence class is a (closed) interval, i.e., it has the form $[x, z]\coloneqq \{y \in P \st x \leqslant y \leqslant z\}$ for some $x \leqslant z \in P$.
\label{op:read:int}
\item The projection $\pi_{\downarrow} \colon P \to P$, mapping each element $x$ of $P$ to the minimal element in $[x]$, is order-preserving.\label{op:read:down}
\item The projection $\pi^{\uparrow} \colon P \to P$, mapping each element $x$ of $P$ to the maximal element in $[x]$, is order-preserving.\label{op:read:up}
\end{enumerate}
We refer to the canonical map $P \to P/\theta$ as an \emph{order quotient}.
If the poset $P$ is a lattice, then order congruences of $P$ are precisely lattice congruences. 

The notion of order congruence is quite strong, and many interesting examples of quotients of posets are not order quotients.
We will consider the following weaker notion.
An equivalence relation~$\theta$ on a poset $P$ is a \emph{contraction congruence} if every $\theta$-equivalence class is connected and the transitive closure $\overrightarrow{R}$ of the quotient relation $R$ on $P/{\theta}$ is a partial order.
We say that $(P/\theta, \overrightarrow{R})$ is a \emph{contraction} of~$P$.
It is known \cite[Proposition~6.24]{njw-survey} that a poset $Q$ is a contraction of a finite poset $P$ if and only if there is an order-preserving map $f \colon P \to Q$ such that
\begin{enumerate}
    \item the fibres of $Q$ are connected, and
    \item for every covering relation $y_1 \lessdot y_2$ in $Q$, there exists a covering relation $x_1 \lessdot x_2$ in $P$ such that $f(x_1) = y_1$ and $f(x_2) = y_2$.
\end{enumerate}
Note that a fibre of an order-preserving map $P \to Q$ with $P$ finite is connected if and only if its Hasse diagram is connected: a pair of elements in a fibre are related by a finite sequence of relations if and only they are related by a finite sequence of covering relations.
Note also that for a finite poset every order congruence is a contraction.

In \cite[Definition~5.1]{cebrian2022directed}, the notion of contraction is extended from posets to preordered sets.
Indeed, a map $f \colon P \to Q$ between preordered sets is a \emph{contraction} if
\begin{enumerate}
    \item it is surjective and order-preserving;
    \item for $y \in Q$, the set $f^{-1}([y])$ is connected sub-preordered-set of $P$, where $[y]$ is the $\theta_{\preord}$-equivalence class of $[y]$;
    \item for any covering relation $y \preordot y'$ in $Q$, there exists a covering relation $x \preordot x'$ of $P$ such that $f(x) = y$ and $f(x') = y'$.
\end{enumerate}
By \cite[Lemma~5.3]{cebrian2022directed},
if $f \colon P \to Q$ is a contraction of preordered sets, then $\coll{f} \colon \coll{P} \to \coll{Q}$ is a contraction of posets.

\subsubsection{Polygonal lattices}
\label{subsec:background_polygonal}

In this paper, we will be particularly interested in the following type of lattice, which was introduced in \cite{reading_regions}.
A \emph{polygon} in a poset is an interval $[x, y]$ which has precisely two distinct maximal chains $C_1$ and $C_2$ such that $[x, y] = C_1 \cup C_2$ and $C_1 \cap C_2 = \{x, y\}$.
In this case, we write $[x, y] = \polygon{C_1}{C_2}$.
A lattice $L$ is called \emph{polygonal} if the following two dual conditions hold.
\begin{enumerate}
    \item If distinct elements $y_1$ and $y_2$ both cover an element $x$, then $[x, y_1 \vee y_2]$ is a polygon.
    \item If an element $y$ covers distinct elements $x_1$ and $x_2$, then $[x_1 \wedge x_2, y]$ is a polygon.
\end{enumerate}

Polygonal lattices are important for us, because of the following operation relating a pair of maximal chains.
In Section~\ref{sect:max_chain_posets}, we will use these to define preorders on equivalence classes of maximal chains.

\begin{definition}
Given a polygonal lattice $L$ and two maximal chains $C$ and $C'$ of $L$, we say that $C$ and $C'$ are related by a \emph{polygon move} if there exists a polygon $P = \polygon{P_l}{P_r}$ of $L$ with $C \supseteq P_l$, $C' \supseteq P_r$, and $C \setminus P_l = C' \setminus P_r$.
\end{definition}

We have the following useful facts due to Reading.

\begin{fact}
[{\cite[Lemma~9-6.3, Proposition~9-6.9]{reading_regions}}]\label{prop:polygonal_lattices}
Polygonal lattices have the following properties.
\begin{enumerate}[label=\textup{(}\arabic*\textup{)}]
    \item All maximal chains of a finite polygonal lattice are related by a sequence of polygon moves.\label{op:polygonal_lattices:polygon_moves}
    \item A lattice quotient of a finite polygonal lattice is a polygonal lattice.\label{op:polygonal_lattices:quotient}
\end{enumerate}
\end{fact}

The case where the polygon move happens across a square is particularly important for us, since it will give us an equivalence relation on maximal chains.

\begin{definition}
Given a polygonal lattice $L$, a \emph{square} $S$ of $S$ is a polygon with four elements.
By the definition of a polygon, we therefore have that $S =\{x, y, y', z\}$ with $x \lessdot y \lessdot z$ and $x \lessdot y' \lessdot z$.
Two distinct maximal chains $C$ and $C'$ are related by a \emph{square move} if they are related by a polygon move, with the polygon being a square.
We say that two maximal chains $C$ and $C'$ are \emph{square-equivalent} if they are related by a sequence of square moves.
We write $\mge{L}$ for the set of square-equivalence classes of maximal chains.
\end{definition}

This notion of square equivalence is used in the iterative construction of the higher Bruhat \cite{ms} or Stasheff--Tamari orders \cite{kv-poly,er,njw-equal} using maximal chains.
It occurs implicitly in other places too, such as considering reduced words for the longest element of a Coxeter group up to commutation \cite[Section~2.1]{stembridge}, as we will see later.
It is one of the possible ways of defining the equivalence relation on maximal green sequences in \cite{gw1}, see Section~\ref{sect:preproj_mgs_orders}.

\subsection{Maximal chains of quotient lattices}

We want to understand the relation between the maximal chains of two lattices when one is a quotient of the other.
The main result we show in this section is that in this situation there is a well-defined surjection from the set of equivalence classes of maximal chains of the original lattice to the set of those of the quotient lattice.
The analogous result does \textit{not} hold for order quotients of posets.
We first note the following straightforward fact.

\begin{lemma}\label{lem:lat_reg}
Let $q\colon L \to M$ be an order quotient of two finite posets $L$ and $M$.
Then, if $q(x) < q(y)$ for $x, y \in L$ and $\hat{y}$ is the maximal element of $q^{-1}(q(y))$, then $x < \hat{y}$.
Dually, $\check{x} < y$ for $\check{x}$ the minimal element of $q^{-1}(q(x))$.
\end{lemma}
\begin{proof}
We only prove the first statement.
Suppose that $q(x) < q(y)$ for $x, y \in L$.
Since $M$ is a quotient poset of $L$, there must exist $x' \in q^{-1}(q(x))$ and $y' \in q^{-1}(q(y))$ such that $x' < y'$.
Recall the map $\pi_{\uparrow}$ from Section~\ref{sect:back:poset_quotients}, which sends each element to the maximal element in its fibre.
Since the map $\pi_{\uparrow}$ is order-preserving, we must have that $\hat{x} = \pi_{\uparrow}(x') < \pi_{\uparrow}(y') = \hat{y}$. Thus, we have $x \leqslant \hat{x} < \hat{y}$.
\end{proof}

This allows us to show that a lattice quotient of finite lattices cannot send a covering relation to a non-trivial non-covering relation.

\begin{lemma}\label{lem:cov_rel}
Let $q \colon L \to M$ be a lattice quotient of two finite lattices $L$ and $M$.
Then, whenever $x \lessdot z$ in $L$, we either have $q(x) \lessdot q(z)$ in $M$ or $q(x) = q(z)$.
\end{lemma}
\begin{proof}
We show the contrapositive.
Suppose that there exist $x, z \in L$ such that $x < z$, and $q(x) < q(y) < q(z)$ for some $y \in L$.
We wish to show that $x < z$ is not a covering relation.
Let $\hat{y}$ be the maximal element of $q^{-1}(q(y))$ and consider $\hat{y} \wedge z$.
We cannot have $\hat{y} \geqslant z$, since $q(\hat{y}) = q(y) < q(z)$, so we must have $\hat{y} \wedge z < z$.
We must have $x < \hat{y}$ by \Cref{lem:lat_reg}, which implies that $x \leqslant \hat{y} \wedge z$.
We cannot have $x = \hat{y} \wedge z$, since this would imply $q(x) = q(\hat{y} \wedge z) = q(\hat{y}) \wedge q(z) = q(y) \wedge q(z) = q(y)$.
Hence, we have $x < \hat{y} \wedge z$.
However, we then have $x < \hat{y} \wedge z < z$, which gives us that $x < z$ is not a covering relation.
We conclude that if $x \lessdot z$, then either $q(x) \lessdot q(z)$ or $q(x) = q(z)$.
\end{proof}

From this, we deduce that a finite lattice quotient induces a surjection on the sets of maximal chains.

\begin{lemma}\label{lem:max_chains}
Let $q\colon L \to M$ be a lattice quotient of two finite lattices $L$ and $M$. 
Then there is a well-defined map
\begin{align*}
    \mgmap{q} \colon \mg{L} &\to \mg{M} \\
    \{x_0, x_1, \dots, x_r\} &\mapsto \{q(x_0), q(x_1), \dots, q(x_r)\},
\end{align*}
which is moreover a surjection.
\end{lemma}
\begin{proof}
Since $q$ is a surjective order-preserving map, we have that it must map the minimum and maximum of $L$ to the respective minimum and maximum of $M$.
From this and \Cref{lem:cov_rel}, it follows that if $\{x_0, x_1, \dots, x_r\}$ is a maximal chain of $L$, then $\{q(x_0), q(x_1), \dots, q(x_r)\}$ is a maximal chain of $M$. 
This establishes that $\mgmap{q}$ is well-defined.

We now show that $\mgmap{q}$ must be a surjection.
Let $C = \{y_0 \lessdot y_1 \lessdot \dots \lessdot y_s\}$ be a maximal chain in $M$.
Let $x_{00}$ be the minimal element of $q^{-1}(y_0)$, which is also therefore the minimal element of~$L$.
If $x_{10}$ is the minimal element of $q^{-1}(y_1)$, then we have that $x_{00} < x_{10}$ using \Cref{lem:lat_reg}.
Hence, we have a chain $x_{00} \lessdot x_{01} \lessdot \dots \lessdot x_{0r_{0}} \lessdot x_{10}$, since $L$ is finite.
Continuing in this way, we obtain a chain \[
x_{00} \lessdot x_{01} \lessdot \dots \lessdot x_{0r_{0}} \lessdot x_{10} \lessdot x_{11} \lessdot \dots \lessdot x_{1r_1} \lessdot x_{20} \lessdot \dots \lessdot x_{s0},
\]
where $q(x_{ij}) = y_i$ and $x_{i0}$ is the minimal element of the fibre of $y_i$.
We may then choose a chain $x_{s0} \lessdot x_{s1} \lessdot \dots \lessdot x_{sr_s}$ from the minimal element of the fibre of $y_s$ to the maximal element.
Since $y_s$ is the maximal element of $M$, we must have that $x_{sr_s}$ is the maximal element of $L$.
Hence \[
x_{00} \lessdot x_{01} \lessdot \dots \lessdot x_{0r_{0}} \lessdot x_{10} \lessdot x_{11} \lessdot \dots \lessdot x_{1r_1} \lessdot x_{20} \lessdot \dots \lessdot x_{s0} \lessdot x_{s1} \lessdot \dots \lessdot x_{sr_s}
\]
is a maximal chain of $L$, and a pre-image of our original maximal chain $C$ of $M$.
We conclude that $\mgmap{q}$ is a surjection.
\end{proof}

The previous two lemmas do not hold for general order congruences of posets.

\begin{example}
An example of an order congruence $\theta$ on a poset $P$ which does not induce a map on maximal chains is as follows. 
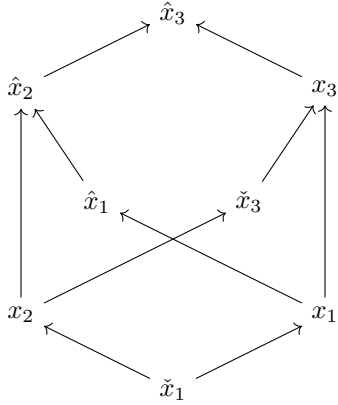
\[
	\begin{tikzpicture}
    
    \node(m1) at (0,0) {$\check{x}_{1}$};
    \node(m2) at (-2,1) {$x_{2}$};
    \node(x1) at (2,1) {$x_{1}$};
    \node(M1) at (-1,2.5) {$\hat{x}_{1}$};
    \node(m3) at (1,2.5) {$\check{x}_{3}$};
    \node(x3) at (2,4) {$x_{3}$};
    \node(M3) at (0,5) {$\hat{x}_{3}$};
    \node(M2) at (-2,4) {$\hat{x}_{2}$};
    
    \draw[->] (m1) -- (m2);
    \draw[->] (m1) -- (x1);
    \draw[->] (x1) -- (M1);
    \draw[->] (m2) -- (M2);
    \draw[->] (x1) -- (x3);
    \draw[->] (m3) -- (x3);
    \draw[->] (M2) -- (M3);
    \draw[->] (x3) -- (M3);
    \draw[->] (M1) -- (M2);
    \draw[->] (m2) -- (m3);
    
    \end{tikzpicture}
\]
The equivalence classes of $\theta$ are $[\check{x}_{1}] = \{\check{x}_{1}, x_{1}, \hat{x}_{1}\}$, $[\check{x}_{2}] = \{\check{x}_{2}, \hat{x}_{2}\}$, and $[\check{x}_{3}] = \{\check{x}_{3}, x_{3}, \hat{x}_{3}\}$.
It is straightforward to check that $\theta$ is indeed an order congruence.
The resulting quotient poset $P/\theta$ is given by $[\check{x}_{1}] < [\check{x}_{2}] < [\check{x}_{3}]$.
Consider the maximal chain $\check{x}_{1} \lessdot x_{1} \lessdot x_{3} \lessdot \hat{x}_{3}$ in the original poset $P$.
This maps to $[\check{x}_1] < [\check{x}_{3}]$, which is not a maximal chain in the quotient poset.
The reason that maximal chains in $P$ do not always map to maximal chains in $P/\theta$ is that \Cref{lem:cov_rel} does not hold: the covering relation $x_1 \lessdot x_3$ of $P$ gets mapped to $[\check{x}_1] < [\check{x}_3]$ in $P/\theta$, which is not a covering relation.
\Cref{lem:cov_rel} does not hold because $P$ is not a lattice, since $x_{3}$ and $\hat{x}_{2}$ do not possess a meet.
\end{example}

The surjection from \Cref{lem:max_chains} descends to a surjection between square-equivalence classes.

\begin{proposition}\label{prop:max_chains_eq}
Let $q\colon L \to M$ be a lattice quotient of two finite lattices $L$ and $M$.
Then there is a well-defined surjection
\begin{align*}
    \mgtmap{q} \colon \mge{L} &\to \mge{M} \\
    [C] &\mapsto [\mgmap{q}(C)].
\end{align*}
\end{proposition}
\begin{proof}
Note first that there is a well-defined map $[\mg{q}] \colon \mg{L} \to \mge{M}$ given by $C \mapsto [\mgmap{q}(C)]$.
In order for $\mgtmap{q}$ to be well-defined, it suffices to show that if $C$ and $C'$ are square-equivalent, then $[\mg{q}](C) = [\mg{q}](C')$.
In turn, it suffices to show that if $C$ and $C'$ are related by a square move, then $[\mg{q}](C) = [\mg{q}](C')$.
It follows from \Cref{lem:cov_rel} that if $C$ and $C'$ are related by a square move, then either $\mgmap{q}(C) = \mgmap{q}(C')$ or $\mgmap{q}(C)$ and $\mgmap{q}(C')$ are related by a square move.
Hence, in either case we have $[\mg{q}](C) = [\mg{q}](C')$, and so $\mgtmap{q}$ is well-defined.
Surjectivity follows from \Cref{lem:max_chains}.
\end{proof}

This result will be key in Section~\ref{sect:max_chain_posets}, where we will upgrade this surjection to a contraction of preorders.
Being a contraction requires connected fibres, and so we prove the following three lemmas in preparation.

\begin{lemma}\label{lem:polygon_moves_in_fibres}
Let $q \colon L \to M$ be a lattice quotient of finite lattices, with $L$ polygonal.
Given two maximal chains $C_{M}$ and $C'_{M}$ of $M$ which differ by a polygon move, there are are maximal chains $C_{L}$ and $C'_{L}$ of $L$ which are related by a polygon move and such that $\mgmap{q}(C_{L}) = C_{M}$ and $\mgmap{q}(C'_{L}) = C'_{M}$. 
\end{lemma}
\begin{proof}
Let $P$ be the polygon relating $C_{M}$ and $C'_{M}$.
Then, by \cite[Proposition~9-5.1]{reading_regions}, we have that $q^{-1}(P)$ is an interval, so $q^{-1}(P) = [x, y]$ for some $x, y \in L$.
Let $C_P = C_{M} \cap P$ and $C'_P = C'_{M} \cap P$.
One can then choose maximal chains $C$ and $C'$ of $[x, y]$ such that $q(C) = C_{P}$ and $q(C') = C'_{P}$ by applying \Cref{lem:max_chains} to $q|_{[x, y]}$.
Moreover, since all maximal chains of $[x, y]$ are connected by polygon moves by \Cref{prop:polygonal_lattices}\ref{op:polygonal_lattices:polygon_moves}, we can choose $C$ and $C'$ to be related by a single polygon move, since maximal chains of $[x, y]$ must either map to $C_P$ or $C'_P$.
One can use the technique of \Cref{lem:max_chains} to extend $C$ and $C'$ to respective maximal chains $C_{L}$ and $C'_{L}$ which agree outside of $[x, y]$ and which are such that $\mgmap{q}(C_{L}) = C_{M}$ and $\mgmap{q}(C'_{L}) = C'_{M}$, since $C_{M}$ and $C'_{M}$ agree outside of $P$.
\end{proof}

One could use the previous lemma to prove the following, but we give a more direct proof.

\begin{lemma}\label{lem:polygon_preimages}
Let $L$ be a finite polygonal lattice with $\lambda \colon \covrel{L} \to E$ a forcing-consistent edge labelling and $\theta$ a lattice congruence on $L$, with quotient map $q \colon L \to L/\theta$.
Then, for every polygon $P \subseteq L/\theta$, there is a polygon $P' \subseteq L$ such that $q(P') = P$.
\end{lemma}
\begin{proof}
Let $P = \polygon{C_1}{C_2}$ be a polygon of $L/\theta$, with $\check{p} = \min P$ and $\hat{p} = \max P$ and $\mathring{C}_i = C_i \setminus \{\check{p}, \hat{p}\}$ for $i \in \{1, 2\}$.
Denoting $\hat{l} = \max q^{-1}(\check{p})$, we must have $\hat{l} < \max q^{-1}(\mathring{C}_i)$ for $i \in \{1, 2\}$ by \Cref{lem:lat_reg}.
If we have $x$ such that $\hat{l} < x \leqslant \max q^{-1}(\mathring{C}_i)$, then $q(x) \in \mathring{C}_i$: indeed, $\hat{l}$ is the maximal element of $q^{-1}(\check{p})$, 
 $[\check{p}, \max\mathring{C}_i] = \{\check{p}\} \cup \mathring{C}_i$,
and $q$ is order-preserving.
Since $L$ is finite, this implies that there exist $x_i \in q^{-1}(\mathring{C}_i)$ for $i \in \{1, 2\}$ such that $\hat{l} \lessdot x_i$.
Since $\hat{l}$ is a polygonal lattice, we have that $[\hat{l}, x_1 \vee x_2]$ is a polygon.
Moreover, $q([\hat{l}, x_1 \vee x_2]) = [q(\hat{l}), q(x_1 \vee x_2)] = [\check{p}, q(x_1) \vee q(x_2)] = [\check{p}, \hat{p}] = P$, which proves the result.
\end{proof}

\begin{lemma}\label{lem:polygon_connected}
Let $q\colon L \to M$ be a lattice quotient of two finite lattices $L$ and $M$, where $L$ is polygonal.
Then in each fibre of the map $\mgmap{q} \colon \mg{L} \to \mg{M}$ any two elements are connected by a sequence of polygon moves.
\end{lemma}
\begin{proof}
Let $C$ be a maximal chain of $M$ given by $m_0 \lessdot m_1 \lessdot \dots \lessdot m_r$.
Since $q$ is a lattice quotient, we have that $q^{-1}(m_i)$ is an interval in $L$.
The fibre $\mgmap{q}^{-1}(C)$ consists of maximal chains in $L$ which are subsets of $\bigcup_{i = 0}^{r} q^{-1}(m_i)$.

Given a maximal chain $C_{L}$ in $\mgmap{q}^{-1}(C)$, we denote by $\min_i(C_{L})$ and $\max_{i}(C_{L})$ the respective minimal and maximal elements of $C_{L} \cap q^{-1}(m_i)$.
Then $[\min_i(C_{L}), \max_i(C_{L})] \subseteq q^{-1}(m_i)$.
By \Cref{prop:polygonal_lattices}\ref{op:polygonal_lattices:polygon_moves} and the straightforward fact that an interval in a polygonal lattice is polygonal, we have that any two maximal chains in $[\min_{i}(C_{L}), \max_{i}(C_{L})]$ are related by a sequence of polygon moves.
Hence, if we have maximal chains $C_{L}$ and $C'_{L}$ in $\mgmap{q}^{-1}(C)$ such that $\min_{i}(C_{L}) = \min_{i}(C'_{L})$ and $\max_{i}(C_{L}) = \max_{i}(C'_{L})$ for all~$i$, then we have that $C_{L}$ and $C'_{L}$ are related by a sequence of polygon moves.

Thus, let $C_{L}, C'_{L} \in \mgmap{q}^{-1}(C)$ and choose the smallest $j$ for which we have either $\max_{j}(C_{L}) \neq \max_{j}(C'_{L})$ or $\min_{j + 1}(C_{L}) \neq \min_{j + 1}(C'_{L})$.
Note that we cannot have $j = r$, since $\max_{r}(C_{L}) = \max_{r}(C'_{L}) = \max L$; for similar reasons, we cannot have $j = -1$.
We will construct $\widehat{C}_L, \widehat{C}'_{L} \in \mgmap{q}^{-1}(C)$ such that $C_L$ and $\widehat{C}_L$ are related by a sequence of polygon moves, $C'_L$ and $\widehat{C}'_L$ are related by a sequence of polygon moves, and $\max_{j}(\widehat{C}_{L}) = \max_{j}(\widehat{C'}_{L})$ and $\min_{j + 1}(\widehat{C}_{L}) = \min_{j + 1}(\widehat{C}_{L})$.
Let $\max_j(C_{L}) = x$, $\max_j(C'_{L}) = x'$, $\min_{j + 1}(C_{L}) = y$, and $\min_{j + 1}(C'_{L}) = y'$.
Note that we must therefore have $x \lessdot y$ and $x' \lessdot y'$.
Let $\overline{x} = x \wedge x'$, $\overline{y} = y \wedge y'$ and $\check{x} := \min_{j}(C_{L}) = \min_{j}(C'_{L})$.

We first construct $\widehat{C}_L$.
We have $[\check{x}, y] \subseteq q^{-1}(m_j) \cup q^{-1}(m_{j + 1})$, since $m_j \lessdot m_{j + 1}$ and $q$ is order-preserving.
We have that $\overline{x} < \overline{y}$ and $\overline{x},\overline{y} \in [\check{x}, y]$, and so there is a maximal chain $\widehat{C}$ in this interval with $\overline{x}, \overline{y} \in \widehat{C}$.
Let $c_j = \max(q^{-1}(m_j) \cap \widehat{C})$ and $c_{j + 1} = \min(q^{-1}(m_{j + 1})\cap\widehat{C})$.
Note then that we must have $\overline{x} \leqslant c_{j} \lessdot c_{j + 1} \leqslant \overline{y}$, since $\overline{x} \in q^{-1}(m_j) \cap \widehat{C}$ and $\overline{y} \in q^{-1}(m_{j + 1}) \cap \widehat{C}$.
Since $C_{L}$ also passes through $\check{x}$ and $y$, we can define $\widehat{C}_{L} := (C_{L} \setminus [\check{x}, y]) \cup \widehat{C}$.
We therefore have that $C_{L}$ and $\widehat{C}_{L}$ are related by polygon moves within the interval $[\check{x}, y]$, noting that they agree outside this interval.

Arguing similarly, we can find a maximal chain $\widehat{C'}$ in $[\check{x}, y']$ which coincides with $\widehat{C}$ on $[\overline{x}, \overline{y}] \subseteq [\check{x}, y']$, and in particular 
also passes through $c_{j}$ and $c_{j + 1}$.
We now define a maximal chain $\widehat{C'}_{L} := (C'_{L} \setminus [\check{x}, y']) \cup \widehat{C'}$. By construction, it is related to $C'_{L}$ by polygon moves, 
and we have $\max_{j}(\widehat{C}_{L}) = \max_{j}(\widehat{C'}_{L}) = c_j$ and $\min_{j + 1}(\widehat{C}_{L}) = \min_{j + 1}(\widehat{C'}_{L}) = c_{j + 1}$.
Continuing inductively in this way, we can suppose that $\widehat{C}_L$ and $\widehat{C'}_L$ are related by polygon moves, and so we have that $C_{L}$ and $C_{L'}$ are related by polygon moves, as desired.
We conclude that any two maximal chains in $\mgmap{q}^{-1}(C)$ are related by a sequence of polygon moves.
\end{proof}

\section{Preorders on equivalence classes of maximal chains via edge labellings}\label{sect:max_chain_posets}

In this section, we study orders on equivalence classes of maximal chains in posets.
This gives a general framework for defining partial orders (in general, preorders) such as we define in Section~\ref{sect:preproj_mgs_orders} in the particular case of maximal green sequences of preprojective algebras.

\subsection{Background}\label{sect:back:posets}

We first give background on a couple of concepts needed to introduce the partial orders on sets of square-equivalence classes of maximal chains.

\subsubsection{Edge labellings}

Edge labellings give us extra data on a polygonal lattice that allow us to define a preorder on its set of equivalence classes of maximal chains.
Recall that the set of covering relations of a poset $P$ is denoted~$\covrel{P}$.
A \emph{set-valued} \emph{edge labelling} of a poset $P$ is a map $\lambda \colon \covrel{P} \to E$, where $E$ is a set.
If $E$ is a poset, then we call $\lambda$ a \emph{poset-valued} \emph{edge labelling}; unless otherwise specified, we will use ``edge labelling'' to refer to poset-valued edge labellings.
A maximal chain $C$ given by $p_1 \lessdot p_2 \lessdot \dots \lessdot p_l$ in $P$ is \emph{ascending} with respect to the edge labelling $\lambda$ if we have $\lambda(p_1, p_2) < \lambda(p_2, p_3) < \dots < \lambda(p_{l - 1}, p_{l})$ in~$E$.
We define \emph{descending} chains similarly.

\subsubsection{Forcing}

Given a lattice $L$, forcing is a particular preorder on the set of covering relations $\covrel{L}$.
It will give us a constraint on the classes of edge labellings we allow so that they behave well with respect to quotient lattices.
If $a \lessdot b$ and $c \lessdot d$ are two covering relations of $L$, we say that $a \lessdot b$ \emph{forces} $c \lessdot d$ if for any lattice congruence $\theta$ such that $a \mathrel{\theta} b$, we have $c \mathrel{\theta} d$ \cite[Section~9-5.4]{reading_regions}.
The equivalence relation induced by the forcing preorder is known as ``forcing-equivalence''.
Explicitly, two covering relations $a \lessdot b$ and $c \lessdot d$ of $L$ are \emph{forcing-equivalent} if for any lattice congruence $\theta$ on $L$, we have $a \mathrel{\theta} b$ if and only if $c \mathrel{\theta} d$.

In the case where $L$ is a finite polygonal lattice, it was shown in \cite[Theorem~9-6.5]{reading_regions} that for any polygon $P$ of $L$ with elements denoted 
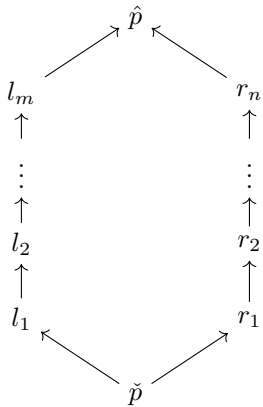
\[
\begin{tikzpicture}[xscale=1.5]

\node (min) at (0,0) {$\check{p}$};

\node (l1) at (-1,1) {$l_1$};
\node (l2) at (-1,2) {$l_2$};
\node (l3) at (-1,3) {$\vdots$};
\node (l4) at (-1,4) {$l_{m}$};

\node (r1) at (1,1) {$r_1$};
\node (r2) at (1,2) {$r_2$};
\node (r3) at (1,3) {$\vdots$};
\node (r4) at (1,4) {$r_{n}$};

\node (max) at (0,5) {$\hat{p}$};

\draw[->] (min) -- (l1);
\draw[->] (l1) -- (l2);
\draw[->] (l2) -- (l3);
\draw[->] (l3) -- (l4);
\draw[->] (l4) -- (max);

\draw[->] (min) -- (r1);
\draw[->] (r1) -- (r2);
\draw[->] (r2) -- (r3);
\draw[->] (r3) -- (r4);
\draw[->] (r4) -- (max);

\end{tikzpicture}
\]
then $\check{p} \lessdot l_1$ and $r_{n} \lessdot \hat{p}$ are forcing-equivalent, $\check{p} \lessdot r_1$ and $l_{m} \lessdot \hat{p}$ are forcing-equivalent, and $\check{p} \lessdot l_1$, $\check{p} \lessdot r_1$, $l_{m} \lessdot \hat{p}$, and $r_{n} \lessdot \hat{p}$ all force $l_i \lessdot l_{i + 1}$ and $r_{i} \lessdot r_{i + 1}$ for all $i$.
Moreover, the forcing preorder is the smallest preorder on $\covrel{L}$ for which these relations hold.

\subsection{Polygonal edge labellings}

In order to get edge labellings which give us partial orders on the set of square-equivalence classes with maximal chains, we need to restrict to certain classes of edge labellings.
First, we will restrict to the following type of edge labellings, which respect the equivalence relation given by forcing.

\begin{definition}
An edge labelling $\lambda \colon \covrel{L} \to E$ of a finite lattice $L$ is \emph{forcing-consistent} if whenever $a \lessdot b$ and $c \lessdot d$ are forcing-equivalent covering relations, we have $\lambda(a, b) = \lambda(c, d)$.
\end{definition}

In the case of a finite polygonal lattice, forcing consistency is equivalent to condition (i) in the definition of a CN-labelling in \cite{reading2003}.
The notion of a forcing-consistent edge labelling is also implicit in \cite[Theorem~3.11(a)]{dirrt}.

\begin{definition}
Given a finite polygonal lattice $L$ and a finite poset $E$, an edge labelling $\lambda \colon \covrel{L} \to E$ is a \emph{polygonal edge labelling} if for every non-square polygon $P$ of $L$, one of the maximal chains of $P$ is ascending with respect to $\lambda$, and the other maximal chain is descending with respect to $\lambda$.
\end{definition}

Polygonal edge labellings are in some ways similar to the EL-labellings from \cite{bjorner}, but one of the key differences is that EL-labellings are required to have unique ascending maximal chains in all intervals, whereas we only require this for intervals which are non-square polygons.
Hence, if $L$ contains a polygon $P$ where the descending chain has at least three covering relations, a polygonal edge labelling $\lambda$ (and, moreover, a linear extension of $\lambda$) is never an EL-labelling.
Indeed, any subchain of length two of the descending chain with respect to $\lambda$ in such a polygon $P$ is itself an interval in $L$, which does have an ascending chain with respect to $\lambda$.

Given a forcing-consistent polygonal edge labelling on a poset $P$, one can define a preorder on the equivalence classes of maximal chains.
This is similar to the notion of a maximal chain descent order from \cite{Lacina}, but there are several differences, namely: maximal chain descent orders are always partial orders; the notion of polygon in \cite{Lacina} is different and the context is finite bounded posets rather than polygonal lattices; in \cite{Lacina} the edge labellings are required to be CL-labellings; and in \cite{Lacina} no equivalence relation is imposed on the maximal chains.

\begin{definition}\label{def:increasing_polygon_move}
Let $L$ be a finite lattice with $\lambda \colon \covrel{L} \to E$ a forcing-consistent polygonal edge labelling.
Given two maximal chains $C, C' \in \mg{L}$, we say that $C'$ is an \emph{increasing polygon move} of $C$ if there is a non-square polygon $P$ of $L$ such that $C \setminus P = C' \setminus P$ and $C$ contains the maximal chain in $P$ which is ascending with respect to $\lambda$, and $C'$ contains the other maximal chain in~$P$.
By extension, we also say that $[C']$ is an increasing polygon move of~$[C]$.

Recall that we write $\mge{L}$ for the set of square-equivalence classes of maximal chains of~$L$.
We define a preorder $\preord_{\lambda}$ on $\mge{L}$ as the smallest reflexive transitive relation $\preord_{\lambda}$ such that \[
[C] \preord_{\lambda} [C']
\]
if $[C']$ is an increasing polygon move of $[C]$.
\end{definition}

We thus denote the preordered set given by $\mge{L}$ with the preorder $\preord_{\lambda}$ by $\mgel{L}$.
If $\mgel{L}$ is such that $[C] \preordot [C']$ whenever $[C']$ is an increasing polygon move of $[C]$, then we say that $\lambda \colon \covrel{L} \to E$ is \emph{polygon-complete}.
The terminology of polygon-completeness comes from \cite[Definition~3.19]{Lacina}.
Note that if the edge labelling is polygon-complete then the preorder $\preord_{\lambda}$ is a partial order.

\begin{lemma}
\label{lem:polygon-complete_poset}
Let $L$ be a finite polygonal lattice with $\lambda \colon \covrel{L} \to E$ a forcing-consistent polygonal edge labelling.
If $\lambda$ is polygon-complete, then $\preord_{\lambda}$ is a partial order.
\end{lemma}
\begin{proof}
Suppose that $\preord_{\lambda}$ is not a partial order.
Then there exists a sequence of equivalence-classes of maximal chains $[C_1], [C_2], \dots, [C_k]$ such that for all $i$, $[C_{i + 1}]$ is an increasing polygon move of $[C_i]$ and $[C_1]$ is an increasing polygon move of $[C_k]$.
But then we have $[C_1] \preord_{\lambda} [C_k]$ and $[C_k] \preord_{\lambda} [C_1]$, which means that we cannot have either $[C_1] \preordot_{\lambda} [C_k]$ or $[C_k] \preordot_{\lambda} [C_1]$.
Hence, $\lambda$ cannot be polygon-complete.
\end{proof}

\subsection{Edge labelling inherited by quotients of polygonal lattices}

In this section we give a natural construction of an edge labelling on a quotient polygonal lattice from a forcing-consistent edge labelling on the original polygonal lattice.

\begin{lemma}\label{lem:cov_rel_labels_equal}
Let $L$ be a finite polygonal lattice with $\theta$ a lattice congruence and $\lambda \colon \covrel{L} \to E$ a forcing-consistent edge labelling.
Then if $[x] \lessdot [y]$ is a covering relation in $L/\theta$, we have that $\lambda(x', y')$ is constant for all covering relations $x' \lessdot y'$ with $x' \in [x]$ and $y' \in [y]$.
\end{lemma}
\begin{proof}
Let $[x] \lessdot [y]$ be a covering relation in $L/\theta$ and let $\hat{x}$ be the maximal element of $[x]$.
It follows from \cite[Proposition~2.2]{reading_lcfha} that there is a unique $\widetilde{y} \in [y]$ which covers $\hat{x}$.
Now let $x_0 \lessdot y_0$ be any covering relation with $x_0 \in [x]$ and $y_0 \in [y]$.
We will show inductively that $x_0 \lessdot y_0$ is forcing-equivalent to $\hat{x} \lessdot \widetilde{y}$.
If $x_0 = \hat{x}$ then this is tautologous, so we may assume $x_0 < \hat{x}$.
Hence, there exists $x'$ with $x_0 \lessdot x' \leqslant \hat{x}$.
We may then take the polygon $P_0 = [x_0, x' \vee y_0]$.
We have that the $\theta$-equivalence class of $\max P_0$ is $[\max P_0] = [x' \vee y_0] = [x'] \vee [y_0] = [y_0]$.
Hence $\max P_0 \in [y]$, so we denote $y_1 := \max P_0$.
We let $x_1$ be the coatom of $P_0$ such that $x' < x_1 \lessdot y_1$.
Since the quotient map is order-preserving, we must have $x_1 \in [x]$ or $x_1 \in [y]$.
The other coatom of $P_0$ is an element of $[y]$, since it lies between $y_0$ and $y_1$.
Since $[y]$ is closed under meets, we must have that $x_1 \in [x]$, as otherwise we would have $x_0 \in [y]$.
Moreover, $x_1 \lessdot y_1$ and $x_0 \lessdot y_0$ are forcing-equivalent by the polygonal characterisation of forcing-equivalence.
Inductively, we obtain that $x_0 \lessdot y_0$ and $\hat{x} \lessdot \widetilde{y}$ are forcing-equivalent.
Since $\lambda$ is forcing-consistent, we obtain that $\lambda(x_0, y_0) = \lambda(\hat{x}, \widetilde{y})$, which establishes the result.
\end{proof}

This allows us to define a quotient edge labelling for forcing-consistent edge labellings of polygonal lattices.

\begin{definition}\label{def:quotient_edge_labelling}
Suppose that $L$ is a finite polygonal lattice with a lattice congruence $\theta$ and a forcing-consistent edge labelling $\lambda \colon \covrel{L} \to E$.
We then define the \emph{quotient edge labelling} $\lambda_\theta \colon \covrel{L/\theta} \to E$ by $\lambda_{\theta}([x], [y]) = \lambda(x', y')$ for any $x' \in [x]$ and $y' \in [y]$ with $x' \lessdot y'$.
\end{definition}

\begin{proposition}\label{prop:qel_fc&polygonal}
Let $L$ be a finite polygonal lattice with $\lambda \colon \covrel{L} \to E$ a forcing-consistent edge labelling and $\theta$ a lattice congruence on $L$, with quotient map $q \colon L \to L/\theta$.
Then the following hold.
\begin{enumerate}
    \item $\lambda_{\theta}$ is also forcing-consistent.\label{op:qel_fc&polygonal:fc}
    \item If $\lambda$ is polygonal, then $\lambda_{\theta}$ is polygonal as well.\label{op:qel_fc&polygonal:polygonal}
\end{enumerate}
\end{proposition}
\begin{proof}
It suffices to check forcing-consistency and polygonality within each polygon $P$ of $L/\theta$.
Hence, let $P$ be a polygon of $L/\theta$, with $P'$ a polygon of $L$ such that $q(P') = P$, which exists by \Cref{lem:polygon_preimages}.
Let the minima and maxima of $P$ and $P'$ be $\check{p}$, $\hat{p}$, $\check{p}'$, and $\hat{p}'$ respectively, and let the respective atoms and coatoms of $P$ be $a_1$, $a_2$, $c_1$, $c_2$.
We let the respective atoms and coatoms of $P'$ be $a'_1$, $a'_2$, $c'_1$, and $c'_2$, such that $q(a'_i) = a_i$ and $q(c'_i) = c_i$ for $i \in \{1, 2\}$.
Hence $\lambda_{\theta}(\check{p}, a_i) = \lambda(\check{p}', a'_i)$ and $\lambda_{\theta}(c_i, \hat{p}) = \lambda(c'_i, \hat{p}')$ and so forcing-consistency of $\lambda$ within the polygon $P'$ ensures that $\lambda_{\theta}$ is forcing-consistent within the polygon~$P$.
For polygonality, it suffices to observe that the sequences of edge labels of the maximal chains of $P$ are obtained from those of $P'$ by removing the edge labels of the edges that are contracted by~$\theta$.
Then since one maximal chain of $P'$ is ascending and the other descending, it follows that the same is true for the maximal chains of $P$.
\end{proof}

Given this set-up, we can now prove one of the main theorems of the paper.

\begin{theorem}\label{thm:contraction}
Suppose that $L$ is a finite polygonal lattice with $\lambda \colon \covrel{L} \to E$ a forcing-consistent polygonal edge labelling.
\begin{enumerate}[label=\textup{(}\arabic*\textup{)}]
    \item Denoting the quotient $q \colon L \to L/\theta$, the induced map \[
    \mgtmap{q} \colon \mgea{\lambda}{L} \to \mgea{\lambda_{\theta}}{L/\theta}
    \] is an order-preserving surjection of preordered sets.\label{op:contraction:op_surj}
    \item If $\lambda$ is polygon-complete, then $\mgtmap{q}$ is a contraction of preordered sets.\label{op:contraction:contraction}
\end{enumerate}
\end{theorem}

\begin{proof}[Proof of \Cref{thm:contraction}, Claim~\ref{op:contraction:op_surj}]
Note that the fact that $\mgtmap{q}$ is a well-defined surjection follows from Corollary~\ref{prop:max_chains_eq}.
We thus only have to show that $\mgtmap{q}$ is order-preserving.
It suffices to show that if there are two maximal chains $C$ and $C'$ of $L$ such that $C'$ is an increasing polygon move of $C$, then $\mgtmap{q}([C]) \preord_{\lambda_{\theta}} \mgtmap{q}([C'])$.
Let the polygon relating $C$ and $C'$ be \[P = \polygon{\{a, p_1, p_2, \dots, p_l, b\}}{\{a, p'_1, p'_2, \dots, p'_{l'}, b\}}\] with $C \cap P = \{a, p_1, p_2, \dots, p_l, b\}$ and $C' \cap P := \{a, p'_1, p'_2, \dots, p'_{l'}, b\}$.
If $q(a) = q(b)$, then $\mgtmap{q}([C]) = \mgtmap{q}([C'])$, so we may assume that $q(a) < q(b)$.
Moreover, if $q(p_i) = q(p'_j)$ for any $i, j$, then we must have $q(a) = q(p_i \wedge p'_j) = q(p_i) \wedge q(p'_j) = q(p_i) \vee q(p'_j) = q(p_i \vee p'_j) = q(b)$.
Hence, we may also assume that $q(p_i) \neq q(p'_j)$ for any $i, j$.
Thus, we have that $q(P)$ is a subposet of $L/\theta$ which has precisely two maximal chains, which intersect only at $q(a)$ and $q(b)$.
By \Cref{prop:polygonal_lattices}\ref{op:polygonal_lattices:quotient}, $L/\theta$ is also polygonal and so $q(P)$ is a polygon.

By assumption, $C \cap P$ is the unique ascending chain in $P$ with respect to the polygonal edge labelling.
By construction of $\lambda_{\theta}$, it is straightforward to see that $q(C \cap P)$ is the unique ascending chain of $q(P)$.
Hence, we have that $q(C')$ is indeed an increasing polygon move of $q(C)$ across the polygon $q(P)$, as desired.
\end{proof}

\begin{proof}[Proof of \Cref{thm:contraction}, Claim~\ref{op:contraction:contraction}]
We first show that if $[C_{1}] \preordot_{\lambda_{\theta}} [C_{2}]$ in $\mgel{L/\theta}$, then there exist $[C'_{1}] \preordot_{\lambda} [C'_{2}]$ in $L$ such that $\mgtmap{q}([C'_{1}]) = [C_{1}]$ and $\mgtmap{q}([C'_{2}]) = [C_{2}]$.
Let $[C_{1}], [C_{2}] \in \mge{L/\theta}$ such that $[C_{1}] \preordot_{\lambda_{\theta}} [C_{2}]$.
Note that, by definition, this means that $[C_{2}] \not\preord_{\lambda_{\theta}} [C_{1}]$.
Since the preorder $\preord_{\lambda_{\theta}}$ is the smallest preorder such that $[C] \preord_{\lambda_{\theta}} [C']$ when $[C']$ is an increasing polygon move of $[C]$, we have that if $[C_{1}] \preordot_{\lambda_{\theta}} [C_{2}]$, then $[C_2]$ must be an increasing polygon move of~$[C_1]$.
By \Cref{lem:polygon_moves_in_fibres}, we can find maximal chains $C'_{1}$ and $C'_{2}$ which differ by a polygon move such that $\mgmap{q}(C'_{1}) = C_{1}$ and $\mgmap{q}(C'_{2}) = C_{2}$.
Moreover $C'_{1}$ and $C'_{2}$ must differ by a non-square polygon move since this is true of $C_{1}$ and $C_{2}$.
Hence, we either have $[C'_{1}] \preordot_{\lambda} [C'_{2}]$ or $[C'_{2}] \preordot_{\lambda} [C'_{1}]$, since $L$ was assumed to be polygon-complete.
However, since $\mgtmap{q}$ is order-preserving, only the former is possible, as desired.

Now we show that the preimages of the $\phi_{\preord_{\lambda}}$-equivalence classes of $\mgtmap{q}$ are connected.
Let $C \in \mg{L/\theta}$.
We first wish to show that the Hasse diagram of the fibre $\mgtmap{q}^{-1}([C])$ is connected.
For this it suffices to show that any two elements of the fibre $\mgtmap{q}^{-1}([C])$ are related by a sequence of polygon moves.
In fact, because we know from \Cref{lem:polygon_connected} that any two elements of the fibre $\mgmap{q}^{-1}(C)$ are related by a sequence of polygon moves, it suffices to show that if $\{C_1, C_2\} \subseteq [C]$ with $C_1$ and $C_2$ related by a square move, then there are elements in $\mgmap{q}^{-1}(C_1)$ and $\mgmap{q}^{-1}(C_2)$ which are related by a polygon move.
This then follows from \Cref{lem:polygon_moves_in_fibres}.

Now we show that the pre-images of the $\phi_{\preord_{\lambda_{\theta}}}$-equivalence classes are likewise connected.
Two elements $[C]$ and $[C']$ are $\phi_{\preord_{\lambda_{\theta}}}$-equivalent if there is a sequence of increasing polygon moves from $[C]$ to $[C']$ and also a sequence of increasing polygon moves from $[C']$ to $[C]$.
Hence, it suffices to show that if $[C']$ is an increasing polygon move of $[C]$ then there are elements of $\mgtmap{q}^{-1}([C])$ and $\mgtmap{q}^{-1}([C'])$ which are related by a polygon move and we already showed this.
\end{proof}

\section{Lattices of regions of simplicial hyperplane arrangements}
\label{sect:regions}

In this section, we show how our framework from the previous section applies very naturally to the lattice of regions of a hyperplane arrangement, where the hyperplanes give us an edge labelling with the properties we want.

\subsection{Background}

Our main reference for lattices of regions of hyperplane arrangements is \cite{reading_regions}.

\subsubsection{Hyperplane arrangements and the poset of regions}

Given a field $K$, a (\emph{linear}) \emph{hyperplane} $H$ in a finite-dimensional $K$-vector space $V$ is a vector subspace of codimension~$1$.
A \emph{hyperplane arrangement} in $V$ is a finite set of hyperplanes in $V$.
In this paper, we will always consider hyperplane arrangements where $K = \mathbb{R}$.

Given a hyperplane arrangement $\mathcal{H}$ in $\mathbb{R}^{n}$, its \emph{complement} is $\mathbb{R}^{n} \setminus \bigcup_{H \in \mathcal{H}} H$, and the closures of the connected components of its complement are called \emph{regions}.

A \emph{cone} in $\mathbb{R}^{n}$ is a subset $R \subseteq \mathbb{R}^{n}$ such that if $x, y \in R$ and $\lambda, \mu \in \mathbb{R}_{>0}$, then $\lambda x + \mu y \in R$.
The regions of a hyperplane arrangement $\mathcal{H}$ are cones.
A \emph{face} $F$ of a closed polyhedral cone $R$ is a subset of $R$ such that there is a linear functional $\psi \in (\mathbb{R}^{\ast})^n$ such that $F = \{x \in R \st \psi(x) \geqslant \psi(x') \, \forall x' \in R\}$.
A cone $R$ is \emph{polyhedral} if there is a finite subset $B$ of $\mathbb{R}^{n}$ such that $R = \{\sum_{b \in B} \lambda_{b} b \in \mathbb{R}^{n} \st \lambda_{b} \in \mathbb{R}_{\geqslant 0}\}$, and is called \emph{simplicial} if $B$ is in fact a basis of $\mathbb{R}^n$.
A hyperplane arrangement $\mathcal{H}$ is \emph{simplicial} if every region $R$ of $\mathcal{H}$ is a simplicial cone.

Let $\mathcal{H}$ be a hyperplane arrangement in $\mathbb{R}^{n}$ and fix a region $B$.
Given a region $A$ of $\mathcal{H}$, we write $S(A)$ for the \emph{separating set} of hyperplanes $H$ of $\mathcal{H}$ such that the interiors $\mathring{A}$ and $\mathring{B}$ lie in different connected components of $\mathbb{R}^n \setminus H$.
This then induces a poset $L(\mathcal{H}, B)$ on the set of regions of $\mathcal{H}$ via $A \leqslant A'$ if and only if $S(A) \subseteq S(A')$.

\begin{fact}
[{\cite[Corollary~9-3.4, Theorem~9-6.10]{reading_regions}}]
If $\mathcal{H}$ is a simplicial hyperplane arrangement, then $L(\mathcal{H}, B)$ is a polygonal lattice for any region $B$ of $\mathcal{H}$.
\end{fact}

\subsubsection{Congruences on the lattice of regions}\label{sect:hyp:back:shards}

Suppose now that $\mathcal{H}$ is a simplicial hyperplane arrangement with $B$ a fixed region, so that $L(\mathcal{H}, B)$ is a polygonal lattice.
Reading describes lattice congruences on the lattice of regions $L(\mathcal{H}, B)$ in terms of the theory of ``shards'', which were introduced in \cite{reading_npsio} and are defined as follows.

Let $U$ be an $(n - 2)$-dimensional subspace of $\mathbb{R}^n$ and write $\mathcal{H}_{U} := \{H \in \mathcal{H} \st U \subset H\}$.
If $|\mathcal{H}_U| \geqslant 2$, then we call $\mathcal{H}_U$ a \emph{rank-two subarrangement} of $\mathcal{H}$.
We write $B_U$ for the region of $\mathcal{H}_U$ containing $B$.
The hyperplanes of $\mathcal{H}_U$ containing facets of $B_U$ are known as the \emph{basic hyperplanes of~$\mathcal{H}_U$}, recalling that a \emph{facet} is a face of codimension one.
Given two hyperplanes $H_1$ and $H_2$ of~$\mathcal{H}$, there is a unique rank-two subarrangement $\mathcal{H}_{H_1 \cap H_2}$ containing them.
We then say that $H_1$ \emph{cuts} $H_2$ if $H_1$ is basic in $\mathcal{H}_{H_1 \cap H_2}$, but $H_2$ is not basic in~$\mathcal{H}_{H_1 \cap H_2}$.
Given a hyperplane $H$ of~$\mathcal{H}$, the \emph{shards} of $H$ are the closures of the connected components of $H \setminus \bigcup_{H' \text{ cuts } H} (H' \cap H)$.
Every shard $\Sigma$ is therefore contained in a unique hyperplane, denoted $H_{\Sigma}$.

Shards of $(\mathcal{H}, B)$ determine lattice congruences on $L(\mathcal{H}, B)$.
Indeed, let $\theta$ be a lattice congruence on $L(\mathcal{H}, B)$.
Further, let $\Sigma$ be a shard of $\mathcal{H}$, with $R$ and $R'$ adjacent regions such that $R \cap R' \subseteq \Sigma$.
We say that $\theta$ \emph{removes the shard} $\Sigma$ if $R \mathrel{\theta} R'$.
It is a fact that a lattice congruence $\theta$ on $L(\mathcal{H}, B)$ is determined by the shards it removes \cite[Proposition~9-7.15]{reading_regions}.
Moreover, if $[R]$ is a $\theta$-equivalence class, then $\bigcup_{R' \in [R]} R'$ is a closed polyhedral cone, which we call a \emph{$\theta$-cone}, and we have that the $\theta$-cones are the closures of the connected components of $\mathbb{R}^{n} \setminus \bigcup \Sigma$, where the union is over all shards $\Sigma$ not removed by $\theta$ \cite[Proposition~9-8.3]{reading_regions}.

\subsection{Maximal chains in the lattice of regions}

Let $\mathcal{H}$ be a simplicial hyperplane arrangement with $B$ a region of $\mathcal{H}$ so that $L(\mathcal{H}, B)$ is the lattice of regions with minimum $B$.
We first define a natural edge labelling of $L(\mathcal{H}, B)$ by $\mathcal{H}$ as a set.
We will then see how to put poset structures on $\mathcal{H}$ to turn this into a poset-valued edge labelling.
We will then show that these edge labellings are polygonal, forcing-consistent, and polygon-complete.

\begin{definition}
Define a set-valued edge labelling $\lambda \colon \covrel{L(\mathcal{H}, B)} \to \mathcal{H}$ by sending $R \lessdot R'$ to the unique hyperplane $H \in \mathcal{H}$ such that $H \supseteq R \cap R'$.
\end{definition}

We now want to upgrade this to a poset-valued edge labelling.
We say that a rank-two subarrangement of $\mathcal{H}$ is \emph{commutative} if it has only two hyperplanes, and that otherwise it is \emph{non-commutative}.

\begin{definition}
Let $C$ be a maximal chain in $L(\mathcal{H}, B)$ with $(H_1, H_2, \dots, H_l)$ the associated sequence of hyperplanes given by the edge labels.
Note that $\mathcal{H} = \{H_1, H_2, \dots, H_l\}$.
Define the \emph{heap poset} $\heap{C}$ to be the smallest partial order on $\mathcal{H}$ containing the relations given by $H_i \leqslant_C H_j$ when $i < j$ and the rank-two subarrangement generated by $H_i$ and $H_j$ is non-commutative.
\end{definition}

This definition in essence coincides with the original definition of the heap poset in \cite[p.550]{Viennot} in the general context of words in monoids.
The following lemma generalises \cite[Proposition~3.1.5]{gl}, where the context is simply-laced Coxeter groups.
We will discuss the context of Coxeter groups in Section~\ref{sect:scm}.

\begin{lemma}\label{lem:heap->equiv}
Two maximal chains $C$ and $C'$ in $L(\mathcal{H}, B)$ are square-equivalent if and only if $\heap{C} = \heap{C'}$.
\end{lemma}
\begin{proof}
The ``only if'' direction is clear, since a square move cannot change the heap poset.
For the converse, let $(H_1, H_2, \dots, H_l)$ and $(H'_1, H'_2, \dots, H'_l)$ be the respective sequences of hyperplanes given by the edge labels of $C$ and $C'$.
Note that we have $\{H_1, H_2, \dots, H_l\} = \{H'_1, H'_2, \dots, H'_l\}$.
Let $\pi$ be the permutation of $[l] := \{1, 2, \dots, l\}$ such that $H_i = H'_{\pi(i)}$ for all $i \in [l]$.
Since $\heap{C} = \heap{C'}$, we have that if $\pi(i) > \pi(l)$, then $H'_{\pi(i)}$ and $H'_{\pi(l)}$ generate a commutative rank-two subarrangement; otherwise, we would have $H'_{\pi(l)} <_{C'} H'_{\pi(i)}$ but $H'_{\pi(l)} = H_l \not<_{C} H_{i} = H'_{\pi(i)}$.
Hence, using square moves repeatedly, one can find a new maximal chain $C''$ whose sequence of hyperplanes is \[(H'_1, H'_2, \dots, \widehat{H'_{\pi(l)}}, \dots, H'_l, H'_{\pi(l)}).\]
Inductively, we can then find a sequence of square moves relating $C''$ and $C$, which completes the proof.
\end{proof}

\begin{definition}
Given a maximal chain $C$ of $L(\mathcal{H}, B)$, we define a poset-valued edge labelling $\lambda^C \colon \covrel{L(\mathcal{H}, B)} \to \heap{C}$ by $\lambda^C = \lambda$.
\end{definition}

The following lemma is known from work of Reading \cite{reading_lcfha}, but the details of the proof will be useful later.

\begin{lemma}\label{lem:polygon->rank2subarr}
Given a polygon $P = \polygon{C_1}{C_2}$ of $L(\mathcal{H}, B)$, there is a rank-two subarrangement $\mathcal{H}'$ of $\mathcal{H}$ such that $\lambda(\covrel{C_i}) = \mathcal{H}'$ for both $i \in \{1, 2\}$.
Moreover $P$ is a square if and only if $\mathcal{H}'$ is commutative.
\end{lemma}
\begin{proof}
Note then that $P \setminus \{\min P, \max P\}$ is homotopy equivalent to two points, that is, a 0-dimensional sphere.
It then follows from \cite[Theorem~5.1]{reading_lcfha} that there is an $(n - 2)$-dimensional cone $F$ such that the regions which are the elements of $P$ are precisely the regions of $\mathcal{H}$ which contain $F$.
It follows that the hyperplanes labelling the covering relations of $P$ are precisely the hyperplanes of $\mathcal{H}$ containing $F$.
Hence, the hyperplanes labelling the covering relations of $P$ form a rank-two subarrangement $\mathcal{H}'$ of $\mathcal{H}$ and for each maximal chain $C_i$ of $P$ we have $\lambda(\covrel{C_i}) = \mathcal{H}'$.
The final statement follows easily by noting that the number of covering relations in $P$ is twice the number of hyperplanes in $\mathcal{H}'$.
\end{proof}

We now show that the edge-labelling $\lambda^{C}$ has the desired properties.

\begin{proposition}\label{prop:hyp_polygonal&forc_cons}
The edge labelling $\lambda^{C} \colon L(\mathcal{H}, B) \to \heap{C}$ is polygonal and forcing-consistent.
\end{proposition}
\begin{proof}
We use the notation of the proof of \Cref{lem:polygon->rank2subarr}.
By projecting onto the orthogonal complement of the span of $F$, we obtain an arrangement of 1-dimensional hyperplanes in $\mathbb{R}^{2}$.
It can then be seen that the two maximal chains of $P$ are labelled by the hyperplanes of $\mathcal{H}'$ in one order, and the hyperplanes of $\mathcal{H}'$ in the opposite order.
From this it follows that forcing-equivalent edges within $P$ must be labelled by the same hyperplane.
(Not all edges labelled by the same hyperplane are necessarily forcing-equivalent, however.)
Since forcing-equivalence is determined by forcing-equivalence within polygons, it follows that $\lambda^{C}$ is forcing-consistent.
In order to show polygonality, note that $C$ must pass through the hyperplanes of $\mathcal{H}'$ in the same order as one of the chains of~$P$.
One can for instance see this by realising $C$ as a path through the regions of $\mathcal{H}$ and then projecting this path onto the orthogonal complement of the span of~$F$.
It follows that one maximal chain of $P$ is ascending with respect to $\lambda^C$ and the other is descending, provided $P$ is non-square.
\end{proof}

\begin{lemma}\label{lem:hyp_partial_order}
The edge labelling $\lambda^{C} \colon L(\mathcal{H}, B) \to \heap{C}$ is polygon-complete and the preordered set $\mgea{\lambda^{C}}{L(\mathcal{H}, B)}$ is a poset.
\end{lemma}
\begin{proof}
We first show that $\preord_{\lambda^C}$ is anti-symmetric on $\mge{L(\mathcal{H}, B)}$, and so a partial order.
It suffices to show that, given a maximal chain $C'$ of $L(\mathcal{H}, B)$, we cannot return to $C'$ using increasing polygon moves and square moves, with at least one increasing polygon move.
But this is clear, since the set of edge labels of every maximal chain is $\mathcal{H}$, and once the order on a non-commutative rank-two subarrangement has been reversed, it cannot be put back in the same order as $C$ by increasing polygon moves and square moves.

Now we show that increasing polygon moves always give covering relations in the partial order.
Suppose that $C_1$ and $C_2$ are maximal chains of $L(\mathcal{H}, B)$ such that $C_2$ is an increasing polygon move of $C_1$ via a non-square polygon $P$.
Then, by \Cref{lem:polygon->rank2subarr}, we have that there is a non-commutative rank-two subarrangement $\mathcal{H}'$ of $\mathcal{H}$ such that $C_1$ orders $\mathcal{H}'$ one way and $C_2$ orders $\mathcal{H}'$ the opposite way.
Since these are the only two possible ways of ordering $\mathcal{H}'$, there can be no maximal chain $C_{1.5}$ such that $[C_1] \preordstr_{\lambda^{C}} [C_{1.5}] \preordstr_{\lambda^{C}} [C_{2}]$, 
since every non-commutative rank-two subarrangment apart from $\mathcal{H}'$ is ordered the same by $C_1$ and $C_2$.
Hence, $[C_1] \preordot_{\lambda^{C}} [C_2]$. 
\end{proof}

Given a lattice congruence $\theta$ on $L(\mathcal{H}, B)$, the set-valued edge labelling $\lambda_{\theta}$ has a natural interpretation in terms of shards.

\begin{proposition}
Let $\theta$ be a lattice congruence on $L(\mathcal{H}, B)$, with $\mathcal{F}_{\theta}$ the corresponding fan.
If $N_1 \lessdot N_2$ is a covering relation in $L(\mathcal{H}, B)$ with $N_1$ and $N_2$ maximal cones of $\mathcal{F}_{\theta}$ separated by a shard~$\Sigma$, then we have $\lambda_{\theta}(N_1, N_2) = H$, where $H$ is the unique hyperplane of $\mathcal{H}$ with $H \supseteq \Sigma$.
\end{proposition}
\begin{proof}
We have by \cite[Proposition~2.2]{reading_lcfha} that there must be at least one pair of regions $R_1$ and $R_2$ of $\mathcal{H}$ such that $R_1 \subseteq N_1$, $R_2 \subseteq R_2$, and $R_1 \lessdot R_2$.
Recalling Section~\ref{sect:hyp:back:shards}, we then have that $N_1 = \bigcup_{R'_1 \in [R_1]} R'_1$ and $N_2 = \bigcup_{R'_2 \in [R_2]} R'_2$.
Moreover, by \Cref{lem:cov_rel_labels_equal}, there is a hyperplane $H$ such that for all $R'_1 \in [R_1]$ and $R'_2 \in [R_2]$ with $R'_1 \lessdot R'_2$, we have that $H \supseteq R'_1 \cap R'_2$.
By construction of the quotient edge labelling, we have that $\lambda_{\theta}(N_1, N_2) = H$. 
Moreover, we have that $H$ is the unique hyperplane with $H \supseteq \Sigma$ as $H \cap \Sigma$ has codimension one, since we have both $H \supseteq R_1 \cap R_2$ and $\Sigma \supseteq R_1 \cap R_2$.
\end{proof}

Applying \Cref{thm:contraction} to \Cref{prop:hyp_polygonal&forc_cons}, we obtain the following result.

\begin{theorem}
Let $\mathcal{H}$ be a simplicial hyperplane arrangement with $B$ a region of $\mathcal{H}$ and $C$ a maximal chain of $L(\mathcal{H}, B)$.
Given a lattice quotient $q \colon L(\mathcal{H}, B) \to L(\mathcal{H}, B)/\theta$, the induced map \[
\mgtmap{q} \colon \mgea{\lambda^{C}}{L(\mathcal{H}, B)} \to \mgea{\lambda^{C}_{\theta}}{L(\mathcal{H}, B)/\theta}
\]
is a contraction of preordered sets.
\end{theorem}

\section{Weak Bruhat order on finite Coxeter groups}\label{sect:scm}

In this section, we apply the results of the previous section to the particular case of simplicial hyperplane arrangements given by finite Coxeter groups.
Here the lattice of regions of the hyperplane arrangement coincides with the so-called ``weak Bruhat order'' on the Coxeter group.
In this situation, we can give particular characterisations of the partial order on equivalence classes of maximal chains.
This set-up will also be applied later to the particular case of Cambrian lattices and their representation-theoretic incarnation.

\subsection{Background}

Our main reference on Coxeter groups is \cite{rs_infinite}, whose conventions we mostly follow.

\subsubsection{Coxeter groups}

A \emph{Coxeter group} $W$ is a group with a presentation given by a set of generators $S$ subject to the relations $(st)^{m(s,t)} = e$ for $s, t \in S$.
We write $e$ for the identity element of $W$.
Here $m(s, s) = 1$ and $2 \leqslant m(s, t)$ for $s \neq t$.
Often one also allows $m(s, t) = \infty$, meaning that no relation of the form $(st)^{m}$ holds, but we will ignore this case since we only ever work with finite Coxeter groups.
The pair $(W, S)$ is called a \emph{Coxeter system} and the elements of $S$ are called \emph{simple reflections}.
The number $|S|$ of simple reflections is called the \emph{rank} of~$W$.
The associated \emph{Coxeter diagram} $\Delta$ of $(W, S)$ has $S$ as vertices with an edge between $s, t \in S$ when $m(s, t) > 2$.
In the case where $m(s, t) > 3$, the associated edge is labelled by $m(s, t)$.
The Coxeter system $(W, S)$ is called \emph{simply-laced} if one always has $m(s, t) \leqslant 3$.

\subsubsection{Weak Bruhat order}

Every element $w \in W$ can be written, usually not uniquely, as a word in the generators $S$.
The smallest length of a word for $w$ is called the \emph{length} of $w$ and is denoted~$\ell(w)$.
A word for $w$ of length $\ell(w)$ is called \emph{reduced}.
We write reduced words for elements of $W$ in bold, so that we may write that $\mathbf{w}$ is a reduced word for $w \in W$.
The (\emph{right}) \emph{weak Bruhat order} is the partial order on $W$ whose covering relations are given by $w \lessdot ws$ whenever $\ell(w) < \ell(ws)$, where $s \in S$.
Equivalently, $wt \lessdot w$ whenever $\ell(wt) < \ell(w)$ for $t \in S$.
When $W$ is finite, as we will be assuming in this paper, this is a lattice.
The maximal element of $W$ in the weak Bruhat order is denoted~$w_{0}$, which thus has the longest length of all elements of $W$.
Note that reduced words for $w_0$ therefore correspond to maximal chains in the weak Bruhat order. 
We say that two reduced words for $w \in W$ are related by a \emph{commutation move} if they are of the form $\mathbf{u}st\mathbf{v}$ and $\mathbf{u}ts\mathbf{v}$ with $s, t \in S$ with $m(s, t) = 2$.
This corresponds to a square move of maximal chains.

\subsubsection{Cartan matrix and reflection representation}

We obtain a reflection representation for $W$ via a choice of Cartan matrix, in the usual way.
Indeed, a (\emph{generalised}) \emph{Cartan matrix} for a Coxeter system $(W, S)$ is a matrix $A$ with rows and columns indexed by $S$ whose entries are real numbers $a_{ss'}$ such that
\begin{enumerate}
    \item $a_{ss} = 2$ for all $s \in S$;
    \item $a_{ss'} \leqslant 0$ with $a_{ss'}a_{s's} = 4 \cos^{2}\bigl(\frac{\pi}{m(s,s')} \bigr)$ for $s \neq s'$;
    \item $a_{ss'} = 0$ if and only if $a_{s's} = 0$.
\end{enumerate}
If the entries of $A$ are integers then we say that $A$ is \emph{crystallographic}, and we say that $W$ itself is \emph{crystallographic} if it has a crystallographic Cartan matrix.
We will require that the Cartan matrix $A$ is \emph{symmetrisable}, meaning that there exists a function $\delta\colon S \to \mathbb{R}_{>0}$ such that $\delta(s)a_{ss'} = \delta(s')a_{s's}$ and $\delta(s) = \delta(s')$ if $s$ and $s'$ are conjugate.
This second condition is from \cite[Section~2.2]{rs_infinite} and is automatic if $A$ is crystallographic \cite[Corollary~2.5]{rs_infinite} or symmetric.

Let $V$ be a real vector space of dimension~$|S|$ with basis denoted $\{\alpha_s \st s \in S\}$.
We define a bilinear form on $V$ by $\sbf{\delta(s)^{-1}\alpha_s}{\alpha_{s'}} = a_{ss'}$.
It can be verified that this form is symmetric.
The reflection representation of $W$ on $V$ is then defined so that $s \in S$ acts by $s(\alpha_{s'}) = \alpha_{s'} - \sbf{\delta(s)^{-1}\alpha_s}{\alpha_{s'}}\alpha_s = \alpha_{s'} - a_{ss'}\alpha_s$.
Although $\delta(s)^{-1}\alpha_s$ is used to define the form $\sbf{-}{-}$, usually the sign of the form is what is important to us, meaning that we can consider $\sbf{\alpha_{s}}{\alpha_{s'}}$ instead of $\sbf{\delta(s)^{-1}\alpha_s}{\alpha_{s'}}$.

\subsubsection{Coxeter arrangement}

The reflection representation of a Coxeter group gives rise to a particular hyperplane arrangement.
The set of \emph{reflections} of $(W, S)$ is defined to be $R = \bigcup_{w \in W} wSw^{-1}$.
For each reflection $r \in R$, $r$ fixes a hyperplane $H_r$ in~$V$.
We let $\mathcal{H}$ be the hyperplane arrangement $\{H_r \st r \in R\}$, known as a \emph{Coxeter arrangement} when $W$ is finite.
If the group $W$ is finite, then it acts transitively on the set of regions of $\mathcal{H}$ \cite[Proposition~10-\nobreak2.2]{reading_fcgwo}.
We fix a region $B$ of $\mathcal{H}$.
We have that $w \mapsto wB$ gives a bijection between $W$ and the regions of~$\mathcal{H}$ \cite[Proposition~10-2.5]{reading_fcgwo}.
The lattice $L(\mathcal{H}, B)$ is then isomorphic to the weak Bruhat order on $W$ \cite[Theorem~10-3.1]{reading_fcgwo}.

\subsubsection{Root systems}

There is a bijection between reflections in $R$ and hyperplanes in $\mathcal{H}$ via $r \mapsto H_r$.
There is another participant in this bijection, namely the positive roots of the associated root system \cite[Proposition~4.4.5]{bb_coxeter}, which is the dual object to the hyperplane arrangement.
For certain purposes in this and future sections it will be more convenient for us to work with the root system than the hyperplane arrangement.

The \emph{root system} of $(W, S)$ is the set $\Phi := \{w(\alpha_s)\}_{w \in W, s \in S}$.
The elements of $\Phi$ are called \emph{roots}, with the elements $\{\alpha_{s}\}_{s \in S}$ called \emph{simple roots}.
Every root $\beta \in \Phi$ can be written as $\beta = \sum_{s \in S}c_{s}\alpha_s$ where $c_{\alpha} \in \mathbb{R}$ and all the non-zero coefficients $c_{\alpha}$ have the same sign. 
One can therefore specify a root $\beta \in \Phi$ via the associated vector of coefficients $(c_{s})_{s \in S}$.
A root $\beta \in \Phi$ is \emph{positive} (respectively, \emph{negative}) if $\beta = \sum_{s \in S}c_{s}\alpha_{s}$ where $c_{s} \geqslant 0$ (respectively, $c_{s} \leqslant 0$) for all $s \in S$.
We denote by $\Phi^{+}$ (respectively, $\Phi^{-}$) the set of \emph{positive} (respectively, negative) roots.
We have that $\Phi = \Phi^{+} \sqcup \Phi^{-}$ \cite[(4.24)]{bb_coxeter}.

The bijection between positive roots $\Phi^{+}$ and reflections $R$ is given by sending a root $\beta$ to the unique reflection $r_{\beta}$ such that $r_{\beta}(\beta) = - \beta$.
The hyperplane corresponding to $\beta$ is then $H_{r_{\beta}}$.

For $w \in W$, we define the \emph{inversion set} $\inv{w}$ by \[\inv{w} := \Phi^{+} \cap w(\Phi^{-}).\]
That is, the inversion set of $w$ is the set of positive roots which are sent to negative roots by $w^{-1}$.
We have that $v \leqslant w$ in the weak Bruhat order if and only if $\inv{v} \subseteq \inv{w}$.
We have $\inv{w} = \Phi^{+}$ if and only if $w$ is the longest element $w_{0}$.
One difference between our conventions and those of \cite{rs_infinite} is that we take inversions to be positive roots, whereas in \cite{rs_infinite} the corresponding reflections are used.

\begin{fact}
[{\cite[Proposition~2]{pilkington}}]
\label{lem:inversion_biclosed}
Given a subset $I \subseteq \Phi^{+}$, we have that $I = \inv{w}$ for $w \in W$ if and only if the following two conditions hold.
\begin{enumerate}
    \item If $\alpha, \beta \in I$ and $\lambda\mu + \alpha\beta \in \Phi^{+}$ for $\alpha, \beta \in \mathbb{R}_{\geqslant 0}$ then $\lambda\alpha + \mu\beta \in I$.
    \item If $\lambda\alpha + \mu\beta \in I$ for $\lambda, \mu \in \mathbb{R}_{\geqslant 0}$, where $\alpha, \beta \in \Phi^{+}$, then either $\alpha \in I$ or $\beta \in I$.
\end{enumerate}
\end{fact}

\subsubsection{Rank-two subsystems}

There is an analogue for root systems of a rank-two subarrangement of $\mathcal{H}$.
A \emph{root subsystem} $\Psi$ of $\Phi$ is a subset of $\Phi$ such that for $\beta \in \Psi$, we have $r_{\beta}(\Psi) = \Psi$.
The \emph{rank} of a root subsystem is the dimension of its span. 
A rank-two root subsystem is \emph{commutative} if the generators $\alpha$ and $\beta$ satisfy $\sbf{\alpha}{\beta} = 0$; in this case, it is straightforward to see that the rank-two subsystem is $\{\alpha, \beta, -\alpha, -\beta\}$.
Otherwise, the rank-two subsystem is \emph{non-commutative}.
We will additionally always require rank-two subsystems not to be strictly contained in larger rank-two subsystems.
(Note that it is possible for one rank-two subsystem to strictly contain another: the rank-two root system of type $B_2$ contains rank-two subsystems of type $A_1 \times A_1$ and the rank-two root system of type $G_2$ contains rank-two subsystems of type $A_2$ and type $A_1 \times A_1$.)
In simply-laced type, all non-commutative rank-two subsystems are of type~$A_2$ and have positive roots of the form $\{\alpha, \alpha + \beta, \beta\}$.

Similarly, there is an analogue of rank-two subsystems and subarrangements for reflections.
A \emph{reflection subgroup} of $W$ is a subgroup generated by reflections; such a subgroup is called \emph{dihedral} if it is generated by two reflections.
In \cite{rs_infinite}, dihedral subgroups are called ``generalised rank two parabolic subgroups''.
We similarly require dihedral subgroups not to be strictly contained in larger dihedral subgroups.
The bijections between roots in $\Phi$, hyperplanes in $\mathcal{H}$, and reflections in $R$ induce bijections between rank-two root subsystems of $\Phi$, rank-two subarrangements of $\mathcal{H}$, and dihedral subgroups of~$W$.

\subsection{Polygonal edge labellings of the weak Bruhat order}\label{sect:cox:edge_label}

Let $(W, S)$ be a Coxeter system with $W$ finite, so that consequently the weak Bruhat order on $W$ is a lattice.
In this section, we apply the edge labellings from the previous section and use the bilinear form $\sbf{-}{-}$ to deduce an interpretation of the partial order on equivalence classes of maximal chains which is similar to the definition of the higher Bruhat orders of Manin and Schechtman.

First, we rephrase the set-up of the previous section in the setting of Coxeter groups.
We use positive roots instead of hyperplanes as edge labels and recall that $L(\mathcal{H}, B)$ is isomorphic to $W$ endowed with the weak Bruhat order.
Hence, we have a set-valued edge labelling $\lambda \colon \covrel{W} \to \Phi^{+}$ of the weak Bruhat order on $W$.

A maximal chain $C$ in $L(\mathcal{H}, B)$ corresponds to a reduced expression $\mathbf{w}_0$ of $w_0$.
Hence, we have the heap poset $\heap{\mathbf{w}_0}$ as a partial order on $\Phi^{+}$.
We furthermore have a poset-valued edge labelling $\lambda^{\mathbf{w}_0} \colon \covrel{W} \to \heap{\mathbf{w}_0}$.

\begin{remark}\label{rmk:combinatorial_ar_quiver}
The Hasse diagram of this poset was studied in \cite{oh-suh, oh-suh_2} under the name ``combinatorial Auslander--Reiten quiver'' in the context of the representation theory of Khovanov--Lauda--Rouquier and quantum affine algebras.
These combinatorial Auslander--Reiten quivers were recently categorified in \cite{canesin}.
For some choices of $\mathbf{w}_0$, combinatorial Auslander--Reiten quivers coincide with Auslander--Reiten quivers in the conventional sense \cite[Theorem~9.3.1]{stump2015cataland}, \cite{bedard}; see also \Cref{rmk:c_sort_hbo}.
\end{remark}

\begin{remark}
Dyer \cite[(4.2)]{dyer1993hecke} defines an edge labelling of the strong Bruhat order on $W$ using reflections, which are in bijection with positive roots.
Our edge labelling is the restriction of this edge labelling to the weak Bruhat order.
\end{remark}

We have the following consequences of the results of the previous section.

\begin{theorem}\label{thm:coxeter}
Let $(W, S)$ be a finite Coxeter system.
\begin{enumerate}[label=\textup{(}\arabic*\textup{)}]
    \item If $\mathbf{w}_0$ and $\mathbf{w}'_0$ are two reduced words for $w_0$, then $\heap{\mathbf{w}_0} = \heap{\mathbf{w}'_0}$ if and only if $\mathbf{w}_0$ and $\mathbf{w}'_0$ are related by a sequence of commutation moves.\label{op:coxeter:heap_vs_commutation}
    \item Given a polygon $P$ in $W$, there is a rank-two root subsystem $\Psi$ of $\Phi$ such that for each maximal chain $C$ of $P$, we have $\lambda(\covrel{C}) = \Psi^{+}$.
    Moreover, $P$ is a square if and only if $\Psi$ is commutative.\label{op:coxeter:polygons}
    \item For every reduced word $\mathbf{w}_0$ for $w_0$, the edge labelling $\lambda^{\mathbf{w}_0} \colon \covrel{W} \to \heap{\mathbf{w}_0}$ is polygonal, forcing-consistent, and polygon-complete.\label{op:coxeter:edge_labelling_properties}
    \item For every reduced word $\mathbf{w}_0$ for $w_0$, $\mgea{\lambda^{\mathbf{w}_0}}{W}$ is a partial order.\label{op:coxeter:partial_order}
    \item Given a reduced word $\mathbf{w}_0$ for $w_0$ and a lattice congruence $\theta$ on $W$ with quotient map $q \colon W \to W/\theta$, the induced map \[\mgtmap{q} \colon \mgea{\lambda^{\mathbf{w}_0}}{W} \to \mgea{\lambda^{\mathbf{w}_0}_{\theta}}{W/\theta}\] is a contraction of preordered sets.\label{op:coxeter:contraction}
\end{enumerate}
\end{theorem}

Indeed, given $v \lessdot w$ in the weak Bruhat order, we have that $\lambda^{\mathbf{w}_0}(v, w) = \beta$ for $\beta$ the unique root such that $\inv{w} \setminus \inv{v} = \{\beta\}$, as follows from the correspondence between roots and hyperplanes.
This means that applying the edge labelling to a maximal chain $C$ of $W$ gives a sequence of roots.
We call these \emph{root sequences}.
These were given the following description in \cite[Proposition~(2.13)]{dyer1993hecke} and \cite{papi}, which also follows from \Cref{lem:inversion_biclosed}.

\begin{fact}
[{\cite[(2.2), Proposition~(2.13)]{dyer1993hecke}
}]\label{lem:admissible_orders}
A total order $<$ on $\Phi^{+}$ is the root sequence of a maximal chain of $W$ if and only if for all $\alpha, \beta \in \Phi^{+}$ with $\lambda\alpha + \mu\beta \in \Phi^{+}$ for some $\lambda, \mu \in \mathbb{R}_{\geqslant 0}$ and $\alpha < \beta$, we have $\alpha < \lambda\alpha + \mu\beta < \beta$.
\end{fact}

In the Coxeter setting, due to the presence of the bilinear form $\sbf{-}{-}$, we can interpret the poset of equivalence classes of maximal chains using the following notion.
We write $\rtss{\Phi}$ for the set of rank-two root subsystems of $\Phi$.
The following notion of inversion generalises that of the two-dimensional higher Bruhat orders.

\begin{definition}
Given a maximal chain $C$ of $W$, with associated root sequence $(\beta_1, \beta_2, \dots, \beta_l)$, we write $\inva{\mathbf{w}_0}{C} \subseteq \rtss{\Phi}$ for the set of non-commutative rank-two root subsystems whose positive roots are ordered differently by $(\beta_1, \beta_2, \dots, \beta_l)$ and $\heap{\mathbf{w}_0}$; we refer to such rank-two subsystems as \emph{inverted} and to $\inva{\mathbf{w}_0}{C}$ as the \emph{inversion~set}.
\end{definition}

We now prove that it satisfies similar properties to the usual notion of inversion for the two-dimensional Bruhat orders.
The following proposition is essentially a special case of \cite[Lemma~3.4(a)]{wellman}.
We will later apply it in Section~\ref{sect:preproj_mgs_orders}.

\begin{proposition}\label{prop:inv<=>equiv}
Given $C, C' \in \mg{W}$, we have that $\inva{\mathbf{w}_0}{C} = \inva{\mathbf{w}_0}{C'}$ if and only if $C$ and $C'$ are square-equivalent.
\end{proposition}
\begin{proof}
This follows immediately from \Cref{thm:coxeter}\ref{op:coxeter:heap_vs_commutation}, since knowing $\inva{\mathbf{w}_0}{C}$ determines the heap poset of $C$, as it determines the partial order between all non-orthogonal roots.
\end{proof}

Hence, we can write $\inva{\mathbf{w}_0}{[C]} := \inva{\mathbf{w}_0}{C}$.
The next result is sketched in \cite[Lemma~3.4(b)]{wellman} and its proof; we provide a detailed proof below.

\begin{proposition}\label{prop:inv_cov_rels}
Given $[C], [C'] \in \mge{W}$, we have that $[C] \lessdot [C']$ if and only if $\inva{\mathbf{w}_0}{[C']} = \inva{\mathbf{w}_0}{[C]} \cup \{\Psi\}$ for some $\Psi \in \rtss{\Phi} \setminus \inva{\mathbf{w}_0}{[C]}$.
\end{proposition}
\begin{proof}
We show the ``only if'' direction first.
We have that $[C] \lessdot [C']$ if and only if there exist $\widehat{C} \in [C]$ and $\widehat{C}' \in [C']$ such that $\widehat{C}'$ is an increasing polygon move of $\widehat{C}$ by polygon-completeness from \Cref{thm:coxeter}\ref{op:coxeter:edge_labelling_properties}.
By polygonality of the edge labelling $\lambda^{\mathbf{w}_0}$ from \Cref{thm:coxeter}\ref{op:coxeter:edge_labelling_properties}, and by the relation between polygons and rank-two subsystems in \Cref{thm:coxeter}\ref{op:coxeter:polygons}, we have that there is exactly one rank-two subsystem which is inverted by the increasing polygon move, which establishes this direction of the proposition.

We now show the ``if'' direction.
Suppose that $\inva{\mathbf{w}_0}{[C']} = \inva{\mathbf{w}_0}{[C]} \cup \{\Psi\}$ for some $\Psi \in \rtss{\Phi} \setminus \inva{\mathbf{w}_0}{[C]}$.
We use the interpretation of the maximal chains $C$ and $C'$ as root sequences.
We let the simple roots of the rank-two root subsystem $\Psi$ be $\alpha$ and $\beta$, so that the positive roots of $\Psi$ are of the form $\lambda\alpha + \mu\beta$ for $\lambda, \mu \in \mathbb{R}_{\geqslant 0}$.
We wish to show that there is a root sequence in the square-equivalence class of $C$ where all of the positive roots of $\Psi$ are consecutive.
Letting $<_C$ denote the total order on $\Phi^+$ given by the edge labels of $C$, we suppose that $\lambda\alpha + \mu\beta <_C \lambda'\alpha + \mu'\beta$ if and only if $\mu/\lambda < \mu'/\lambda'$, which is one of two possible orders by \Cref{lem:admissible_orders}.
Suppose that the positive roots of $\Psi$ are not consecutive in $C$.
Hence, let $\gamma \in \Phi^{+}$ be the $<_C$-minimal root such that $\alpha <_C \gamma <_C \beta$ but with $\gamma \notin \Psi^{+}$.
Let the two roots of $\Psi^{+}$ which $\gamma$ lies immediately between be $\lambda\alpha + \mu\beta$ and $\lambda'\alpha + \mu'\beta$.
In particular, the root immediately preceding $\gamma$ in the root sequence $C$ is $\lambda\alpha + \mu\beta$.

Suppose for contradiction that $\gamma$ and $\lambda\alpha + \mu\beta$ are not orthogonal.
This implies that they generate a non-commutative rank-two root subsystem $\Xi_{\gamma, \lambda\alpha + \mu\beta}$, since adjacent roots in a non-commutative rank-two subsystem cannot be orthogonal.
Note then that
\begin{align*}
\sbf{\gamma}{\lambda\alpha + \mu\beta} &= \sbfb{\gamma}{\frac{\lambda}{\lambda'}(\lambda'\alpha + \mu'\beta) + \bigl(\mu - \mu'\frac{\lambda}{\lambda'}\bigr)\beta} \\
&= \frac{\lambda}{\lambda'}\sbf{\gamma}{\lambda'\alpha + \mu'\beta} + \bigl(\mu - \mu'\frac{\lambda}{\lambda'}\bigr)\sbf{\gamma}{\beta}.
\end{align*}
Since by assumption $\sbf{\gamma}{\lambda\alpha + \mu\beta} \neq 0$, we must have that at least one of $\sbf{\gamma}{\lambda'\alpha + \mu'\beta}$ and $\sbf{\gamma}{\beta}$ are non-zero.
Hence, let $\delta \in \{\lambda'\alpha + \mu'\beta, \beta\}$ be a root such that $\sbf{\gamma}{\delta} \neq 0$.
Then $\gamma$ and $\delta$ also generate a non-commutative rank-two subsystem, which we denote $\Xi_{\gamma, \delta}$.
By assumption, we have that $\Psi$ is the only non-commutative rank-two subsystem which is ordered differently in $C$ and $C'$.
Hence, in $C'$ we must have that $\Xi_{\gamma, \lambda\alpha + \mu\beta}$ and $\Xi_{\gamma, \delta}$ are ordered as they are in~$C$, since these root subsystems cannot equal $\Psi$, since $\gamma \notin \Psi$.
This implies that in $C'$ we have $\lambda\alpha + \mu\beta <_C \gamma <_C \delta$.
However, this is a contradiction, since we must have $\lambda\alpha + \mu\beta >_C \delta$ in $C'$, as $\Psi$ is inverted in $C'$.
We conclude that $\gamma$ and $\lambda\alpha + \mu\beta$ are in fact orthogonal, so that we can swap their order in the root sequence by a square move.
Note that it follows from \Cref{thm:coxeter}\ref{op:coxeter:polygons} and \Cref{lem:admissible_orders} that maximal chains $C$ and $C'$ are related by a square move if and only if the corresponding root sequences are related by swapping adjacent orthogonal roots, noting that such roots generate a commutative rank-two subsystem.

Inductively, we can continue to move $\gamma$ further back in the root sequence until $\gamma <_C \alpha$.
We then have that there is one fewer root in $\Phi^{+} \setminus \Psi^{+}$ lying in between $\alpha$ and $\beta$.
Inductively, we can continue this process until there are no roots in $\Phi^{+} \setminus \Psi^{+}$ in between $\alpha$ and $\beta$, and so the elements of $\Psi^{+}$ are all consecutive.
We can then replace the order on $\Psi^{+}$ with the opposite order by \Cref{lem:admissible_orders}, obtaining a new root sequence; this corresponds to an increasing polygon move from $C$ giving a maximal chain $\widehat{C}'$ such that $\invc{\widehat{C}'} = \invc{C} \cup \{\Psi\}$.
By \Cref{prop:inv<=>equiv}, we have that $[\widehat{C}'] = [C']$, so that $[C] \lessdot [C']$, as desired.
\end{proof}

\section{Cambrian lattices}\label{sect:cambrian}

We now specialise the setting of the previous section to a particular class of congruences, known as \emph{Cambrian congruences}, which were introduced in \cite{reading_cambrian}.
One reason that these congruences are interesting is that they have representation-theoretic interpretations, which we will consider in Section~\ref{sect:preproj_mgs_orders}.
The main result we prove in this section is an interpretation of the maps on maximal chains which resembles taking the stable objects of a Rudakov stability condition \cite{rudakov}.

\subsection{Background}

Cambrian congruences are determined by a choice of Coxeter element $c$ using the notion of $c$-sortable elements of~$W$.

\subsubsection{$c$-sortability}

A \emph{Coxeter element} $c$ of $W$ is an element given by a word $s_{\pi(1)}s_{\pi(2)} \dots s_{\pi(n)}$ where $\pi$ is a rank-$n$ permutation.
Coxeter elements correspond to orientations of the Coxeter diagram: a word for a Coxeter element gives a total order on the vertices of the Coxeter diagram, which gives an orientation of each edge.

Fix such a Coxeter element $c$ given by a word $s_{\pi(1)} s_{\pi(2)} \dots s_{\pi(n)}$.
We write \[c^{\infty} = s_{\pi(1)} s_{\pi(2)} \dots s_{\pi(n)} s_{\pi(1)} s_{\pi(2)} \dots s_{\pi(n)} \dots.\]
Given an element $w \in W$, the \emph{$c$-sorting word} $\mathbf{w}(c)$ for $w$ is the lexicographically first subword of $c^{\infty}$ which is a reduced word for $w$.
The $c$-sorting word may be written $\mathbf{w}(c) = c_{K_1}c_{K_2} \dots c_{K_p}$ where $c_{K}$ denotes the subword of $c$ only using the generators $K \subseteq S$.
The element $w$ is called \emph{$c$-sortable} if its $c$-sorting word $\mathbf{w}(c) = c_{K_1}c_{K_2} \dots c_{K_p}$ is such that $K_1 \supseteq K_2 \supseteq \dots \supseteq K_p$.
Note that the property of being $c$-sortable only depends upon $c$ and not upon the word $s_{\pi(1)} s_{\pi(2)} \dots s_{\pi(n)}$ which represents it.
The \emph{$c$-Cambrian lattice} $W_c$ is the sublattice of $W$ given by the $c$-sortable elements.
It can also be realised as a quotient lattice of $W$ by sending $w \in W$ to the maximal $c$-sortable element beneath $w$, which we denote by $\pi_{\downarrow}^{c}(w)$.
We denote the corresponding lattice congruence by $\theta_{c}$.
The map $\pi_{\downarrow}^{c}$ is, of course, the map sending an element of $W$ to the minimal element in its $\theta_{c}$-equivalence class, as in Section~\ref{sect:back:poset_quotients}.

\subsubsection{$c$-alignment}

It is known that $c$-sortability can be characterised using the alternative notion of $c$-alignment, which will be important for us.
This notion was originally introduced in \cite{reading_clusters}, and was later refined in \cite{rs_infinite}.

Choose a reduced word $\mathbf{c} = s_{\pi(1)} s_{\pi(2)} \dots s_{\pi(n)}$ for the Coxeter element $c$.
The \emph{Euler form} $\ef{-}{-}$ is defined by \[
\ef{\alpha_{s_{\pi(i)}}}{\alpha_{s_{\pi(j)}}} = \left\{
    \begin{array}{ll}
        \sbf{\alpha_{s_{\pi(i)}}}{\alpha_{s_{\pi(j)}}} & \text{if } i > j \\
         1 & \text{if } i = j, \text{ or} \\
        0 & \text{if } i < j.
    \end{array}
\right.
\]

We have that the symmetrisation of the Euler form $\ef{\alpha}{\beta} + \ef{\beta}{\alpha}$ is equal to $\sbf{\alpha}{\beta}$.
The \emph{skew-symmetrised Euler form} $\asef{-}{-}$ \cite[p.19]{rs_infinite} is defined by $\asef{\alpha}{\beta} = \ef{\alpha}{\beta} - \ef{\beta}{\alpha}$.

Let $w \in W$ and let $\Psi$ be a non-commutative rank-two subsystem of $\Phi$.
Label the roots of ${\Psi}^+$ $\beta_1, \beta_2, \dots, \beta_m$ such that $\asef{\beta_i}{\beta_j} > 0$ for $i < j$ (see \cite[Proposition~4.1]{rs_infinite}).
We say that $w$ is \emph{$c$-aligned with respect to $\Psi$} if 
$\inv{w} \cap {\Psi}^+$ is either empty, the singleton $\beta_m$, or an initial segment of $\beta_1, \beta_2, \dots, \beta_m$ \cite[Remark~4.5]{rs_infinite}. Note that,
by \Cref{lem:inversion_biclosed}, 
$w$ is not $c$-aligned with respect to $\Psi$ if an only if  $\inv{w} \cap {\Psi}^+ = \{\beta_j, \beta_{j+1}, \ldots, \beta_m\}$, for some $1 < j < m$.
We say that $w$ is \emph{$c$-aligned} if it is $c$-aligned with respect to all non-commutative rank-two subsystems $\Psi$ of $\Phi$ \cite[p.23 and Theorem~4.3]{rs_infinite}.
We have that $w$ is $c$-aligned if and only if it is $c$-sortable \cite[Theorem~4.3]{rs_infinite}.

\subsection{Contraction of edge labels}

We first show that $c$-alignment can be understood using ordering in the heap poset.

\begin{lemma} \label{lem:rank_2_heap}
Let $\beta_1, \beta_2, \dots, \beta_m$ be the positive roots of a non-commutative rank-two subsystem $\Psi$ of $\Phi$.
Then $\asef{\beta_i}{\beta_j} > 0$ if and only if $\beta_i < \beta_j$ in $\heap{\mathbf{w}_0(c)}$.
\end{lemma}

This follows from the $w_0$ case of \cite[Proposition~3.11]{rs_infinite}, which is then used in the proof of the equivalence of $c$-sortability and $c$-alignedness \cite[Theorem~4.3]{rs_infinite}.
However, it also can be deduced from the equivalence of $c$-sortability and $c$-alignedness in the following way.

\begin{proof}
Note that the $c$-sorting word $\mathbf{w}_0(c)$ gives a maximal chain of $W$, and that it is clear from the definition of $c$-sortability that every element of this maximal chain is $c$-sortable and hence $c$-aligned.
We have that $\beta_i < \beta_j$ in $\heap{\mathbf{w}_0(c)}$ if and only if there is an element $w$ of this maximal chain such that $\beta_i \in \inv{w}$ but $\beta_j \notin \inv{w}$.
Since $w$ must be $c$-aligned, this is true if and only if $\asef{\beta_i}{\beta_j} > 0$.
This is because $\inv{w}$ must contain an initial segment of the total order $<$ on the roots of ${\Psi}^{+}$ given by $\beta < \beta'$ if and only if $\asef{\beta}{\beta'} > 0$.
Note that $\inv{w}$ cannot just contain the final root of $\Psi^{+}$ under the order~$<$, since this would mean that later elements in the maximal chain determined by $\mathbf{w}_0(c)$ would fail to be $c$-aligned.
\end{proof}

We can describe the map $\mgtmap{q_c}$ in terms of edge labels using the following notion.

\begin{definition}
We say that a covering relation $wt \lessdot w$ is \emph{$c$-stable} if $w$ is $c$-aligned with respect to all rank-two subsystems $\Psi$ containing $\beta \coloneqq \lambda^{\mathbf{w}_0(c)}(wt, w)$.
\end{definition}

The main result of this section is then as follows.
Recall that a covering relation $x \lessdot y$ is \emph{contracted} by a congruence $\theta$ if $x \mathrel{\theta} y$.

\begin{theorem}\label{thm:c-stable}
Suppose that $W$ is a finite Coxeter group of simply-laced type, with $c$ a Coxeter element. 
A covering relation $wt \lessdot w$ is contracted by the $c$-Cambrian congruence $\theta_c$ if and only if it is not $c$-stable. 
\end{theorem}

We break down the proof of this theorem into a series of lemmas, the first in fact applying for all finite $W$ and
showing the ``if'' direction.

\begin{lemma}\label{lem:cstable}
Suppose that $W$ is a finite Coxeter group, with $c$ a Coxeter element.
A covering relation $wt \lessdot w$ with edge label $\beta \coloneqq \lambda^{\mathbf{w}_0(c)}(wt, w)$ is $c$-stable if  $\beta \in \inv{\pi^{c}_{\downarrow}(w)}$.  
\end{lemma}

\begin{proof}
Let $\Psi$ be a non-commutative rank-two subsystem containing $\beta$.
We write $(\beta_1, \beta_2, \dots, \beta_m)$ for the list of all positive roots of $\Psi$, ordered as in $\mathbf{w}_0(c)$.
Since $\pi^{c}_{\downarrow}(w)$ is $c$-aligned and $\beta \in \inv{\pi^{c}_{\downarrow}(w)}$, we must have that $\inv{\pi^{c}_{\downarrow}(w)}$ contains an initial segment of $(\beta_1, \beta_2, \dots, \beta_m)$, or $\{\beta\} = \{\beta_m\} = \Psi^{+} \cap \inv{\pi^{c}_{\downarrow}(w)}$.
In the first case, since $\pi^{c}_{\downarrow}(w) \leqslant w$ in $W$, we have that $\inv{w}$ must also contain an initial segment of $(\beta_1, \beta_2, \dots, \beta_m)$.
In the second case, since $\beta$ labels the covering relation $wt \lessdot w$, we have that $\inv{w} \cap \Psi^{+} = \{\beta\} = \{\beta_m\}$ too.
Hence $w$ is $c$-aligned with respect to all non-commutative rank-two subsystems containing $\beta$, and so $\beta$ is $c$-stable.
\end{proof}

If $\alpha$ is a simple root, and $\beta$ is a positive root with a positive coefficient of $\alpha$ in its expansion in terms of simple roots, then we say that $\beta$ \emph{is supported on $\alpha$}.
The following lemma is part of the inductive step for the proof of the other direction of \Cref{thm:c-stable}.
Note that the previous lemma applies to not necessarily simply-laced $W$, whereas our proof of this lemma requires $W$ to be simply-laced.

\begin{lemma}\label{lem:root-theoretic_proof}
Suppose that $W$ is a finite simply-laced Coxeter group, with $c \in W$ a Coxeter element.
Let $s$ be an initial letter of $c$, with corresponding simple root $\alpha_s$.
Let $\beta$ be a positive root with $\beta = \lambda^{\mathbf{w}_0(c)}(wt, w)$ for $wt \lessdot w$ such that $\beta$ is supported on $\alpha_s$.
Then, if $wt \lessdot w$ is $c$-stable, then $\alpha_s \in \inv{w}$.
\end{lemma}
\begin{proof}
Note first that if we are in simply-laced type then there is a unique crystallographic Cartan matrix which is symmetric, and so we can choose $\delta(s) = 1$ for all $s \in S$.

Suppose that the covering relation $wt \lessdot w$ is $c$-stable with edge label~$\beta$.
The case where $\beta = \alpha_s$ is immediate, so we can assume $\beta \neq \alpha_s$.
If $\beta - \alpha_s$ is a root, then it is a positive root, since $\beta$ is supported on $\alpha_s$, and so we are also done, since $\{\alpha_s, \beta, \beta - \alpha_s\}$ form the positive roots of a rank-two subsystem.
Indeed, since $\alpha_s$ is initial in $c$, we conclude that if $\beta$ is $c$-stable, then we must have that $\alpha_s \in \inv{w}$.

Hence, we may assume that $\beta - \alpha_s$ is not a root, and so we have the canonical decomposition $\beta - \alpha_s = \sum_{i = 1}^{l} \gamma_i$ where $\gamma_i \in \Phi^{+}$ \cite{kac_ii}. 
By \cite[Proposition~3(b)]{kac_ii}, we have that $\ef{\gamma_i}{\gamma_j} \geqslant 0$ for all $i \neq j$.
Hence, we have that $\sbf{\gamma_i}{\gamma_j} = \ef{\gamma_i}{\gamma_j} + \ef{\gamma_j}{\gamma_i} \geqslant 0$ for all $i \neq j$. 
Thus, for all $i$,
\begin{align*}
    \sbf{\gamma_i}{\beta} &= \sbf{\gamma_i}{\alpha_s} + \sbf{\gamma_i}{\gamma_i} + \sum_{j \neq i} \sbf{\gamma_i}{\gamma_j} \\
    &= \sbf{\gamma_i}{\alpha_s} + 2 + \sum_{j \neq i} \sbf{\gamma_i}{\gamma_j}.
\end{align*}
However, since we are in simply-laced type, we have that $\sbf{\gamma_i}{\beta} \leqslant 1$ and $\sbf{\gamma_i}{\alpha_s} \geqslant -1$ 
\cite[Section~9.4]{humphreys}, and so we must have $\sbf{\gamma_i}{\beta} = 1$, $\sbf{\gamma_i}{\alpha_s} = -1$, and $\sbf{\gamma_i}{\gamma_j} = 0$ for all $i$ and~$j$.

We moreover have
\begin{align*}
    \sbf{\alpha_s}{\beta} &= \sbf{\alpha_s}{\alpha_s} + \sum_{i = 1}^{l} \sbf{\gamma_i}{\alpha_s} \\
    &= 2 - l.
\end{align*}
If we have $\sbf{\alpha_s}{\beta} = 1$, then $s(\beta) = \beta - \sbf{\alpha_s}{\beta}\alpha_s = \beta - \alpha_s$ is a root, and we have already dealt with this case.
Hence, we may assume that $\sbf{\alpha_s}{\beta} \in \{0, -1\}$, in which case $l \in \{2, 3\}$.

We just write down the $l = 3$ case, since the $l = 2$ case is similar but more straightforward.
Since $\sbf{\gamma_i}{\alpha_s} = -1$, we have roots $\gamma_{i\alpha} := \gamma_i + \alpha_s$, and since $\sbf{\gamma_i}{\beta} = 1$, we have the root $\gamma_{12\alpha} := \beta - \gamma_3$, with $\gamma_{13\alpha}$ and $\gamma_{23\alpha}$ defined similarly.

We have that \[
\asef{\gamma_1}{\gamma_{23\alpha}} = \asef{\gamma_1}{\gamma_2} + \asef{\gamma_1}{\gamma_3} + \asef{\gamma_1}{\alpha_s}.
\]
We have that $\asef{\gamma_1}{\gamma_2} = \asef{\gamma_1}{\gamma_3} = 0$ since $\sbf{\gamma_1}{\gamma_2} = \sbf{\gamma_1}{\gamma_3} = 0$.
Hence, $\asef{\gamma_1}{\gamma_{23\alpha}} = \asef{\gamma_1}{\alpha_s}$, and similarly $\asef{\gamma_{23\alpha}}{\gamma_1} = \asef{\alpha_s}{\gamma_1}$.
We then have $\asef{\alpha_s}{\gamma_1} \geqslant 0$ by \cite[Lemma~3.9]{rs_infinite}, since $s$ is initial in~$c$.
Hence, we have that $\asef{\gamma_{23\alpha}}{\gamma_1} \geqslant 0$ and similarly, we have that $\asef{\gamma_{13\alpha}}{\gamma_2}, \asef{\gamma_{12\alpha}}{\gamma_3} \geqslant 0$.

By considering the rank-two subsystems with positive roots $\{\gamma_{jk\alpha}, \beta, \gamma_i\}$, we must have that $\{\gamma_{23\alpha},$ $\gamma_{13\alpha}, \gamma_{12\alpha}\} \subseteq \inv{w}$ and $\{\gamma_1, \gamma_2, \gamma_3\} \cap \inv{w} = \emptyset$, since in each case $\inv{w}$ must contain the initial segment $(\gamma_{jk\alpha}, \beta)$ of the tuple $(\gamma_{jk\alpha}, \beta, \gamma_i)$ in order to be $c$-aligned with respect to the rank-two subsystem.
Note that, since $\lambda^{\mathbf{w}_0(c)}(wt, w) = \beta$, we cannot have that $\inv{w} \supseteq \{\gamma_{jk\alpha}, \beta, \gamma_i\}$.
Since $\gamma_{ij\alpha} = \gamma_{i\alpha} + \gamma_{j}$ for all $i, j$, by \Cref{lem:inversion_biclosed}, we must have that $\{\gamma_{1\alpha}, \gamma_{2\alpha}, \gamma_{3\alpha}\} \subseteq \inv{w}$.
Similarly, one can apply \Cref{lem:inversion_biclosed} again using $\{\alpha_s, \gamma_{i\alpha}, \gamma_i\}$ to obtain that $\alpha_s \in \inv{w}$, as desired.
\end{proof}

We need the following lemma, which for instance can be found in \cite[Section~2.1]{rs_infinite} or \cite[Lemma~2.5]{enomoto_bruhat}.

\begin{lemma}\label{lem:inversions_after_s}
Let $s \in S$ be a simple reflection with associated simple root $\alpha_{s}$.
For $w \in W$, we have \[
\inv{sw} = \left\{
    \begin{array}{ll}
        \{\gamma \in \Phi^{+} \st s(\gamma) \in \inv{w}\} & \text{if } w \geqslant s \\
        \{\gamma \in \Phi^{+} \st s(\gamma) \in \inv{w}\} \cup \{\alpha_s\} & \text{if } w \not\geqslant s
    \end{array}
\right.
\]
\end{lemma}

We can now prove the converse of \Cref{lem:cstable}, for which we use the following.
A \emph{standard parabolic subgroup} of $W$ is a subgroup $W_J$ generated by a subset $J \subseteq S$.
Given $w \in W$, there is then a unique factorisation $w = w_{J} \cdot \prescript{J}{}{w}$ maximising $\ell(w_J)$ subject to the condition that $w_{J} \in W_J$ and $\ell(w_J) + \ell(\prescript{J}{}{w}) = \ell(w)$.
We will use the case where $J = S \setminus s$.

\begin{lemma}\label{lem:difficult direction}
Suppose that $W$ is a finite simply-laced Coxeter group, with $c$ a Coxeter element.
Let $\beta$ be a positive root with $\beta = \lambda^{\mathbf{w}_0(c)}(wt, w)$ for $wt \lessdot w$ a covering relation in~$W$.
If $wt \lessdot w$ is $c$-stable,
then $\beta \in \inv{\pi^{c}_{\downarrow}(w)}$.
\end{lemma}
\begin{proof}
We use induction on $\ell(w)$ and the rank of $W$.
We cannot have $\ell(w) = 0$, since then $w$ covers nothing.
If $\ell(w) = 1$, then $\pi^{c}_{\downarrow}(w) = w$, as $w$ is clearly $c$-sortable, and so the claim holds in this case.
It also follows from this reasoning that the claim holds when the rank of $W$ is one.

We now suppose for induction that the claim is true for all covering relations $w'u \lessdot w'$ with $\ell(w') < \ell(w)$ and for $W'$ with rank smaller than that of $W$.
Let $s$ be an initial letter of $c$.
In order to carry out the induction, we use the fact that \[
\pi_{\downarrow}^{c}(w) = \left\{
    \begin{array}{ll}
        s\pi_{\downarrow}^{scs}(sw) & \text{if } w \geqslant s, \\
        \pi_{\downarrow}^{sc}(w_{S \setminus s}) & \text{if } w \not\geqslant s, 
    \end{array}
\right.
\]
by \cite[(3.1), Proposition~3.2]{reading_sortable}.

We suppose that the covering relation $wt \lessdot w$ is $c$-stable and first consider the case where $w \geqslant s$.
We need to show that $\beta \in \inv{\pi_{\downarrow}^{c}(w)} = \inv{s\pi_{\downarrow}^{scs}(sw)}$.
We have that $\inv{s\pi_{\downarrow}^{scs}(sw)} = \{\gamma \in \Phi^{+} \st s(\gamma) \in \inv{\pi_{\downarrow}^{scs}(sw)}\} \cup \{\alpha_s\}$ by \Cref{lem:inversions_after_s}.
Hence, it suffices to show that if $\beta \neq \alpha_{s}$, then we have that $s(\beta) \in \inv{\pi_{\downarrow}^{scs}(sw)}$.

We assume that $\beta \neq \alpha_s$. We have that $s(\beta)$ is the edge label of $swt \lessdot sw$ by \Cref{lem:inversions_after_s}.
We claim that $s(\beta)$ is an $scs$-stable edge label here.
Let $\Psi$ be a non-commutative rank-two root subsystem of $\Phi$ such that $s(\beta) \in \Psi^{+}$. 
There are two cases: either $\alpha_s \notin \Psi^{+}$ or $\alpha_s \in \Psi^+$.
\begin{itemize}
    \item If $\alpha_s \notin \Psi^{+}$, then $s(\Psi^{+})$ is the set of positive roots of the non-commutative rank-two subsystem $s(\Psi)$.
This is because $s$ is an isometry of the root system, and the only positive root it sends to a negative root is $\alpha_{s}$.
Let $s(\Psi^+) = \{s(\gamma_1), s(\gamma_2), s(\gamma_3)\}$ such that $\asef{s(\gamma_i)}{s(\gamma_j)} > 0$ for $i < j$.
By \cite[Lemma~3.8]{rs_infinite}, we have that $\asefa{scs}{s(\gamma_i)}{s(\gamma_j)} > 0$ for $i < j$.
Since
$wt \lessdot w$ is $c$-stable,
we have that $w$ is $c$-aligned with respect to $s(\Psi)$.
Hence, $sw$ is $scs$-aligned with respect to $\Psi$ by \Cref{lem:inversions_after_s}.

\item If
$\alpha_s \in \Psi^{+}$, then $s(\Psi) = \Psi$.
Since $\alpha_s$ is a simple root of $\Phi$, it must be a simple root of~$\Psi$.
First note that if $s(\beta)$ is a simple root of $\Psi$, then it is immediate that $sw$ is $scs$-aligned with respect to $\Psi$, since $s(\beta)$ labels the covering relation $swt \lessdot sw$ by \Cref{lem:inversions_after_s}.

Now assume that $s(\beta)$ is not a simple root of $\Psi$, so that we have $\Psi^{+} = \{\alpha_{s}, \alpha_{s} + \beta, \beta\}$ with $s(\beta) = \beta + \alpha_s$.
Because $\{\alpha_s, \beta\} \subseteq \inv{w}$, we must have that $\inv{w} \supseteq \Psi^{+}$ by \Cref{lem:inversion_biclosed}.
Hence, $\inv{sw} \cap \Psi^{+} = \{\beta, \alpha_s + \beta\}$ by \Cref{lem:inversions_after_s}.
Since $s$ is initial in $c$, we have that $\asef{\alpha_s}{\beta} > 0$ by \cite[Lemma~3.9]{rs_infinite}.
Then, by \cite[Lemma~3.8]{rs_infinite}, we have that $\asef{\alpha_s}{\beta} = \asefa{scs}{s(\alpha_s)}{s(\beta)} = \asefa{scs}{-\alpha_s}{\beta + \alpha_s} = -\asefa{scs}{\alpha_s}{\beta + \alpha_s} = -\asefa{scs}{\alpha_s}{\beta} = \asefa{scs}{\beta}{\alpha_s}$.
Hence, $\asefa{scs}{\beta}{\alpha_s} > 0$, and so $\mathbf{w}_0(scs)$ orders $\Psi^{+}$ as $(\beta, \alpha_s + \beta, \alpha_s)$.
Hence, $\inv{sw}$ contains an initial segment of this tuple with respect to the ordering given to it by $\mathbf{w}_0(scs)$.
Thus, $sw$ is $scs$-aligned with respect to $\Psi$.
\end{itemize} 
We thus have that $sw$ is $scs$-aligned with respect to any $\Psi$ containing $s(\beta)$.
Therefore, $swt \lessdot sw$ is $scs$-stable.
Since $\ell(sw) < \ell(w)$, by the induction hypothesis, we infer that $s(\beta) \in \inv{\pi_{\downarrow}^{scs}(sw)}$, which was what we wanted to show.

Now we consider the case where $w \not\geqslant s$.
In this case, we have $\alpha_s \notin \inv{w}$.
Since $s$ is initial in~$c$, we have by \Cref{lem:root-theoretic_proof} that $\beta$ cannot be supported on $\alpha_s$.
Hence, let $\Phi_{S \setminus s}$ be the corank-one root subsystem generated by all of the simple roots excluding $\alpha_s$.
We then have that $\beta \in \Phi_{S \setminus s}^+$.
We moreover have that $\inv{w_{S \setminus s}} = \inv{w} \cap \Phi_{S \setminus s}^+$; see \cite[p.708]{rs_infinite}.
It follows that $\beta \in \inv{w_{S \setminus s}}$.
Moreover, $\beta$ labels the covering relation $(wt)_{S \setminus s} \lessdot w_{S \setminus s}$ by this reasoning too.
Furthermore, $\beta$ is an $sc$-stable edge label of $(wt)_{S \setminus s} \lessdot w_{S \setminus s}$ since $\inv{w_{S \setminus s}} = \inv{w} \cap \Phi_{S \setminus s}^+$ and because $w$ is $c$-aligned with respect to all rank-two subsystems $\Psi$ containing $\beta$.
By the induction hypothesis, we then have that $\beta \in \inv{\pi_{\downarrow}^{sc}(w_{S \setminus s})}$, and thus we have $\beta \in \inv{\pi_{\downarrow}^{c}(w)} =  \inv{\pi_{\downarrow}^{sc}(w_{S \setminus s})}$, as desired.

By induction, we conclude that if $wt \lessdot w$ is $c$-stable, then $\beta \in \inv{\pi_{\downarrow}^{c}(w)}$.
\end{proof}

We can finally prove \Cref{thm:c-stable} after noting the following lemma.

\begin{lemma}\label{lem:contraction_criterion}
Let $\beta$ be the label of a covering relation $wt \lessdot w$.
We have that $wt \lessdot w$ is not contracted by the $c$-Cambrian congruence if and only if $\beta \in \inv{\pi_{\downarrow}^{c}(w)}$.
\end{lemma}
\begin{proof}
Suppose that $\beta \in \inv{\pi_{\downarrow}^{c}(w)}$.
We have that $\beta \notin \inv{wt} \supseteq \inv{\pi_{\downarrow}^{c}(wt)}$, so that $\inv{\pi_{\downarrow}^{c}(w)}$ strictly contains $\inv{\pi_{\downarrow}^{c}(wt)}$ and so $q_c(wt) \lessdot q_c(w)$.
Hence, if $\beta \in \inv{\pi_{\downarrow}^{c}(w)}$, then the covering relation $wt \lessdot w$ is not contracted.

Suppose conversely that $\beta \notin \inv{\pi_{\downarrow}^{c}(w)}$.
Then $\inv{\pi_{\downarrow}^{c}(w)} \subseteq \inv{wt}$.
This implies that $\pi_{\downarrow}^{c}(w) \leqslant wt$, and so we have $\pi_{\downarrow}^{c}(w) \leqslant \pi_{\downarrow}^{c}(wt)$, since $\pi_{\downarrow}^{c}(wt)$ is defined to be the maximal $c$-sortable element less than~$wt$.
But then $\pi_{\downarrow}^{c}(w) = \pi_{\downarrow}^{c}(wt)$ since $\pi_{\downarrow}^{c}$ is order-preserving.
\end{proof}

\begin{proof}[{Proof of \Cref{thm:c-stable}}]
Let $\beta$ be the label of a covering relation $wt \lessdot w$, where $\ell(wt) < \ell(w)$.
We have that $wt \lessdot w$ is not contracted by the $c$-Cambrian congruence if and only if $\beta \in \inv{\pi_{\downarrow}^{c}(w)}$ by \Cref{lem:contraction_criterion}.
Hence, by \Cref{lem:cstable} and \Cref{lem:difficult direction}, we have that this is the case if and only if 
$wt \lessdot w$ is $c$-stable.
\end{proof}

From the definition of a $c$-stable covering relation, we can now define a $c$-stable root in a root sequence labelling the covering relations of a maximal chain.

\begin{definition}
Let $C \in \mg{W}$ be a maximal chain given by $e = u_1 \lessdot \ u_2 \lessdot \ldots \lessdot u_m = w_0$.
We know that every positive root $\beta$ is the label of an exactly one covering relation $u_i \lessdot u_{i+1}$ in~$C$. 
We say that that $\beta$ is \emph{$c$-stable in $C$} if the relation $u_i \lessdot u_{i+1}$ is $c$-stable.
\end{definition}

We obtain the following nice characterisation of being $c$-stable in a maximal chain.

\begin{proposition}
Let $W$ be a finite Coxeter group and let $C$ be a maximal chain of $W$.
Then a positive root $\beta$ is $c$-stable in $C$ if and only if for every rank-two subsystem $\Psi$ with $\beta$ a non-simple root of $\Psi$, we have that $\Psi \notin \inva{\mathbf{w}_0(c)}{C}$.
\end{proposition}
\begin{proof}
First suppose that $\beta \in \Phi^{+}$ is $c$-stable in $C$ and let $\Psi$ be a rank-two subsystem with $\beta$ a non-simple root of $\Psi$.
Note that $\Psi$ is necessarily non-commutative, since all of the positive roots of a commutative rank-two subsystem are simple.
Since the covering relation $wt \lessdot w$ of $C$ that $\beta$ labels is $c$-stable, we must have that $w$ is $c$-aligned with respect to $\Psi$.
Therefore, $\inv{w}$ contains an initial segment of $\Psi^+$ with respect to the ordering given by $\heap{\mathbf{w}_0(c)}$, since $\beta$ is not a simple root of~$\Psi^{+}$.
Consequently, $C$ must order $\Psi^{+}$ in the same way that $\heap{\mathbf{w}_0(c)}$ does, and so $\Psi \notin \inva{\mathbf{w}_0(c)}{C}$.

Conversely, suppose that for every rank-two subsystem $\Psi$ with $\beta$ a non-simple root of $\Psi$, we have that $\Psi \notin \inva{\mathbf{w}_0(c)}{C}$.
Let $wt \lessdot w$ be the covering relation of $C$ labelled by $\beta$ and let $\Psi$ be a rank-two root subsystem with $\beta \in \Psi^+$.
If $\beta$ is a simple root of $\Psi^{+}$, then $w$ is automatically $c$-aligned with respect to $\Psi^{+}$.
If $\beta$ is not a simple root of $\Psi^{+}$, then since $\Psi \notin \inva{\mathbf{w}_0(c)}{C}$, $w$ must also be $c$-aligned with respect to $\Psi$.
We conclude that $wt \lessdot w$ is $c$-stable, as desired.
\end{proof}

We then have the following corollaries of \Cref{thm:c-stable}.

\begin{corollary}\label{cor:chain_stability}
Let $W$ be a finite simply-laced Coxeter group with $c$ a Coxeter element, $q_c \colon W \to W_c$ the associated $c$-Cambrian quotient, and $C$ a maximal chain in the weak Bruhat order on $W$.
Then a covering relation of $C$ labelled by a root $\beta$ is not contracted by $q_c$ if and only if $\beta$ is $c$-stable in~$C$.
Hence, the maximal chain $\mgmap{q_c}(C)$ is labelled by the sequence of $c$-stable edge labels of~$C$.
\end{corollary}

Let $\overset\longrightarrow\Stab(C)$ be the sequence of $c$-stable roots of $C$, ordered in the way they are added to inversion sets along $C$, and let $\Stab(C)$ its underlying unordered set.

\begin{corollary}\label{cor:stability}
For a pair of chains $C, C' \in \mg{W}$, we have $\mgmap{q_c}(C) = \mgmap{q_c}(C')$ if and only if  $\overset\longrightarrow\Stab(C) = \overset\longrightarrow\Stab(C)$.
\end{corollary}
\begin{proof}
For $C \in \mg{W}$ given by $e = u_1 \lessdot u_2 \lessdot \ldots \lessdot u_m = w_0$, the inversion sets in $\mgmap{q_c}(C)$ ordered by inclusion consist of $\inv{\pi^{c}_{\downarrow}(u_i)}$.
By \Cref{thm:c-stable}, we have that 
the edge $u_i \lessdot u_{i + 1}$ is not contracted by $\theta_c$
if and only if
the relation $u_i \lessdot u_{i+1}$
is $c$-stable.
Hence, applying the quotient edge labelling $\lambda_{\theta_c}$ to $\mgmap{q_c}(C)$ gives the sequence of $c$-stable roots of $C$ precisely in the same order they appear in $\overset\longrightarrow\Stab(C)$.
Conversely, given the sequence $\overset\longrightarrow\Stab(C)$, one can reconstruct the maximal chain $\mgmap{q_c}(C)$ using the edge labelling $\lambda_c$.
Thus, $\mgmap{q_c}(C)$ and $\overset\longrightarrow\Stab(C)$ determine each other uniquely.
\end{proof}

We note that in type $A$ with linear orientation, i.e., for $c = s_1 s_2 \cdots s_{n-1}$ in $S_n$, $\overset\longrightarrow\Stab(C) = \overset\longrightarrow\Stab(C')$ if and only if the underlying unordered sets of $c$-stable roots $\Stab(C)$ and $\Stab(C')$ coincide.
This is not true in general, a counterexample existing already in type $A_3$ with bipartite orientation, i.e., in $S_4$ with $c = s_2 s_1 s_3$.
In the next section, we will interpret posets of maximal chains $ \mgea{\lambda_c}{W_c}$  via maximal green sequences of finite-dimensional algebras.
The latter can be described by a certain sequence of bricks, but we showed in our previous work \cite[Example~3.25]{gw1} that it is not always defined uniquely by the underlying set of this sequence, although it is defined uniquely in the case of Nakayama algebras, including the path algebra of the linearly oriented $A_n$ quiver.

\subsection{Contraction of posets}

We can also apply the results of this section to prove that in the Cambrian case, the preorders on maximal chains are in fact partial orders.

\begin{proposition}\label{prop:cambrian_partial_order}
Let $W$ be a finite Coxeter group, with $c$ a Coxeter element, so that we have the quotient edge-labelling $\lambda_{\theta_c}^{\mathbf{w}_0(c)} \colon \covrel{W_c} \to \heap{\mathbf{w}_0(c)}$ induced by the Cambrian congruence.
Then the preorder $\mgea{\lambda_{\theta_c}^{\mathbf{w}_{0}(c)}}{W_c}$ is a partial order.
\end{proposition}

For this, we first need the following result.

\begin{lemma}\label{lem:length_two_max_chain}
If $P$ is a polygon in $W_c$, then $P$ has a maximal chain of length two.
Moreover, if $P$ is a non-square polygon, then this maximal chain is descending with respect to the edge labelling $\lambda_c$.
\end{lemma}

\begin{proof}
Suppose that $P$ is a polygon in $W_c$.
The result is clear if $P$ is square, so we may assume that $P$ is not square.
By \Cref{lem:polygon_preimages}, there exists a polygon $P'$ of $W$ such that $q_c(P') = P$.
We denote the elements of $P'$ as follows: \[
\begin{tikzpicture}[xscale=1.5]

\node (min) at (0,0) {$\check{p}.$};

\node (l1) at (-1,1) {$l_1$};
\node (l2) at (-1,2) {$l_2$};
\node (l3) at (-1,3) {$\vdots$};
\node (l4) at (-1,4) {$l_{m}$};

\node (r1) at (1,1) {$r_1$};
\node (r2) at (1,2) {$r_2$};
\node (r3) at (1,3) {$\vdots$};
\node (r4) at (1,4) {$r_{n}$};

\node (max) at (0,5) {$\hat{p}$};

\draw[->] (min) -- (l1);
\draw[->] (l1) -- (l2);
\draw[->] (l2) -- (l3);
\draw[->] (l3) -- (l4);
\draw[->] (l4) -- (max);

\draw[->] (min) -- (r1);
\draw[->] (r1) -- (r2);
\draw[->] (r2) -- (r3);
\draw[->] (r3) -- (r4);
\draw[->] (r4) -- (max);

\node at (-1.5,2.5) {$C_1$};
\node at (1.5,2.5) {$C_2$};

\end{tikzpicture}
\]

By \Cref{thm:coxeter}\ref{op:coxeter:polygons}, there is a rank-two subsystem $\Psi$ of $\Phi$ such that for each maximal chain $C_1$ and $C_2$ of $P'$, we have $\lambda_c(C_i) = \Psi^+$.
Since $C_1$ and $C_2$ give the two opposite total orders of $\Psi^{+}$, one of these orders will agree with the induced order from $\heap{\mathbf{w}_0(c)}$, and the other will not.
Suppose without loss of generality that $C_1$ agrees with $\heap{\mathbf{w}_0(c)}$ and $C_2$ does not. 
Then, all roots labelling the covering relations of $C_2$ must be non-$c$-stable in $C_2$, apart from perhaps the first and last one. Hence, by \Cref{lem:cstable}, we have that the image $q_c(C_2)$ of $C_2$ in $W_c$ has length at most two. Since $P$ is a polygon, we must have that $q_c(C_2)$ actually has length exactly two. Moreover, since $C_2$ ordered $\Phi^{+}$ the opposite way to $\heap{\mathbf{w}_0(c)}$, we must have that $q_c(C_2)$ is descending with respect to the labelling~$\lambda_c$.
\end{proof}

The proof of \Cref{prop:cambrian_partial_order} is now straightforward.

\begin{proof}[{Proof of \Cref{prop:cambrian_partial_order}}]
As in the proof of \Cref{lem:hyp_partial_order}, it suffices to show that, given a maximal chain $C$ of $W_c$, we cannot return to $C$ using increasing polygon moves and square moves, with at least one increasing polygon move.
But this is clear, since, by \Cref{lem:length_two_max_chain}, we have that an increasing polygon move in $W_c$ strictly decreases the number of covering relations in the maximal chain.
\end{proof}

We then have the following result.

\begin{theorem} \label{thm:cambrian_dim2_contraction}
Let $W$ be a finite Coxeter group, with $c$ a Coxeter element and $q_c \colon W \to W_c$ the Cambrian quotient.
Then the induced map \[
\mgtmap{q_c} \colon \mgea{\lambda^{\mathbf{w}_0(c)}}{W} \to \mgea{\lambda^{\mathbf{w}_0(c)}_{\theta_c}}{W_c}
\]
is a contraction of posets.
\end{theorem}
\begin{proof}
This follows from \Cref{thm:coxeter}\ref{op:coxeter:contraction} using \cite[Lemma~5.3]{cebrian2022directed}, which states that applying the collapse functor to a contraction of preordered sets gives a contraction of posets, and from the fact that applying the collapse functor to a poset returns the same poset.
\end{proof}

\section{Maximal green sequences of preprojective and hereditary algebras}\label{sect:preproj_mgs_orders}

In this section, we consider the algebraic interpretation of Cambrian congruences in terms of representation theory of finite-dimensional algebras.
This algebraic interpretation was carried out in works including \cite{it,airt,birs,mizuno-preproj,dirrt,mizuno2020torsion}. 
The partial orders on equivalence classes of maximal chains in this context relate back to the work on maximal green sequences in \cite{gw1}.

\subsection{Background}

We write $\Lambda$ for a finite-dimensional algebra over a field $K$ and write $\modules \Lambda$ for the category of finite-dimensional left $\Lambda$-modules.

\subsubsection{Torsion classes and maximal green sequences}

In order to interpret Cambrian congruences algebraically, we need the following concept from representation theory.
Given a finite-dimensional algebra $\Lambda$ over a field $K$, a \emph{torsion-free class} is a full subcategory of $\modules \Lambda$ which is closed under submodules and extensions \cite[Theorem~2.3]{dickson}.
Every torsion class fits into a unique \emph{torsion pair} of subcategories, the other partner in the pair being a \emph{torsion class}.
Torsion-free classes in $\modules \Lambda$ form a complete lattice under inclusion, denoted $\torf \Lambda$ \cite[Proposition~2.3(a)]{irtt}.
Torsion classes form an anti-isomorphic lattice $\tors \Lambda$, with the canonical anti-isomorphisms interchanging the torsion class with the torsion-free class in each torsion pair.

\subsubsection{Algebraic interpretation of Cambrian lattices and congruences}\label{sect:preproj_mgs_orders:alg_back}

We refer to \cite[Section~7]{aoki2022fans} for detailed presentation on tensor algebras and preprojective algebras of $K$-species, see also \cite[Section~4.2]{dst}, \cite{Kul}.
Let $(W, S)$ be a finite crystallographic Coxeter system with associated Coxeter diagram $\Delta$.
We call a tuple $(D_s, {}_{s}M_{s'})_{s, s' \in S}$ a \emph{$K$-species} if $D_i$ is a finite-dimensional division $K$-algebra and ${}_{s}M_{s'}$ is a finitely generated $D_s \otimes_{K} D_{s'}^{\mathrm{op}}$-module.
We refer to \cite[Definition~7.2(b)]{aoki2022fans} for the definition of the Cartan matrix $A$ of a $K$-species $(D_s, {}_{s}{M}_{s'})_{s, s' \in S}$; in particular, if $a_{ss'} = 0$, then ${}_{s}{M}_{s'} = {}_{s'}M_{s} = 0$.

We now fix a symmetrisable crystallographic Cartan matrix $A$ for $(W, S)$ and choose a $K$-species $(D_s, {}_{s}{M}_{s'})_{s, s' \in S}$ with Cartan matrix $A$.
Given a Coxeter element $c \in W$, we say that the species $(D_s, {}_{s}M_{s'})_{s, s' \in S}$ is \emph{oriented according to $c$} if ${}_{s}M_{s'} \neq 0$ 
if and only if $a_{ss'} \neq 0$ and $s$ precedes $s'$ in reduced expressions of~$c$.
Coxeter elements correspond to orientations of the Coxeter diagram and $K$-species of the same orientation have non-zero bimodules for arrows in the forwards direction.
By replacing a bimodule by its opposite, one can reverse the direction of the arrows and swap the order of generators in reduced expressions for the Coxeter element.

To this data, one can associate a \emph{preprojective algebra} $\Pi$ \cite[Definition~7.3]{aoki2022fans}, \cite{gp_preproj,dr_preproj}, which is a finite-dimensional $K$-algebra.
It is invariant under changing the Coxeter element $c$ by replacing bimodules by their opposites.
We obtain further a hereditary tensor algebra $\Lambda_c$ \cite[Section~7.1]{aoki2022fans}, which does depend on the Coxeter element~$c$.
When $(W, S)$ is simply-laced, $c$ determines a quiver $Q$ as an orientation of $\Delta$, and $\Lambda_c$ is isomorphic to the path algebra $KQ$.

\begin{fact}
\label{thm:tors=weak}
Let $(W, S)$ be a finite crystallographic Coxeter system with $c$ a Coxeter element and let $\Pi$ and $\Lambda_c$ be as above.
\begin{enumerate}[label=\textup{(}\arabic*\textup{)}]
    \item \cite[Theorem~2.30]{mizuno-preproj}, \cite[Theorem~7.9]{aoki2022fans}
    There is an order isomorphism $\phi\colon W \to \torf \Pi$ between the weak Bruhat order on $W$ and $\torf \Pi$.\label{op:tors=weak:pi_bij}
    \item \cite[Theorem~4.3, Section~4.4]{it}  
    There is an isomorphism of posets $
    \phi_c \colon W_c \to 
    \torf \Lambda_c$ between the $c$-Cambrian lattice $W_c$ and $\torf \Lambda_c$.\label{op:tors=weak:lambda_bij}
    
    \item \cite[Theorem~1.3]{mizuno2020torsion} \cite[Theorem~7.2]{dirrt} 
    Assume that $W$ is simply-laced.
    There is an ideal $I$ of $\Pi$ such that $\Lambda_c \cong \Pi/I$.
    The maps $\tors \Pi \to \tors \Lambda_c$, $\mathcal{T} \mapsto \mathcal{T} \cap \modules \Lambda_c$ and  $\torf \Pi \to \torf \Lambda_c$, $\mathcal{F} \mapsto \mathcal{F} \cap \modules \Lambda_c$
    are then lattice quotients fitting into a commutative diagram \[
    \begin{tikzcd}
        W \ar[r,"\phi"] \ar[d,"{q_c}"] 
        & \torf \Pi \ar[d] \\
        W_c \ar[r,"{\phi_c}"] 
        & \torf \Lambda_c ,
    \end{tikzcd}
    \]
    with horizontal maps being the above isomorphisms.\label{op:tors=weak:commutativity}
\end{enumerate}
\end{fact}

The statement of \ref{op:tors=weak:pi_bij} in \cite{mizuno-preproj,aoki2022fans} is in terms of support $\tau$-tilting modules, but this can be translated into torsion-free classes by \cite{air,dij}.
We note that in \cite{mizuno2020torsion} part \ref{op:tors=weak:commutativity} is proved more generally for infinite (but still crystallographic and simply-laced)~$W$.

We were not able to find in the literature the analogue of part \ref{op:tors=weak:commutativity} for $W$ non-simply-laced. While it is plausible that the arguments of \cite{mizuno2020torsion} or \cite{dirrt} might generalise with some care, we present in \Cref{subseq:mgs_results} a different argument using our results about contractions of posets. We prove the statement of \ref{op:tors=weak:commutativity} for $W$ possibly non-simply-laced in \Cref{prop:not_simply_laced}

\begin{remark}
As ever, there are many possible choices of different conventions.
Our conventions are based on those of \cite{enomoto_bruhat}.
Note that, in particular, \cite{mizuno-preproj} uses the left weak Bruhat order and right modules, whereas we use the right weak Bruhat order and left modules.
\cite{it} uses torsion classes rather than torsion-free classes, because they are using quiver representations, which are equivalent to right modules.
\end{remark}

\begin{remark}
Aside from the framework of \cite{gp_preproj,dr_preproj}, there is another framework from Geiss, Leclerc, and Schroer \cite{gls_symm_i} where the hereditary algebras $\Lambda_c$ are replaced by $1$-Iwanaga--Gorenstein algebras.
Several parts of \Cref{thm:tors=weak} also hold in this framework.
Indeed, the equivalent of \ref{op:tors=weak:pi_bij} is shown in \cite{fg,murakami}, and the equivalent of \ref{op:tors=weak:lambda_bij} is shown in \cite{gyoda}.
\end{remark}

\subsubsection{Algebraic interpretation of edge labelling}\label{sect:mgs:back:alg_edge_label}

In this subsection, we explain how the edge labelling of $W$ from Subsection~\ref{sect:cox:edge_label} relates to the algebraic interpretation of $W$ we have been considering.
Recall that a module $B \in \modules \Lambda$ is a \emph{brick} if its endomorphism algebra $\End_{\Lambda}B$ is a division ring.

\begin{fact}[{\cite[Theorem~2.2.6]{bcz}, \cite[Theorem~3.3(b)]{dirrt}}]\label{thm:brick_labels}
We have that $\mathcal{F} \supseteq \mathcal{U}$ is a covering relation in $\torf \Lambda$ if and only if $\mathcal{T} \cap \hol{\mathcal{U}} = \Filt (B)$ for a brick $B$.
Moreover, this brick $B$ is unique up to isomorphism.
\end{fact}

Here $\Filt$ and $\hor{\mathcal{U}}$ are defined as follows. We have
\[\hor{\mathcal{U}} := \{M \in \modules \Lambda \st \Hom_{\Lambda}(U, M) = 0 \text{ for all } U \in \mathcal{U}\}.\]
Given a full subcategory $\mathcal{C}$ of $\modules \Lambda$, $\Filt(\mathcal{C})$ is the full subcategory of $\modules \Lambda$ consisting of modules $M$ with a finite filtration \[M = M_0 \supset M_1 \supset \dots \supset M_{l - 1} \supset M_l = 0\] such that $M_{i - 1}/M_{i} \in \additive \mathcal{C}$ for all $1 \leqslant i \leqslant l$.
Here $\additive \mathcal{C}$ is the full subcategory of $\modules \Lambda$ consisting of direct summands of finite direct sums of objects in $\mathcal{C}$.

Hence, we have a set-valued edge labelling $\lambda_{\mathsf{brick}} \colon \covrel{\torf \Lambda} \to \brick \Lambda$ for any finite-dimensional algebra $\Lambda$ given by sending $\mathcal{U} \lessdot \mathcal{T}$ to the unique brick $B$ such that $\mathcal{T} \cap \hor{\mathcal{U}} = \Filt (B)$.
This is called the \emph{brick labelling}.
By \cite[Theorem~3.11]{dirrt}, the brick labelling is forcing-consistent.

Consider now the brick labelling of $\torf \Pi$. 
We have a commutative diagram \[
\begin{tikzcd}
    \covrel{\torf \Pi} \ar[r,"{\phi^{-1}}"] \ar[d,"{\lambda_{\mathsf{brick}}}"] & \covrel{W} \ar[d,"\lambda"] \\
    \brick \Pi \ar[r,"\dimu"] & \Phi^{+}.
\end{tikzcd}
\]
This is shown in \cite[Proposition~4.3]{enomoto_bruhat} for simply-laced types and the arguments generalise to non-simply-laced types using \cite[Remark~1.7, Proposition~4.5]{dst}.
Here $\dimu$ denotes the dimension vector of a species representation \cite[p.391]{dlab-ring}, which generalises that of a quiver representation \cite[Definition~III.3.1]{ass}.
Dimension vectors $\dimu B$ of indecomposable $\Pi$-modules $B$ give positive roots $\dimu B = \beta \in \Phi^{+}$ by interpreting $\dimu B$ as the vector of coefficients of simple roots in the representation of $\beta$ as a positive sum of simple roots.
It is worth also recalling Gabriel's Theorem at this juncture, which gives that $B \mapsto \dimu B$ gives a bijection between indecomposable $\Lambda_c$-modules and positive roots \cite{bgp,gabriel,dlab-ring}.

\subsubsection{Maximal green sequences}

A \emph{maximal green sequence} is a finite maximal chain in $\torf \Lambda$.
Hence, when $\torf \Lambda$ is finite, as it will be for the cases we consider in this paper, maximal green sequences will be precisely maximal chains in $\torf \Lambda$.
Maximal green sequences are usually defined to be finite maximal chains in $\tors \Lambda$, but our definition is equivalent due to the canonical anti-isomorphisms $\torf \Lambda \longleftrightarrow \tors \Lambda$.
These anti-isomorphisms also preserve edge labels, so working with torsion-free classes is entirely equivalent to working with torsion classes.
We denote the set of maximal green sequences of $\Lambda$ by~$\mgs{\Lambda}$.
In this subsection, we explain the partial orders on equivalence classes of maximal green sequences that were introduced in \cite{gw1}, where more extensive background on maximal green sequences can be found.
Maximal green sequences were originally introduced by Keller using a combinatorial definition in terms of quivers \cite{kel-green}, although they were already implicit in the physics literature \cite{ccv}, and were later interpreted in terms of torsion classes by Nagao \cite{nagao}.

The brick labelling gives us the following notion of maximal green sequence, which will be useful in this paper.

\begin{definition}
A sequence of bricks \[B_1, B_2, \dots, B_r\] is then called \emph{backwards $\Hom$-orthogonal} if $\Hom_{\Lambda}(B_j, B_i) = 0$ for $i < j$.
A \emph{maximal} backwards $\Hom$-orthogonal sequence of bricks is one which cannot be extended to a longer such sequence.
\end{definition}

It is shown in \cite{dem-kel} that maximal green sequences are equivalent to finite maximal backwards $\Hom$-orthogonal sequences of bricks by sending a maximal green sequence to its sequence of brick labels, generalising \cite{igusa_mgs}.

Reduced words for the longest element $w_{0}$ of $W$ are equivalent to maximal chains in the weak Bruhat order on $W$, and hence are equivalent to maximal green sequences of $\Pi$ by \Cref{thm:tors=weak}\ref{op:tors=weak:pi_bij}.
For $W$ simply-laced, Enomoto \cite[Proposition~4.3]{enomoto_bruhat} shows that if $B_1, B_2, \dots, B_r$ is a maximal green sequence of $\Pi$, then $\dimu B_1, \dimu B_2, \dots, \dimu B_r$ is a root sequence for $w_0$ (see also \cite{airt}); his arguments generalise to 
the non-simply-laced types by using \cite[Remark~1.7, Proposition~4.5]{dst}.

As shown in \cite{gw1}, the set of maximal green sequences of $\Lambda$ forms a poset when subject to the square-equivalence relation.

\begin{definition}\label{def:mgs_poset}
Two maximal green sequences $\mathcal{G}$ and $\mathcal{G}$ are \emph{equivalent} if they are equivalent as maximal chains of $\torf \Lambda$, that is, related by a sequence of square moves.
In \cite{gw1}, several other interpretations of this equivalence relation were given.

The partial order from \cite{gw1} is defined as follows.
We say that a maximal green sequence $\mathcal{G}$ is an \emph{increasing elementary polygonal deformation} 
of a maximal green sequence $\mathcal{G}'$ if $\mathcal{G}'$ is \[B_1, B_2, \dots, B_r,\] while $\mathcal{G}$ is \[B_1, \dots, B_{i - 1}, B_{i + 1}, B'_1, B'_2, \dots, B'_l, B_i, \dots, B_{i + 2}, \dots, B_r,\] where $l \geqslant 1$.
Note that an increasing elementary polygonal deformation decreases the length of a maximal green sequence.
By extension, we also say that $[\mathcal{G}]$ is an increasing elementary polygonal deformation of $[\mathcal{G}']$ here.
The partial order is then defined by having covering relations $[\mathcal{G}] \lessdot [\mathcal{G}']$.
Two other partial orders were defined in \cite{gw1}, conjecturally equivalent to this one.
We denote the poset of equivalence classes of maximal green sequences of $\Lambda$ by~$\mgse{\Lambda}$.
\end{definition}

In \Cref{prop:mgs_poset_iso}, we will relate increasing polygonal deformations with the increasing polygon moves of \Cref{def:increasing_polygon_move}.

\begin{remark}
We choose to work within the framework of \cite{gp_preproj,dr_preproj} rather than that of \cite{gls_symm_i}. Some relevant results established in the former setting (in the Dynkin case) use the fact that the classes of bricks and indecomposable $\tau$-rigid modules coincide. This is not the case in the latter setting, where the distinction between algebraic realisations of roots and scaled coroots is more apparent, see \cite{geiss2020rigid}. 
\end{remark}

\subsection{Results}
\label{subseq:mgs_results}

We now explain the applications of the results of the previous sections to the partial orders on the equivalence classes of maximal green sequences of $\Pi$ and $\Lambda_c$ from \Cref{def:mgs_poset}.

\subsubsection{Posets of maximal green sequences}

We first show that the partial order on the equivalence classes of maximal green sequences of $\Lambda_c$ given by the edge labelling $\lambda_c$ coincides with the partial order defined in \Cref{def:mgs_poset}.

\begin{proposition}\label{prop:mgs_poset_iso}
The isomorphism of posets $\torf \Lambda_c \cong W_c$ induces an isomorphism of posets $\mgse{\Lambda_c} \cong \mgea{\lambda_c}{W_c}$.
\end{proposition}

\begin{proof}
Since $\torf \Lambda_c \cong W_c$ by \Cref{thm:tors=weak}\ref{op:tors=weak:lambda_bij}, it follows that there is a bijection between their sets of square-equivalence classes of maximal chains, $\mge{\torf \Lambda_c} \cong \mgse{\Lambda_c}$ and $\mge{W_c}$.
Hence, all we need to do is show that the relations on these sets coincide.
The relations on the former are generated by increasing polygon moves and the relations on the latter are generated by increasing elementary polygonal deformations.
Hence, it suffices to show that increasing elementary polygonal deformations coincide with increasing polygon moves for maximal green sequences of $\Lambda_c$.

We first show that every increasing polygon move in $\mgea{\lambda_c}{W_c}$ gives an increasing elementary polygonal deformation in $\mgse{\Lambda_c}$.
An increasing polygon move is given by replacing a chain in a non-square polygon $P$ which is ascending with respect to $\heap{\mathbf{w}_0(c)}$ by the other chain in the polygon, which is in fact descending.
By \Cref{lem:length_two_max_chain}, we have that the descending chain of $P$ has length two, and so by \cite[Lemma~4.2]{gw1}, we have that this gives an increasing polygonal deformation in $\mgse{\Lambda_c}$.

By running this logic in reverse, we see that every increasing elementary polygonal deformation in $\mgse{\Lambda_c}$ gives an increasing polygon move in $\mgea{\lambda_c}{W_c}$.
Indeed, increasing elementary polygonal deformations in $\mgse{\Lambda_c}$ are given by replacing a chain of a non-square polygon of length $\geqslant 2$ with a chain of length two.
By \Cref{lem:length_two_max_chain} that the chain of length two must be the descending chain with respect to the edge labelling, and so we have an increasing polygon move in $\mgea{\lambda_c}{W_c}$.
\end{proof}

\begin{corollary} \label{cor:minima_and_maxima}
The poset $\mgea{\lambda_c}{W_c}$ has a unique minimum given by $\mathbf{w}_0(c)$. It also has a maximum given by $\mathbf{w}_0(c^{-1})$. If $W$ is of type $A$ and $c$ induces the linear orientation, then this maximum is unique.
\end{corollary}

\begin{proof}
This is a translation of \cite[Proposition 4.20, Proposition 4.21]{gw1} for general $\Lambda_c$, and further of \cite[Proposition 4.20, Corollary 5.17]{gw1} in linearly oriented type $A$,  via the isomorphism from \Cref{prop:mgs_poset_iso}.
\end{proof}

\begin{remark}
\label{rem:unique_maxima}
In fact, we conjecture in \cite[Conjecture 4.11]{gw1} that the partial order on $\mgse{\Lambda}$ that we consider in the present paper, for an arbitrary finite-dimensional algebra $\Lambda$, coincides with two other partial orders on the same underlying set. For those orders, in the generality covering $\Lambda = \Lambda_c$, \cite[Proposition 4.20]{gw1} implies the existence of a unique maximum. To summarize, validity \cite[Conjecture 4.11]{gw1} would imply that $\mathbf{w}_0(c^{-1})$ is the unique maximum of 
$\mgea{\lambda_c}{W_c}$ for an arbitrary finite crystallographic $W$
and an arbitrary $c$.
\end{remark}

The following proposition shows that equivalence classes of maximal green sequences for preprojective algebras of 
Dynkin type are determined by their bricks. 
We know from \cite[Example~3.26]{gw1} that this does not hold for general algebras.
Indeed, \cite[Example~3.26]{gw1} gives a counter-example for an algebra $\Lambda_c$ for $c$ a Coxeter element in type $A_3$.
Given a maximal green sequence $\mathcal{G}$ of a finite-dimensional algebra $\Lambda$, we write $\bricks{\mathcal{G}}$ for the set of brick labels of $\mathcal{G}$ as a maximal chain of $\torf \Lambda$.
Equivalently, $\bricks{\mathcal{G}}$ is the set of bricks of $\mathcal{G}$ when realised as a maximal backwards $\Hom$-orthogonal sequence of bricks.

\begin{proposition}\label{prop:preproj_bricks}
Let $\mathcal{G}$ and $\mathcal{G}'$ be two maximal green sequences of $\Pi$. If $\bricks{\mathcal{G}} = \bricks{\mathcal{G}'}$, then $\mathcal{G}$ and $\mathcal{G}'$ are equivalent.
\end{proposition}
\begin{proof}
We have that $\mathcal{G}$ gives a total order $<_{\mathcal{G}}$ on the set of positive roots $\Phi^{+}$ by taking $\dimu$ of the corresponding maximal backwards $\Hom$-orthogonal sequence of bricks, and we similarly denote the total order on $\Phi^{+}$ from $\mathcal{G}'$ by~$<_{\mathcal{G}'}$.
We claim that $<_{\mathcal{G}'}$ orders every non-commutative rank-two subsystem in the same direction as in $<_{\mathcal{G}}$.
Suppose, to the contrary, that there is a non-commutative rank-two subsystem $\Psi$ of $\Phi$ such that $\Psi^{+}$ is ordered differently by $<_{\mathcal{G}}$ and $<_{\mathcal{G}'}$.
Let $\alpha$ and $\gamma$ be the simple roots of $\Psi$, and let $\beta$ be a non-simple positive root of $\Psi$ such that $\sbf{\alpha}{\beta} > 0$.
It can be verified case-by-case that all non-commutative rank-two subsystems have such a root~$\beta$.
Since $\mathcal{G}$ and $\mathcal{G}'$ have the same bricks, the roots $\alpha, \beta, \gamma$ must correspond to the same bricks $L, M, N$ in $\mathcal{G}$ and $\mathcal{G}'$.
Suppose for contradiction that this triple is ordered $\alpha <_{\mathcal{G}} \beta <_{\mathcal{G}} \gamma$ and $\gamma <_{\mathcal{G}'} \beta <_{\mathcal{G}'} \alpha$.
We now use an argument similar to that of \cite[Theorem~4.8]{enomoto_bruhat}.
By \cite[Lemma~1]{cb} and \cite[Lemma~4.7]{enomoto_bruhat}
in the simply-laced case, \cite[Theorem~4.2]{dst} in the non-simply-laced case,
we then have that $0 < \sbf{\dimu M}{\dimu L} = \dim \Hom_{\Pi}(M, L) + \dim \Hom_{\Pi}(L, M) - \dim \Ext_{\Pi}^{1}(M, L)$. 
Thus, either $\Hom_{\Pi}(M, L) \neq 0$ or $\Hom_{\Pi}(L, M) \neq 0$.
However, the former contradicts the bricks' being ordered $L <_{\mathcal{G}} M <_{\mathcal{G}} N$, and the latter contradicts the bricks' being ordered $N <_{\mathcal{G}'} M <_{\mathcal{G}'} L$.
Therefore, $<_{\mathcal{G}}$ and $_{\mathcal{G}'}$ order all packets in the same direction as each other.
Hence, $\mathcal{G}$ and $\mathcal{G}'$ are equivalent by \Cref{prop:inv<=>equiv}.
\end{proof}

\begin{lemma}\label{lem:min_mgs}
There exists a unique equivalence class of maximal green sequences $[\mathcal{G}]$ of $\Pi$ such that the bricks in $\bricks{\mathcal{G}}$ are all $\Lambda_c$-modules: the class corresponding to the $c$-sorting word $\mathbf{w}_0(c)$.
\end{lemma}
\begin{proof}
The bricks of $\Lambda$ form a backwards Hom-orthogonal sequence under any total order which extends the order given by the Auslander--Reiten quiver of $\Lambda_c$.
This backwards Hom-orthogonal sequence then must form a maximal green sequence $\mathcal{G}$ for $\Pi$, since it contains exactly one brick for each root.
Since this maximal green sequence is defined by its bricks, it must be unique up to equivalence, by \Cref{prop:preproj_bricks}.
The fact that this corresponds to the equivalence class of $\mathbf{w}_0(c)$ follows from \cite[Theorem~9.3.1]{stump2015cataland}.
Indeed, the Auslander--Reiten quiver of $\Lambda_c$ coincides with $\heap{\mathbf{w}_0(c)}$.
\end{proof}

The equivalence relation on maximal green sequences of $\Pi$ thus behaves nicely.
The partial order from \cite{gw1} however has no relations.

\begin{proposition}
The poset $\mge{\Pi}$ has no relations.
\end{proposition}
\begin{proof}
We have that $\torf \Pi \cong W$ by \Cref{thm:tors=weak}.
By \Cref{thm:coxeter}\ref{op:coxeter:polygons}, we have that every polygon $P$ of $W$ corresponds to a rank-two subsystem $\Psi_P$ of $\Phi$.
This implies that the two maximal chains of $P$ have the same length.
If this length is greater than two, then $P$ does not give a relation in $\mge{\Pi}$.
If this length is two, then $P$ gives an equivalence, and so does not give a relation either.
\end{proof}

The constructions of Sections~\ref{sect:max_chain_posets} and~\ref{sect:cox:edge_label} show however that there is a non-trivial partial order on $\mge{\Pi}$ for every Coxeter element or, equivalently, every orientation of the Coxeter diagram~$\Delta$.
This explains where the partial order from \cite{gw1} should be trivial on $\mge{\Pi}$: otherwise there would be a single distinguished partial order.
The correct perspective is that there are several equally valid partial orders determined by the different Coxeter elements, each one refining the trivial partial order.
Indeed, given a Coxeter element $c \in W$, when we put the order determined by $\lambda_c$ on $\mgse{\Pi} \cong \mge{\torf \Pi}$, we obtain a contraction of posets $\mgea{\lambda_c}{\torf \Pi} \to \mgse{\Lambda}$.

\begin{corollary}\label{cor:mgs_contraction}
Let $c$ be a Coxeter element of $W$ and $\Lambda_c$ be the corresponding quotient of $\Pi$.
Then there is a contraction of posets $\mgtmap{q_c}\colon \mgea{\lambda_c}{\torf \Pi} \to \mgse{\Lambda_c}$.
For $W$ simply-laced,  
the map is induced by the algebra quotient $\Pi \twoheadrightarrow \Lambda_c$. 
\end{corollary}
\begin{proof}
The first part follows from combining \Cref{thm:coxeter}\ref{op:coxeter:contraction} with \Cref{prop:mgs_poset_iso}. The second part then follows by \Cref{thm:tors=weak}\ref{op:tors=weak:commutativity}.
\end{proof}

We now explain that our result about $\mgtmap{q_c}$ being a contraction and $\mgse{\Lambda_c}$ having a unique minimum can be used to prove \Cref{thm:tors=weak}\ref{op:tors=weak:commutativity} and the second part of \Cref{cor:mgs_contraction} even in non-simply-laced types.

\begin{proposition}
\label{prop:not_simply_laced}
Let $W$ be finite and crystallographic, but not necessarily simply-laced.
\Cref{thm:tors=weak}\ref{op:tors=weak:commutativity} still holds, and the contraction of posets $\mgtmap{q_c}\colon \mgea{\lambda_c}{\torf \Pi} \to \mgse{\Lambda_c}$ is induced by the algebra quotient.
\end{proposition}

\begin{proof}
We have a natural algebra map 
$\Pi \twoheadrightarrow \Lambda_c$ in all  Dynkin types including non-simply-laced, as $\Pi$ can be described as $T_A(M \oplus DM)/\langle 1_{{}_A M} - 1_{M_A}\rangle$ 
whenever $\Lambda_c$ is a finite-dimensional hereditary tensor algebra $T_A(M)$, with $A$ a basic finite-dimensional semisimple algebra over a field and $M$ is a finite-dimensional $A$-bimodule on which this field acts centrally,\footnote{We can also interpret $\Pi$ as $T_{\Lambda_c}(\tau^{-} \Lambda_c)$ in this generality, but the description we use makes it clearer that $\Lambda_c \cong \Pi/I$, where $I$ is  the ideal generated by $DM$.} which is the case in particular in Dynkin types.
By \cite[Theorem~5.12(a)]{dirrt}, the quotient $\Pi \twoheadrightarrow \Lambda_c$ induces a lattice quotient $\torf \Pi \twoheadrightarrow \torf \Lambda_c$ given by $\mathcal{F} \mapsto \mathcal{F} \cap \modules \Lambda_c$.
We denote the corresponding lattice congruence by $\theta_{\mathsf{alg}}$.
By \cite[Theorem~5.15(a)]{dirrt}, $\theta_{\mathsf{alg}}$ is precisely the lattice congruence which contracts covering relations labelled by bricks which are not $\Lambda_c$-modules.
We want to show that $\theta_{\mathsf{alg}}$ coincides with the image of $\theta_c$ under the isomorphism $W \overset{\sim}{\to} \torf \Pi$. 
 
We first show that every edge contracted by $\theta_{\mathsf{alg}}$ is contracted by $\theta_c$. 
We recall \Cref{thm:cambrian_dim2_contraction} and Corollary~\ref{cor:minima_and_maxima} and let $\mathbf{w}_0$ be a reduced expression of $w_0$.
Together these results imply that there exists a sequence of maximal chains $(C_1, C_2, \dots, C_k)$ of $W \cong \torf \Pi$ with the following properties.
\begin{enumerate}
    \item The chain $C_k$ corresponds to the reduced expression $\mathbf{w}_0$.
    \item The chain $C_1$ corresponds to the reduced expression $\mathbf{w}_0(c)$.
    \item For all $i$, $C_i$ and $C_{i + 1}$ are related by a polygon move.
    \item For all $i$, one of the following holds.
        \begin{enumerate}
            \item $q_c(C_i)$ and $q_c(C_{i + 1})$ are equivalent. In particular, we may have $q_c(C_i) = q_c(C_{i + 1})$.
            \item $q_c(C_{i + 1})$ is an increasing polygon move of $q_c(C_i)$.
        \end{enumerate}
\end{enumerate}
We know from Corollary~\ref{cor:minima_and_maxima} that we can find a sequence of increasing polygon moves in $W_c$ from $q_c(C_1)$ to $q_c(C_k)$, and then we use the fact from \Cref{thm:cambrian_dim2_contraction} that $\mgtmap{q_c}$ is a contraction of posets to lift this sequence of polygon moves in $W_c$ to the above sequence of maximal chains $(C_1, C_2, \dots, C_k)$ with the desired properties.
(Note that we do not ask all non-square polygon moves in this sequence to be increasing in $\mgea{\lambda^{\mathbf{w}_0(c)}}{W}$: at the moment we do not know if this is possible, as while we prove that $\mgtmap{q_c}$ is a contraction, it is not known whether it is a weak congruence, see Remark~\ref{rem:weak_order_congruence} below.)

We prove inductively on the sequence $(C_1, C_2, \dots, C_k)$ that every edge of $C_k$ contracted by $\theta_{\mathsf{alg}}$ is contracted by $\theta_c$.
For the base case, we have by \Cref{lem:min_mgs} that none of the edges of $C_1$ is contracted by $\theta_{\mathsf{alg}}$, and we also have by construction that no edges of $C_1$ are contracted by~$\theta_c$. 
For the inductive step, we need to show the following claim: if every edge of $C_i$ contracted by $\theta_{\mathsf{alg}}$ is contracted by $\theta_c$, then every edge of $C_{i + 1}$ contracted by $\theta_{\mathsf{alg}}$ is contracted by $\theta_c$. 
Let $P_i$ be the polygon relating $C_i$ and $C_{i+1}$, and denote $\check{p}_i := \min P_i$ and $\hat{p}_i := \max P_i$.
Let $a$ and $b$ be the respective atom and coatom of $P_i$ which are in $C_{i}$, and similarly let $a'$ and $b'$ be the respective atom and coatom of $P_i$ which are in $C_{i + 1}$.

Suppose for contradiction that there is an edge $x \lessdot y$ of $C_{i + 1}$ that is contracted by $\theta_{\mathsf{alg}}$ and not by $\theta_c$.
First note that $x \lessdot y$ must be an edge of $C_{i + 1} \cap P_i$, otherwise it is also an edge of $C_i$, which contradicts the induction hypothesis.
Then note that we cannot have $(x, y) = (\check{p}_i, a')$, since this edge is forcing-equivalent to $(b, \hat{p}_i)$.
Indeed, if $\check{p}_i \lessdot a'$ were contracted by $\theta_{\mathsf{alg}}$ and not $\theta_c$, then we would have to have that $b \lessdot \hat{p}_i$ was contracted by $\theta_{\mathsf{alg}}$ and not by $\theta_c$, which contradicts the induction hypothesis.
Similarly, we cannot have that $(x, y) = (b', \hat{p}_i)$.

Now note that by our assumption on the sequence $(C_1, C_2, \dots, C_k)$ and \Cref{lem:length_two_max_chain}, we must have that $q_c(C_{i + 1} \cap P_i)$ is a chain of length at most two.
Hence, every edge of $C_{i + 1} \cap P_i$ is contracted by $\theta_{c}$ apart from possibly $\check{p}_i \lessdot a'$ and $b \lessdot \hat{p}_i$.
This leaves no further possibility for $x \lessdot y$, and so we conclude that every edge of $C_{i + 1}$ which is contracted by $\theta_{\mathsf{alg}}$ is contracted by~$\theta_c$.

This establishes the inductive step, and so we conclude that every edge in every maximal chain in $W \cong \torf \Pi$ contracted by $\theta_{\mathsf{alg}}$ is contracted by $\theta_c$.
Since every covering relation in $W$ belongs to some maximal chain, 
every edge in $W \cong \torf \Pi$ contracted by $\theta_{\mathsf{alg}}$ is contracted by $\theta_c$.
So $W_c$ is isomorphic to a quotient of $\torf \Lambda_c$, and this quotient is $\torf \Lambda_c$ itself if and only if $\theta_c$ does not contract any edges not contracted by $\theta_{\mathsf{alg}}$. But this is exactly the case by \Cref{thm:tors=weak}\ref{op:tors=weak:lambda_bij}.
This completes the proof of 
\Cref{thm:tors=weak}\ref{op:tors=weak:commutativity}.
The second part of the statement follows directly as in 
the proof of \Cref{cor:mgs_contraction}.
\end{proof}

\begin{corollary}
\label{cor:ascending_none_contracted}
Let $P$ be a non-square polygon in $W$ such that both its edges from the minimum to atoms are not contracted by $\theta_c$. Then none of the edges in the ascending chain in $P$ with respect to $\lambda^{\mathbf{w}_0(c)}$ is contracted by $\theta_c$.
\end{corollary}

\begin{proof}
By \Cref{prop:not_simply_laced}, we can equivalently prove the statement about $\theta_{\mathsf{alg}}$.
Let $M, N$ be the bricks labelling the edges from the minimum to atoms, which therefore belong to $\mod \Lambda_c$.
We also have $\Hom(M, N) = \Hom(N, M) =0$ both over $\Pi$ and over $\Lambda_c$, and $\Ext^1_{\Lambda_c}(M, N) = 0$ or $\Ext^1_{\Lambda_c}(N, M) = 0$.
We know that $\sbf{\dimu M}{\dimu N}$ is the symmetrisation of the Euler form and has a description in terms of $\Hom$ and $\Ext^1_{\Pi}$ as in the proof of \Cref{prop:preproj_bricks}.
From this it follows that if, say, $\Ext^1_{\Lambda_c}(N, M) = 0$, then $\dim \Ext^1_{\Pi}(M, N) = \dim \Ext^1_{\Lambda_c}(M, N)$, but since one is
isomorphic to
a subspace of the other,  
we have $ \Ext^1_{\Pi}(M, N) \cong \Ext^1_{\Lambda_c}(M, N)$.
Then the polygon in $\torf \Lambda_c$ corresponds to the lattice of torsion-free classes of the hereditary algebra with two isomorphism classes of simples with extensions given by $\Ext^1_{\Lambda_c}(M, N) \cong \Ext^1_{\Pi}(M, N)$ by \cite[Theorem~3.12 and Corollary~3.19]{jasso_red}.
This is the hereditary algebra of the same type as the rank-two subsystem formed by $\dimu M, \dimu N$. 
Hence, the polygon is isomorphic to the cluster complex for the same rank-two Dynkin type, and we know that none of the edges there is contracted
(see e.g., \cite[Proposition~22 and its proof]{hubery} and references in \emph{loc.cit.}).
Further, the other side has precisely two non-contracted edges, as $\Ext^1_{\Lambda_c}(N, M) = 0$; this together with \Cref{lem:length_two_max_chain} and its proof (or alternatively \Cref{lem:rank_2_heap}) immediately imply that the side with contracted edges is the descending one with respect to $\lambda^{\mathbf{w}_0(c)}$, and so the side with no edges being contracted is the ascending one.
\end{proof}

\begin{remark}
Hothem \cite{hothem,hothem_thesis} defines higher Bruhat orders for words in $S_n$ besides $w_0$.
Using the algebraic interpretation of the weak Bruhat order on $S_n$ in terms of $\torf \Pi$ of Dynkin type $A$, this can be interpreted as higher Bruhat orders for maximal chains inside a particular torsion-free class. 
\end{remark}

\begin{remark} \label{rem:extrema}
In our previous work \cite{gw1} we showed that the poset $\mgse{\Lambda_c}$ has a natural maximum and a unique minimum, see \Cref{cor:minima_and_maxima}.
We cannot deduce the same property for $\mgea{\lambda_c}{\Pi}$ from \Cref{cor:mgs_contraction} since the fibres of contractions $\mgtmap{q_c}$ are not always intervals.
\end{remark}

\section{Relation to the higher Bruhat orders and higher Stasheff--Tamari orders}\label{sect:hbo_hst}

In this section, we briefly explain the relation between our general setting in this paper and the map from the higher Bruhat orders to the higher Stasheff--Tamari orders first considered by Kapranov and Voevodsky \cite[Theorem~4.4]{kv-poly}.
This corresponds to the case where the Coxeter diagram of $(W, S)$ is the $A_n$ Dynkin diagram with vertices labelled $1$ to $n$ in order (so that $W$ is the symmetric group $\mathfrak{S}_{n+1}$), and $c = s_1s_2 \dots s_n$.
We fix such $(W, S)$ and $c$, with $\Lambda_c$ the corresponding hereditary path algebra, throughout this section, unless specified otherwise.
We refer to \cite[Definition~2.2]{ms} for the definition of the higher Bruhat orders $\bruhat{n}{d}$.

\begin{proposition}
$\mgea{\lambda^{\mathbf{w}_0(c)}}{W} \cong \bruhat{n}{2}$.
\end{proposition}
\begin{proof}
It is well-known and straightforward to show that $\bruhat{n}{1}$ is isomorphic to the weak Bruhat order on~$W$.
Elements of $\bruhat{n}{2}$ are square-equivalence classes of maximal chains in $\bruhat{n}{1}$.
Hence, we have that $\mge{W}$ is in bijection with $\bruhat{n}{2}$.
Let $\alpha_i$ be the simple root corresponding to $s_i$.
It is well-known that for our choice of $c$, $\mathbf{w}_0(c)$ orders $\Psi^{+}$ in the ``lexicographic order'', meaning that $\alpha_i + \dots \alpha_j < \alpha_{i'} + \dots \alpha_{j'}$ if and only if $i < i'$ or $i = i'$ and $j < j'$.
This implies that the order on $\bruhat{n}{2}$ is the same as the order on $\mgea{\lambda^{\mathbf{w}_0(c)}}{W}$, as desired.
\end{proof}

\begin{remark}\label{rmk:c_sort_hbo}
The partial orders $\mgea{\lambda^{\mathbf{w}_0}}{W}$ for $W = \mathfrak{S}_{n+1}$ are precisely the re-oriented higher Bruhat orders (RHBOs), as considered in \cite[Section~3]{felsner_weil}.
The partial orders $\mgea{\lambda^{\mathbf{w}_0(c)}}{W}$ for $\mathbf{w}_{0}(c)$ a $c$-sorting word for $\mathbf{w}_0$ for some~$c$ (possibly different from our standard choice $s_1 s_2 \ldots s_n$) form a particular subclass.
We note that in \textit{op.\ cit.}\ it is shown that in general RHBOs are quite poorly behaved: in particular, they may not have a unique minimum or maximum.
It is interesting to observe that the examples of this phenomenon given in \cite[Section~3]{felsner_weil} do not come from $c$-sorting words.
This can be seen using the following reasoning.
By \cite[Theorem~9.3.1]{stump2015cataland}, 
the Hasse diagram of the heap poset $\heap{\mathbf{w_0}(c)}$ is the Auslander--Reiten quiver of $\Lambda_c$.
One can take the equivalence class of maximal chains used to define the pathological RHBOs in \cite[Section~3]{felsner_weil}, compute their heap poset, and observe that the resulting Hasse diagram cannot be the Auslander--Reiten quiver of $\Lambda_c$ for any $c$. 

We tentatively expect that the class of RHBOs produced by $c$-sorting words are reasonably well-behaved, compared to other RHBOs.
Several other authors have studied higher Bruhat orders in other Dynkin types \cite{wellman,sv,dkk_hbo}.
\end{remark}

\begin{remark}
The posets $\bruhat{n}{2}$ are the two-dimensional case of the higher Bruhat orders $\bruhat{n}{d}$.
The posets $\bruhat{n}{d}$ are defined inductively as having elements square-equivalence classes of maximal chains in $\bruhat{n}{d - 1}$.
The partial order on $\bruhat{n}{d}$ is then defined as the transitive closure of single-step inclusion of the inversion sets of the maximal chains.
It is also possible to consider the poset $\bruhat{n}{d}_{\subseteq}$ where the order relation is defined by inclusion of inversion sets.
It is known that $\bruhat{n}{d}$ and $\bruhat{n}{d}_{\subseteq}$ coincide for $d \leqslant 2$ \cite[Theorem~1]{felsner_weil} but otherwise differ in general \cite[Theorem~4.5]{ziegler-bruhat}.
Hence, based on this remark and the previous, it is natural to formulate the following problem.
Note that a positive answer to \ref{op:coxeter_inclusion:inclusion} here also gives a positive answer to~\ref{op:coexeter_inclusion_minmax}.
\end{remark}

\begin{problem} \label{pbm:coxeter_inclusion}
Let $(W, S)$ be a finite Coxeter system and $c$ a Coxeter element.
\begin{enumerate}[label=\textup{(}\roman*\textup{)}]
\item Does the order $\mgea{\lambda^{\mathbf{w}_0(c)}}{W}$ have a unique minimum and maximum?\label{op:coexeter_inclusion_minmax}
\item Does the order $\mgea{\lambda^{\mathbf{w}_0(c)}}{W}$ coincide with the partial order on the same underlying set given by inclusion of inversion sets?\label{op:coxeter_inclusion:inclusion}
\end{enumerate}
\end{problem}

For the poset $\mgea{\lambda^{\mathbf{w}_0(c)}_{\theta_c}}{W_c}$, we instead obtain the higher Stasheff--Tamari orders $\stash{n + 2}{3}$, for whose definition we refer to \cite[p.132]{er} or \cite[Definition~2.13]{njw-equal}.

\begin{proposition}
$\mgea{\lambda^{\mathbf{w}_0(c)}_{\theta_c}}{W_c} \cong \stash{n + 2}{3}$.
\end{proposition}
\begin{proof}
This follows from \Cref{prop:mgs_poset_iso}, which gives us that $\mgea{\lambda^{\mathbf{w}_0(c)}_{\theta_c}}{W_c} \cong \mgse{\Lambda_c}$, and \cite[Theorem~5.23(1)]{njw-hst} showing that $\mgse{\Lambda_c} \cong \stash{n + 2}{3}$.
Of course, one could also show this proposition using more elementary means.
\end{proof}

We now compare our map to the map $f \colon \bruhat{n}{2} \to \stash{n + 2}{3}$ from \cite[Theorem~4.4]{kv-poly}, for whose definition we refer to \cite[Section~4]{thomas-bst}.

\begin{proposition}\label{prop:we_have_f}
There is a commutative diagram \[
\begin{tikzcd}
    \mgea{\lambda^{\mathbf{w}_0(c)}}{W} \ar[r,"{\mgtmap{q_c}}"] \ar[d,"\wr"] & \mgea{\lambda^{\mathbf{w}_0(c)}_{\theta_c}}{W_c} \ar[d,"\wr"] \\
    \bruhat{n}{2} \ar[r,"f"] & \stash{n + 2}{3}
\end{tikzcd}
\]
\end{proposition}
\begin{proof}
This follows from the fact that before taking equivalence classes of maximal chains, we have the commutative diagram \[
\begin{tikzcd}
W \ar[r,"q_c"] \ar[d,"\wr"] & W_c \ar[d,"\wr"] \\
\bruhat{n}{1} \ar[r,"g"] & \stash{n + 2}{2}.
\end{tikzcd}
\]
Note that $\stash{n + 2}{2}$ is the Tamari lattice, so the right-hand vertical isomorphism is well-known here; see, for instance, \cite[Introduction, Section~6]{reading_cambrian}.
The map $g$ is the usual map from the higher Bruhat order to the Tamari lattice, which is known to coincide with $q_c$ by \cite[Section~5]{thomas-bst}.
Then the commutative diagram in the statement of the proposition results from applying this commutative diagram to the individual elements of maximal chains as in \Cref{lem:max_chains} and Corollary~\ref{prop:max_chains_eq}.
In other words, $f = \mgtmap{g}$, which is shown in \cite[Section~4]{thomas-bst}.
\end{proof}

\begin{corollary}\label{cor:map_f_contraction}
The map $f \colon \bruhat{n}{2} \to \stash{n + 2}{3}$ is a contraction of posets.
\end{corollary}
\begin{proof}
This follows from \Cref{thm:coxeter}\ref{op:coxeter:contraction} and \Cref{prop:we_have_f}.
\end{proof}

The surjectivity of the map $f \colon \bruhat{n}{2} \to \stash{n + 2}{3}$ was proven in \cite[Proposition~6.1]{thomas-bst}, while the fact that it satisfies the other defining properties of contractions is genuinely new.
It is known that the fibres of the map $f \colon \bruhat{n}{2} \to \stash{n + 2}{3}$ do not always have unique maxima \cite[Section~6]{thomas-bst} and so cannot always be intervals.
Hence, the map is quite far from being an order quotient.
However, Corollary~\ref{cor:map_f_contraction} shows that it is still a reasonably nice quotient of posets, in particular, it has connected fibres. 

\begin{remark} \label{rem:weak_order_congruence}
There exists another notion of quotient maps of posets not mentioned in the background \Cref{sect:back:poset_quotients} but relevant in the context of the map $f$. Namely, an order-preserving map of posets $f\colon P \to Q$ is \emph{strong} if for any  relation $y < y'$ in $Q$, there exists a relation $x < x'$ of $P$ such that $f(x) = y$ and $f(x') = y'$, and a congruence on a poset is a \emph{weak order congruence} if and only if it is the kernel of a strong map, see \cite[Section~4.3]{njw-survey} and references therein. Just as for contractions, for a finite poset every order congruence is a weak order congruence, but the converse is not true. Roughly speaking, quotients by weak order congruences are surjective on relations, while contractions are surjective on covering relations and have connected fibres; neither of these two classes of quotients contains the other.

Just as with the higher Bruhat orders, the higher Stasheff--Tamari orders $\stash{n + 2}{3}$ likewise exist in higher dimensions and are denoted $\stash{n + 2}{3}$.
The map $f$ then extends to a map $f \colon \bruhat{n}{d} \to \stash{n + 2}{3}$.
A difficult open problem is the surjectivity of the map $f$ in higher dimensions. 
It is natural to ask other nice properties of~$f$; the above mentioned result of Thomas \cite{thomas-bst} implies that it cannot be an order quotient for $d \geqslant 2$.
An open conjecture is that $f$ is a quotient by a weak order congruence \cite[Theorem~4.10]{kv-poly}, \cite[Conjecture~8.9]{njw-survey}.
Following Corollary~\ref{cor:map_f_contraction}, it is natural to make the following conjecture.
\end{remark}

\begin{conjecture}
The map $f \colon \bruhat{n}{d} \to \stash{n + 2}{d+1}$ is a contraction of posets for all $n, d \in \mathbb{N}$.
\end{conjecture}

We also raise the following problem, which is the analogue of \Cref{thm:cambrian_dim2_contraction} for weak order congruences.

\begin{problem}
Let $W$ be a finite Coxeter group, with $c$ a Coxeter element and $q_c \colon W \to W_c$ the Cambrian quotient.
Is the induced map \[
\mgtmap{q_c} \colon \mgea{\lambda^{\mathbf{w}_0(c)}}{W} \to \mgea{\lambda^{\mathbf{w}_0(c)}_{\theta_c}}{W_c}
\]
a quotient by a weak order congruence?
\end{problem}

We now discuss another setting where the two-dimensional higher Bruhat orders and their possible analogues in other Coxeter types are relevant.

\begin{remark} \label{rem:elias}
In the context of Soergel calculus and its ``thicker'' versions, partial orders generalising two-dimensional higher Bruhat orders have been studied \cite{elias_bruhat, hothem, hothem_thesis, Koley, sv} from the following perspective. 
Given a Coxeter group $W$ (which does not have to be of type $A$ in this remark), one associates a Bott--Samelson bimodule $B_{\mathbf{w}}$ to every reduced expression $\mathbf{w}$.
The commutation relations between different expressions give natural isomorphisms of such bimodules, which are compatible in the sense that one can associate a single bimodule to a commutation class, i.e., in case of reduced expressions of $w_0$, to an element of $\mge{W}$. A general braid move corresponds to a morphism  projecting to a common summand. The set of indecomposable Soergel bimodules is parameterised, up to isomorphism and grading shift, by elements of $W$. In particular, there is a canonical Soergel bimodule $B_{w_0}$; for every reduced expression $\mathbf{w}_0$, it appears as an indecomposable direct summand of $B_{\mathbf{w}_0}$ with multiplicity one, and it does not appear as a summand of $B_{\mathbf{w}}$ for reduced expressions $\mathbf{w}$ of shorter elements $w \in W$ (see \cite{Soergel}).
As long as $\mge{W}$ has a partial order satisfying certain conditions including having a unique source and a unique sink, one can recover $B_{w_0}$ 
as the image of the projection induced by a sequence of morphisms going from the source to the sink and back along maximal chains in such an order. The full classification of such nice partial orders is not known. One does not expect to have such an order in type $D_4$; the orders in type $A$ are studied in \cite{elias_bruhat}
and in type $B$ in \cite{sv, Koley}. We note that each such order in these references does have the form  $\mgea{\lambda^{\mathbf{w}_0(c)}}{W}$ for some choice of $c$.
We hope that our systematic approach to partial orders on $\mge{W}$ might be helpful for a better understanding of orders relevant for Soergel calculus.
In particular, it is natural to expect that every relevant order must indeed have the form  $\mgea{\lambda^{\mathbf{w}_0(c)}}{W}$ for some choice of $c$.
While we do not know if every $\mgea{\lambda^{\mathbf{w}_0(c)}}{W}$ has a unique sink and a unique source, we wonder whether the contraction $\mgtmap{q_c}$, its fibres, or the polygon moves may have meaningful interpretations in this context (see also Remark \ref{rem:extrema}).
We thank Ben Elias for comments related to this remark.
\end{remark}

\printbibliography

\end{document}